\documentclass[11pt,a4paper]{article}

\usepackage{latexsym,amsfonts,amsmath,amsthm,graphics}


\usepackage{graphicx}

\usepackage[latin1]{inputenc}
\usepackage[T1]{fontenc}
\usepackage{microtype} 

\usepackage{cite}
\usepackage{array,colortbl}
\usepackage{amsmath,amsfonts,amssymb,bm,amsthm,mathtools,mathrsfs}
\DeclareFontFamily{U}{mathx}{\hyphenchar\font45}
\DeclareFontShape{U}{mathx}{m}{n}{
	<5> <6> <7> <8> <9> <10>
	<10.95> <12> <14.4> <17.28> <20.74> <24.88>
	mathx10
}{}
\DeclareSymbolFont{mathx}{U}{mathx}{m}{n}
\DeclareFontSubstitution{U}{mathx}{m}{n}
\DeclareMathAccent{\widecheck}{0}{mathx}{"71}

\def\citep#1#2{\cite[{#1}]{#2}}

\newcommand{\RefThm}[1]{Theorem~\textup{\ref{#1}}}

\newcommand{\RefAlg}[1]{Algorithm~\textup{\ref{#1}}}

\usepackage{algorithm}
\usepackage{algorithmic}

\newcommand{\bbN}{{\mathbb{N}}}

\newcommand{\bbZ}{{\mathbb{Z}}}
\newcommand{\N}{{\mathbb{N}}} 
\newcommand{\R}{{\mathbb{R}}} 
\newcommand{\Z}{{\mathbb{Z}}} 
\newcommand{\RR}{{\mathbb{R}}} 
\newcommand{\ZZ}{{\mathbb{Z}}} 

\DeclareSymbolFont{bbold}{U}{bbold}{m}{n}
\DeclareSymbolFontAlphabet{\mathbbold}{bbold}

\newcommand{\bsm}{{\boldsymbol{m}}}

\newcommand{\bsp}{{\boldsymbol{p}}}

\newcommand{\bsx}{{\boldsymbol{x}}}

\newcommand{\bsz}{{\boldsymbol{z}}}

\newcommand{\bszero}{{\boldsymbol{0}}} 

\newcommand{\bsgamma}{{\boldsymbol{\gamma}}}

\newcommand{\bstau}{{\boldsymbol{\tau}}}

\newcommand{\calF}{{\mathcal{F}}}

\newcommand{\calO}{{\mathcal{O}}}

\newcommand{\setu}{{\mathfrak{u}}}
\newcommand{\setv}{{\mathfrak{v}}}

\newcommand{\rme}{{\mathrm{e}}}

\newcommand{\icomp}{\mathrm{i}}
\newcommand{\abs}[1]{\left\vert#1\right\vert}
\newcommand{\norm}[1]{\left\Vert#1\right\Vert}
\newcommand{\rd}{\,\mathrm{d}} 

\providecommand{\argmin}{\operatorname*{argmin}}

\newcommand{\mcol}{{\mathpunct{:}}}

\usepackage[hidelinks,hypertexnames=false,hyperindex=true,pdfpagelabels,linktoc=all]{hyperref}
\usepackage{bookmark}
\makeatletter 
\providecommand*{\toclevel@author}{999}
\providecommand*{\toclevel@title}{0}
\makeatother

\usepackage{enumitem}

\theoremstyle{plain}
\newtheorem{theorem}{Theorem}
\newtheorem{proposition}{Proposition}
\newtheorem{lemma}{Lemma}
\newtheorem{corollary}{Corollary}

\theoremstyle{definition}
\newtheorem{definition}{Definition}

\theoremstyle{remark}
\newtheorem{remark}{Remark}

\usepackage[T1]{fontenc}
\usepackage{lmodern}
\usepackage[a4paper]{geometry}
\usepackage{inputenc}
\usepackage{amsmath,amsfonts,amssymb}
\usepackage{bm}
\usepackage{mathtools}
\usepackage{hyperref}
\usepackage{tikz}
\usepackage{caption}
\usepackage{listings}
\usepackage{float}
\usepackage{pgfplots}
\usepackage{subcaption}
\usepackage{enumitem}
\usepackage{xcolor}

\usepackage[normalem]{ulem} 

\usepackage{booktabs}
\usepackage{multirow}

\allowdisplaybreaks

\usepackage{changebar}

\pgfplotsset{every tick label/.append style={font=\scriptsize}}
\newenvironment{customlegend}[1][]{
	\begingroup
	\csname pgfplots@init@cleared@structures\endcsname
	\pgfplotsset{#1}
}{
	\csname pgfplots@createlegend\endcsname
	\endgroup
}

\def\addlegendimage{\csname pgfplots@addlegendimage\endcsname}
\pgfplotsset{
	cycle list={%
		{draw=black,mark=star,solid},
		{draw=black, mark=square,solid},
		{draw=black,mark=+,solid},
		{black,mark=o},
		{draw=black, mark=none,solid}}
}

\definecolor{mycolor1}{rgb}{0.48500,0.70000,1.00000}
\definecolor{mycolor2}{rgb}{0.00000,0.20000,0.50000}
\definecolor{mycolor-alpha1}{rgb}{0.30000,0.45000,0.85000}
\definecolor{mycolor-alpha2}{rgb}{0.05000,0.65000,0.20000}
\definecolor{mycolor-alpha3}{rgb}{0.64000,0.22500,0.75000}

\definecolor{mycolor3}{rgb}{0.10000,0.80000,0.15000}
\definecolor{mycolor4}{rgb}{0.00000,0.50000,0.25000}
\definecolor{mycolor5}{rgb}{0.84000,0.29000,1.00000}
\definecolor{mycolor6}{rgb}{0.44000,0.16500,0.50000}

\definecolor{mycolor1-fig}{rgb}{0.48500,0.70000,1.00000}%
\definecolor{mycolor2-fig}{rgb}{0.00000,0.20000,0.50000}%
\definecolor{mycolor3-fig}{rgb}{0.85000,0.11800,0.09800}%
\definecolor{mycolor4-fig}{rgb}{0.10000,0.80000,0.15000}%
\definecolor{mycolor5-fig}{rgb}{0.00000,0.50000,0.25000}%
\definecolor{mycolor6-fig}{rgb}{0.84000,0.29000,1.00000}%
\definecolor{mycolor7-fig}{rgb}{0.44000,0.16500,0.50000}%

\newcommand{\tpmod}[1]{{\;(\operatorname{mod}\;#1)}}
\newcommand{\supp}{\operatorname{supp}}

\begin{document}
	
\title{Digit-by-digit and component-by-component constructions of lattice rules for periodic functions with unknown smoothness}

\author{Adrian Ebert, Peter Kritzer, Dirk Nuyens, Onyekachi Osisiogu}

\date{\today}

\maketitle

\begin{abstract}
    Lattice rules are among the most prominently studied quasi-Monte Carlo methods to approximate multivariate integrals.
    A rank-$1$ lattice rule to approximate an $s$-dimensional integral is fully specified by its \emph{generating vector} $\bsz \in \bbZ^s$ and its number of points~$N$.
    While there are many results on the existence of ``good'' rank-$1$ lattice rules, there are no explicit constructions for good generating vectors for dimensions $s \ge 3$. 
    This is why one usually resorts to computer search algorithms. Motivated by earlier work of Korobov from 1963 and 1982, more specifically \cite{Kor63} and 
    \cite{Kor82}, we present two variants of search algorithms for good lattice rules and show that the resulting rules exhibit 
    a convergence rate in weighted function spaces that can be arbitrarily close to the optimal rate. Moreover, contrary to most other algorithms, we do not need to 
    know the smoothness of our integrands in advance, the generating vector will still recover the convergence rate associated with the smoothness of 
    the particular integrand, and, under appropriate conditions on the weights, the error bounds can be stated without dependence on $s$.
    The search algorithms presented in this paper are two variants of the well-known \emph{component-by-component (CBC) construction}, one of which is combined with a \emph{digit-by-digit (DBD) construction}. We present numerical results for both algorithms using fast construction algorithms in the case of product weights. 
    They confirm our theoretical findings.
\end{abstract}



\section{Introduction} \label{sec:intro}

The present paper is studying the efficient construction of high-dimensional \emph{quadrature rules} 
(also referred to as \emph{cubature rules} when $s \ge 3$) for numerically approximating $s$-dimensional integrals
\begin{equation*}
	I(f)
	:=
	\int_{[0,1]^s} f(\bsx) \rd\bsx
	\quad\approx\quad
	Q_N(f, \{(w_k,\bsx_k)\}_{k=0}^{N-1})
	:=
	\sum_{k=0}^{N-1} w_k \, f(\bsx_k)
	,
\end{equation*}
where $\bsx_0,\ldots,\bsx_{N-1}\in [0,1]^s$ and $w_0,\ldots,w_{N-1}\in\R$, and where we assume that the integrand $f$ lies in a Banach space $(\calF,\norm{\cdot}_{\calF})$.
In this paper we assume the quadrature nodes and weights $\{(w_k,\bsx_k)\}_{k=0}^{N-1}$ to be chosen deterministically. 
The quality of $Q_N$, or equivalently, of the chosen quadrature nodes and weights, can be assessed by the \emph{worst-case error}
\[
	e_{N,s}(Q_N)
	:=
	\sup_{\substack{f \in \calF \\ \|f\|_{\calF} \le 1}} \left| I(f) - Q_N(f, \{(w_k,\bsx_k)\}_{k=0}^{N-1}) \right|.
\]
We are interested in the behavior of the worst-case error in terms of $N$ and~$s$.
When the dimensionality $s$ is high it is convenient to consider equal-weight rules with $w_k \equiv 1/N$. We call such rules \emph{quasi-Monte Carlo rules}.
In general, it is highly non-trivial to choose the set of quadrature nodes such that the resulting rule has a low worst-case error, 
and it is usually necessary to tailor the choice of the quadrature nodes to the function space $\calF$ under consideration.
In this paper, we will consider Banach spaces which are based on assuming sufficient decay of the Fourier coefficients of its elements to guarantee certain smoothness properties. 
These spaces will be denoted by $E_{s,\bsgamma}^{\alpha}$, where $s$ denotes the number of variables the functions depend on, $\alpha>1$ 
is a real number frequently referred to as the smoothness parameter, and $\bsgamma = \{\gamma_\setu\}_{\setu \subset \bbN}$ is a sequence of strictly positive weights to 
model the importance of different subsets of dimensions. Intuitively, a large $\gamma_{\setu}$ corresponds to a high influence of the variables $x_j$ with $j\in\setu$, 
while a small $\gamma_{\setu}$ means low influence. This will be made more precise by incorporating the weights in the norm of the space $E_{s,\bsgamma}^{\alpha}$ in Section~\ref{sec:space}. The idea of these weights goes back to Sloan and Wo\'{z}niakowski \cite{SW98}, see also \cite{SW01,Hic98}.
We are interested in conditions on these weights such that we can bound the worst-case error independently of~$s$ with the nearly optimal rate in~$N$.
This is called \emph{strong tractability}, see, e.g., \cite{NW08}, for a general reference.

It is known from the classical literature on quasi-Monte Carlo methods (see standard textbooks such as \cite{HW81, N92b, SJ94}, 
and for more recent overviews \cite{DKS13} and also \cite{N14}) that an excellent choice of quadrature rules for approximating integrals of functions in 
$E_{s,\bsgamma}^{\alpha}$ are lattice rules, which were introduced by Korobov \cite{Kor59} and Hlawka \cite{H62}.
We consider \emph{rank-$1$ lattice rules}
\[
	Q_N(f,\bsz)
	:=
	\frac1N \sum_{k=0}^{N-1} f\left(\left\{\frac{k \bsz}{N}\right\}\right)
	,
\]
which are equal-weight quadrature/cubature rules with the quadrature nodes given by
\[
  \bsx_k
  :=
  \left( \left\{ \frac{k z_1}{N}\right\} , \ldots , \left\{ \frac{k z_s}{N}\right\} \right)
  \in
  [0, 1)^s
  ,
  \qquad
  \text{for } k = 0, 1, \ldots, N-1,
\]
and where $\{x\}=x-\lfloor x \rfloor$ denotes the fractional part of $x$.
Note that, given $N$ and $s$, the lattice rule is completely determined by the choice of the \emph{generating vector} $\bsz=(z_1,\ldots,z_s) \in \bbZ_N^s$, where $\bbZ_N := \{0, \ldots, N-1\}$ is the least residue system modulo~$N$. We remark that it is sufficient to consider the choice of $z_j$ to be modulo~$N$ since $\{k z_j/N\} = (k z_j \bmod{N})/N$ for integer $k$, $N$ and~$z_j$. However, it should be obvious that not every choice of a generating vector $\bsz$ also yields a lattice rule with good quality for approximating the integral. For dimensions $s\le 2$, explicit constructions of good generating vectors are available, see, e.g., \cite{SJ94,N92b}, but there are no explicit constructions of good generating vectors known for $s > 2$.

A complete search for a good generating vector $\bsz \in \bbZ_N^s$ would be infeasible even for moderate values of $N$ or $s$ due to the size of $\calO(N^s)$ of the search space. Therefore, Korobov \cite{Kor63}, and later Sloan and his collaborators \cite{SR02,SKJ02}, introduced a \emph{component-by-component (CBC) construction}, which is a greedy algorithm constructing the vector $\bsz \in \bbZ_N^s$ by successively increasing the dimension, choosing one $z_j$ at a time, 
and keeping previous components fixed, thus reducing the size of the search space to be $\calO(s N)$. It was shown in \cite{K03} for prime~$N$ and in \cite{D04} for non-prime~$N$ that the CBC construction yields generating vectors with essentially optimal convergence rates for a Hilbert space variant of the function space $E_{s,\bsgamma}^{\alpha}$ which we consider here (see also Remark~\ref{rem:other-spaces} below). Since, by judicious choice, the worst-case error expressions in both spaces 
take a similar form (squared in case of the Hilbert space setting), see Remark~\ref{rem:other-spaces}, the same applies here.
Furthermore, a fast component-by-component construction was introduced in \cite{NC06b,NC06}, for spaces with product weights $\gamma_\setu = \prod_{j\in\setu} \gamma_j$, 
reducing the computational cost of these algorithms to be only $\mathcal{O}(s N \ln N)$. See also \cite{N14} for other choices of weights, in both Hilbert and non-Hilbert space settings. For further refinements of the CBC construction, see also \cite{ELN18} and \cite{DKLP15, EKN20}.

In this paper, we study two construction algorithms for generating vectors of good lattice rules that have not been studied in a modern QMC setting before. Both are inspired by articles of Korobov, see \cite{Kor63} and \cite{Kor82} (English translation in \cite{Kor82Eng}). On the one hand, we will consider an algorithm that constructs the generating vector $\bsz$ in a component-by-component (CBC) fashion in which each component $z_j$ is assembled digit-by-digit (DBD), that is, 
for a number $N=2^n$ of points  we greedily construct the components $z_j$ bit-by-bit starting from the least significant bit. We call this the component-by-component digit-by-digit (CBC-DBD) algorithm.
We remark that also the paper \cite{NP09} studies a digit-by-digit construction of lattice rule generating vectors but in the slightly different context of lattice rules that are extensible in the number of points~$N$. The algorithm in \cite{NP09} is not a component-by-component algorithm.
Furthermore, the lattice rules in \cite{NP09} are, due to technical reasons, shown to yield an error convergence rate that is not close to optimal, whereas we will show here that the rules constructed by our CBC-DBD algorithm yield a convergence rate that is arbitrarily close to the optimal rate.
Moreover, the present paper also studies a variant of a CBC algorithm that, to our best knowledge, 
has also not been considered for weighted spaces so far. 
We stress that the error analysis for both algorithms discussed here is such that no prior knowledge of the smoothness parameter $\alpha$ is required to construct the generating vector.
The resulting generating vector will still deliver the near optimal rate of convergence, for arbitrary smoothness parameters $\alpha> 1$, and this result can be stated independently of the dimension, 
assuming that the weights satisfy certain conditions that depend on the smoothness. The standard CBC algorithms construct the generating vector specifically with the smoothness $\alpha$ as an input parameter. We see the independence of $\alpha$ in the considered construction algorithms as a big advantage.

\medskip

The rest of the paper is structured as follows. In Section~\ref{sec:space}, we give the precise definition of the function space under consideration, and outline how to analyze the error of lattice rules when integrating elements of the space. Section~\ref{sec:constr} is concerned with the two algorithms mentioned above. After showing some technical lemmas, we discuss the CBC-DBD construction in Section~\ref{sec:dbd} and provide an error analysis for the resulting quadrature rules. We then move on to the other new variant of the CBC construction in Section~\ref{sec:cbc} and prove similar error bounds. In Section~\ref{sec:fast_impl}, we give details on fast implementations of both new construction algorithms, and we present numerical experiments in Section~\ref{sec:num}. 

\medskip

We write $\bbN := \{1, 2, \ldots\}$ for the set of natural numbers and $\bbN_0 := \{0,1,2,\ldots\}$, $\bbZ$ for the set of integers and $\bbZ_N := \{0, \ldots, N-1\}$ for the least residues modulo~$N$. Additionally we define the set of non-zero integers by $\Z_\ast := \Z \setminus \{0\}$. To denote subsets of dimensions we use fraktur font, e.g., $\setu \subset \bbN$.
As a shorthand we write $\{1 \mcol s\} := \{1,\ldots,s\}$. To denote the projection of a vector $\bsx \in [0,1)^s$ or $\bsm \in \bbZ^s$ to the components in a subset of dimensions 
$\setu \subseteq \{1 \mcol s\}$ we write $\bsx_\setu := (x_j)_{j\in\setu}$ or $\bsm_\setu := (m_j)_{j\in\setu}$, respectively. The weights of the weighted function spaces will be denoted by $\gamma_\setu > 0$. To study the effect of increasing $s$ we consider the sequence of weights $\bsgamma = \{\gamma_\setu\}_{\setu\subset\bbN}$.
We set $\gamma_\emptyset = 1$ for proper normalization.

\section{Function space setting and rank-$1$ lattice rules}\label{sec:space}

We consider integrands which have an absolutely converging Fourier series,
\begin{align*}
	f(\bsx)
	=
	\sum_{\bsm \in \Z^s} \hat{f}(\bsm) \, \rme^{2 \pi \icomp \bsm \cdot \bsx}
	\qquad \text{with} \qquad
	\hat{f}(\bsm)
	:=
	\int_{[0,1]^s} f(\bsx) \, \rme^{-2 \pi \icomp \bsm \cdot \bsx} \rd \bsx ,
\end{align*}   
where $\hat{f}(\bsm)$ is the $\bsm$-th Fourier coefficient of $f$ and $\bsm \cdot \bsx := \sum_{j=1}^s m_j x_j$ is the vector dot product.
Since the Fourier series are absolutely summable, the functions are $1$-periodic and continuous and we have pointwise equality between $f$ and its series expansion.
Furthermore, this allows us to write the error of approximating the integral by a lattice rule in terms of the Fourier coefficients on the ``dual lattice''. By interchanging the order of 
summation and using the character property of lattice points, we obtain:
\begin{align}
  \notag
  Q_N(f, \bsz)
  -
  I(f)
  &=
  \sum_{\bszero \ne \bsm \in \bbZ^s} \hat{f}(\bsm) \left[ \frac1N \sum_{k=0}^{N-1} \rme^{2\pi\icomp (\bsm\cdot\bsz) \, k / N} \right]
  \\
  \label{eq:error}
  &=
  \sum_{\substack{\bszero \ne \bsm \in \bbZ^s \\ \bsm\cdot\bsz \equiv 0 \tpmod{N}}} \hat{f}(\bsm)
  =
  \sum_{\bszero \ne \bsm \in \bbZ^s} \hat{f}(\bsm) \, \delta_N(\bsm \cdot \bsz)
  ,
\end{align}
where $\delta_N$ is defined below. The error is the sum of the Fourier coefficients $\hat{f}(\bsm)$ for all $\bsm \ne \bszero$ in the \emph{dual lattice} which is defined by $\Lambda^\perp(\bsz, N) := \{ \bsm \in \Z^s \mid \bsm \cdot \bsz \equiv 0 \pmod{N} \}$. For later convenience we define the indicator function, for $a \in \bbZ$,
\begin{align*}
  \delta_N(a)
  &:=
  \frac1N \sum_{k=0}^{N-1} \rme^{2\pi\icomp \, a \, k / N}
  =
  \begin{cases}
    1, & \text{if } a \equiv 0 \pmod{N}, \\
    0, & \text{if } a \not\equiv 0 \pmod{N} ,
  \end{cases}
\end{align*}
such that $\delta_N(\bsm \cdot \bsz)$ is the indicator function of the dual lattice $\Lambda^\perp(\bsz, N)$ with respect to the argument $\bsm \in \bbZ^s$.

\subsection{Definition of the function space}

Based on the decay of the Fourier coefficients $\hat{f}(\bsm)$ we will define a function space for our integrands. For given smoothness parameter $\alpha > 1$ and strictly positive weights $\{\gamma_\setu\}_{\setu \subseteq \{1 \mcol s\}}$, we define, for any $\bsm \in \Z^s$,
\begin{equation}\label{eq:decay_func}
	r_{\alpha,\bsgamma}(\bsm)
	:=
	\gamma_{\supp(\bsm)}^{-1} \prod_{j\in\supp(\bsm)} \abs{m_j}^\alpha ,
\end{equation}
where $\supp(\bsm) := \{ j \in \{1 \mcol s\}: m_j \ne 0 \}$ is the support of $\bsm$. We note that we set $\gamma_\emptyset = 1$ such that $r_{\alpha,\bsgamma}(\bszero) = 1$. Using this decay function, we can apply H\"older's inequality to the expression~\eqref{eq:error}. Specifically, by multiplying and dividing each summand by $r_{\alpha,\bsgamma}(\bsm)$ in~\eqref{eq:error}, and then applying H\"older's inequality with $p=\infty$ and $q=1$ we obtain
\begin{align}\label{eq:Holder-infty}
  |Q_N(f, \bsz) - I(f)|
  &\le
  \Bigg( \sup_{\bsm \in \Z^s} |\hat{f}(\bsm)| \, r_{\alpha,\bsgamma}(\bsm) \Bigg)
  \Bigg( \sum_{\bszero \ne \bsm \in \Z^s} \frac{\delta_N(\bsm\cdot\bsz)}{r_{\alpha,\bsgamma}(\bsm)} \Bigg)
  .
\end{align}
We define the first factor in this error bound as the norm of our Banach space $E_{s,\bsgamma}^{\alpha}$,
\begin{align*}
	\|f\|_{E_{s,\bsgamma}^{\alpha}} 
	:= 
	\sup_{\bsm \in \Z^s} |\hat{f}(\bsm)| \, r_{\alpha,\bsgamma}(\bsm),
\end{align*}
and define, for $\alpha > 1$, our weighted function space by
\begin{equation}\label{eq:f_space}
	E_{s,\bsgamma}^{\alpha}
	:=
	\left\{f \in L^2([0,1]^s) \mid \|f\|_{E_{s,\bsgamma}^{\alpha}}  < \infty \right\}
	.
\end{equation}

\begin{remark}\label{rem:props-of-E}
Note that since $\alpha>1$, the membership of $f$ to the space $E_{s,\bsgamma}^{\alpha}$ implies the absolute convergence of its Fourier series, which in turn entails that $f$ is continuous and $1$-periodic with respect to each variable. In addition, if $f \in E_{s,\bsgamma}^{\alpha}$, $f$ has $1$-periodic continuous mixed partial derivatives $f^{(\bstau)}$ for any $\bstau \in \N_0^s$ with all $\tau_j < \alpha - 1$. This can be seen by differentiating the Fourier series of $f$ and checking the absolute summability using the property from~\eqref{eq:f_space} that $|\hat{f}(\bsm)| \le \|f\|_{E_{s,\bsgamma}^{\alpha}} / r_{\alpha,\bsgamma}(\bsm)$ for all $\bsm \in \bbZ^s$. 
These results can also be found in \cite{Kor63}, see also \cite[Section 5.1]{N92b} for a related result.
\end{remark}

\begin{remark}\label{rem:Korobov}
In his works, see, e.g., \cite{Kor59,Kor63,S71}, Korobov mostly considers functions originating from the class
\begin{equation*}
	E_{s}^{\alpha}(C)
	:=
	\left\{ f \in L^2([0,1]^s) \mid \forall\, \bsm \in \Z^s: |\hat{f}(\bsm)| \le C \, r_{\alpha}^{-1}(\bsm) \right\}
	,
\end{equation*}
for a fixed positive constant $C$ and for $\alpha > 1$. This is a subset of the unweighted version of the space 
$E_{s,\bsgamma}^{\alpha}$, i.e., the space $E_{s,\bsgamma}^{\alpha}$ with all $\gamma_{\setu}=1$, 
and where $r_{\alpha}$ is the function in \eqref{eq:decay_func} with all weights $\gamma_{\setu}=1$.
This space only contains $f$ for which the unweighted norm, i.e., with all $\gamma_\setu = 1$, satisfies $\|f\|_{E_{s,\bsgamma}^{\alpha}} \le C$.
This means that Korobov's results can be readily interpreted for the unweighted version of $E_{s,\bsgamma}^{\alpha}$ by replacing the constant $C$
in his bounds by the supremum norm $\|f\|_{E_{s,\bsgamma}^{\alpha}}$.
The introduction of weights, as they appear in the modern quasi-Monte Carlo theory, in the definition of the function space $E_{s,\bsgamma}^{\alpha}$ 
makes it possible to extend Korobov's results and to make them applicable to integration problems with large dimension $s$.
\end{remark}

\begin{remark}
 Note that, for technical reasons, we restrict ourselves to the cases where all weights $\gamma_{\setu}$, $\setu\subseteq \{1 \mcol s\}$, 
 are strictly positive in our considerations. We assume that analogous results would hold for the situation where weights 
 are allowed to be zero. We refrain from considering this more general situation in order to make the arguments in the paper not too technical. 
\end{remark}

\subsection{The worst-case error of rank-$1$ lattice rules}

We can now state the formula for the worst-case error $e_{N,s,\alpha,\bsgamma}$ for $N$-point rank-$1$ lattice rules in the space $E_{s,\bsgamma}^{\alpha}$. 
This result and similar results are well known, see, e.g., \cite{N92b,SJ94,DKS13,N14}.

\begin{theorem}[Rank-$1$ lattice rule worst-case error] \label{thm:wce}
	Let $N,s \in \N$, $\alpha > 1$ and a sequence of positive weights $\bsgamma = \{\gamma_{\setu}\}_{\setu \subseteq \{1 \mcol s\}}$ be given. Then the worst-case error $e_{N,s,\alpha,\bsgamma}(\bsz)$ for the rank-$1$ lattice rule $Q_N(\cdot,\bsz)$ in the space $E_{s,\bsgamma}^{\alpha}$ satisfies
	\begin{equation}\label{eq:wce-infty}
		e_{N,s,\alpha,\bsgamma}(\bsz)
		=
		\sum_{\substack{\bszero \ne \bsm \in \Z^s \\ \bsm \cdot \bsz \equiv 0 \tpmod{N}}} r_{\alpha,\bsgamma}^{-1}(\bsm) 
		=
		\sum_{\bszero \ne \bsm \in \Z^s} \frac{\delta_N(\bsm \cdot \bsz)}{r_{\alpha,\bsgamma}(\bsm)}
		.
	\end{equation}
\end{theorem}
\begin{proof}
	The proof of the statement is given in the appendix.
\end{proof}

The worst-case error~\eqref{eq:wce-infty} is referred to as $P_\alpha$ in \cite{N92b,SJ94} and other sources.

\begin{remark}\label{rem:other-spaces}
Although, as we noted in Remark~\ref{rem:Korobov}, Korobov used the supremum norm, as we do here, in the recent literature on lattice rules, see, e.g., \cite[Appendix A.1]{NW08}, the name ``Korobov space'' is most often used in the Hilbert space setting. That is, instead of taking $p=\infty$ and $q=1$ when applying the H\"older inequality to~\eqref{eq:error}, one uses $p=q=2$. If we then multiply and divide by $(r_{\alpha,\bsgamma}(\bsm))^{1/2}$ in~\eqref{eq:error} before applying H\"older's inequality, we arrive at the bound
\begin{align}\label{eq:Holder-2}
	|Q_N(f, \bsz) - I(f)|
	&\le
	\Bigg( \sum_{\bsm \in \Z^s} |\hat{f}(\bsm)|^2 \, r_{\alpha,\bsgamma}(\bsm) \Bigg)^{1/2}
	\Bigg( \sum_{\bszero \ne \bsm \in \Z^s} \frac{\delta_N(\bsm\cdot\bsz)}{r_{\alpha,\bsgamma}(\bsm)} \Bigg)^{1/2}
	.
\end{align}
The first factor on the right-hand side of \eqref{eq:Holder-2} 
can be used as the norm for this Hilbert space setting, and similar to Theorem~\ref{thm:wce} the second factor can be shown to equal the worst-case error for the lattice rule 
with generating vector $\bsz$ in this space. Also here we require $\alpha > 1$ (which is necessary to guarantee summability of the worst-case error part).
From~\eqref{eq:Holder-2} we note that the square of the worst-case error in this Hilbert space setting equals the worst-case error for our function space $E_{s,\bsgamma}^\alpha$, with the same $\bsgamma$ and $\alpha$, because of the specific choices of multiplying and dividing by $(r_{\alpha,\bsgamma}(\bsm))^{1/2}$ and $r_{\alpha,\bsgamma}(\bsm)$, respectively, when applying H\"older's inequality, but of course, the effect of $r_{\alpha,\bsgamma}(\bsm)$ in the norm is different. Similar to Remark~\ref{rem:props-of-E}, with the specific choice of multiplying and dividing by $(r_{\alpha,\bsgamma}(\bsm))^{1/2}$, functions in this Hilbert space setting with the norm defined as in~\eqref{eq:Holder-2} have $1$-periodic continuous mixed partial derivatives $f^{(\bstau)}$ for any $\bstau \in \bbN_0^s$ with all $\tau_j < \alpha/2$, and all $f^{(\bstau)}$ are $L_2$-integrable when all $\tau_j \le \alpha/2$. This can also be found in \cite{NW08}. When applying the H\"older inequality for general $p$ and $q$ one could also just multiply and divide by the $r_{\alpha,\bsgamma}(\bsm)$ factor, resulting in a $p$-th and $q$-th power of $r_{\alpha,\bsgamma}(\bsm)$ in the norm and worst-case error, respectively, see \cite{N14,Hic98}, in which case one requires $\alpha > 1/q$. 
This choice is also popular for the Hilbert space setting and requires $\alpha > 1/2$.
\end{remark}

It is well known, see, e.g., \cite{SJ94}, that the optimal convergence rate of the worst-case integration error in the space $E_{s,\bsgamma}^\alpha$ is of order $\calO(N^{-\alpha})$ 
while for the Hilbert space setting mentioned above it is $\calO(N^{-\alpha/2})$ both with $\alpha>1$. This is consistent with how the respective worst-case errors in these spaces are related. Note also that any results on the error in $E_{s,\bsgamma}^\alpha$ immediately yield results on the error in the Hilbert space setting, which in turn can be related to the worst-case error of ``tent-transformed'' lattice rules in certain Sobolev spaces of functions whose mixed partial derivatives of order~$1$ and~$2$ in each variable are square integrable 
(with $\alpha=2$ and $\alpha=4$ in our notation, respectively), see \cite{DNP14,CKNS16,GSY19} for further details.

\subsection{The existence of good lattice rules}

The previous section provided us with an expression for the worst-case error $e_{N,s,\alpha,\bsgamma}$. 
This naturally raises the question whether there exist lattice rules whose worst-case error is sufficiently small and satisfies a certain asymptotic error decay. 
Such existence results are typically proven using the worst-case error expression~\eqref{eq:wce-infty} itself, but that expression depends on 
$\alpha$, and requires $\alpha > 1$ for the infinite sums to converge. To avoid the dependence on 
$\alpha$ we take another commonly used approach by considering truncated sums such that we can work with the limiting case of $\alpha = 1$.
In Theorem~\ref{thm:ex_res} we will then bound the worst-case error~\eqref{eq:wce-infty}, for any $\alpha > 1$, in terms of bounds for this truncated quantity with $\alpha = 1$.
To this end, define the quality measure $T(N,\bsz)$ for $N \in \N$, $\bsz \in \Z^s$, by
\begin{equation*}
	T(N,\bsz) 
	:= 
	\sum_{\bszero \ne \bsm \in M_{N,s}} \frac{\delta_N(\bsm \cdot \bsz)}{r_{1,\bsgamma}(\bsm)}, 
\end{equation*}      
where
\begin{equation*}
	M_{N,s}
	:=
	\{-(N-1),\ldots,N-1\}^s.
\end{equation*}
Furthermore, we write
\[
	M_{N,s}^*
	:=
	\left(\{-(N-1),\ldots,N-1\} \setminus \{0\}\right)^s 
	.
\]

Likewise we define the auxiliary quantity
\begin{equation*}
	T_\alpha(N,\bsz) 
	:= 
	\sum_{\bszero \ne \bsm \in M_{N,s}} \frac{\delta_N(\bsm \cdot \bsz)}{r_{\alpha,\bsgamma}(\bsm)}
	.
\end{equation*}

We remark that two very similar quantities, usually referred to as $R$ and $R_\alpha$, appear in the literature, where the sums are instead truncated to
\begin{equation*}
	C_{N,s} := \left((-N/2, N/2] \cap \bbZ\right)^s
	.
\end{equation*} 
For example, such an approach was used by Niederreiter and Sloan \cite{NS90}, and later also Joe \cite{J06}, to obtain lattice point sets with low (weighted) star discrepancy. 
The same approach with relation to the worst-case error was, e.g., used by Disney and Sloan \cite{DS91}, and others, see also the book \cite{LP14}, to study
lattice rules yielding essentially optimal convergence rates. We remark that these authors study slightly different (mostly unweighted) settings for the integration problem 
than the one considered here. For further references to classical results based on $R$ and $R_\alpha$, we refer to the books \cite{N92b} and \cite{SJ94}.

For historical reasons and in order to stay close to the method of Korobov, we continue to use the quantities $T(N,\bsz)$ and $T_\alpha(N,\bsz)$ in the following.

The following theorem assures the existence of good generating vectors $\bsz=(z_1,\ldots,z_s) \in \Z^s$ for $N$ prime in terms of~$T(N,\bsz)$. 
We state the result only for prime $N$ to avoid too many technical details, but a similar bound should hold for composite $N$. Indeed, below, we will show a constructive bound of the same flavor for $N$ of the form $N=2^n$ (cf. Theorem \ref{thm:T_target_CBCDBD}).

\begin{theorem}[Existence of good rules w.r.t.\ $T$] \label{thm:existence-T}
	Let $\bsgamma = \{\gamma_{\setu}\}_{\setu \subseteq \{1 \mcol s\}}$ be a sequence of positive weights. For every prime $N$ there exists a generating vector $\bsz \in \{1,\ldots,N-1\}^s$ such that
	\begin{align*}
		T(N,\bsz) 
		= 
		\sum_{\bszero \ne \bsm \in M_{N,s}} \frac{\delta_N(\bsm \cdot \bsz)}{r_{1,\bsgamma}(\bsm)} 
		\le
		\frac{2}{N} \sum_{\emptyset\neq\setu\subseteq \{1 \mcol s\}}\gamma_\setu \, (2(1 + \ln N))^{\abs{\setu}}
		.
	\end{align*}	
\end{theorem}
\begin{proof}
	The proof of the statement can be found in the appendix.
\end{proof}

In the following proposition we bound the truncation error between the worst-case error $e_{N,s,\alpha,\bsgamma}$ as in 
Theorem \ref{thm:wce} and its ``restriction'' to the set $M_{N,s}$.

\begin{proposition}[Truncation error for $T_\alpha$]\label{prop:trunc_error}
	Let $\bsgamma = \{\gamma_{\setu}\}_{\setu \subseteq \{1 \mcol s\}}$ be a sequence of positive weights and let 
	$\bsz = (z_1,\ldots,z_s) \in \Z^s$ with $\gcd(z_j,N)=1$ for all $j=1,\ldots,s$. Then, for $\alpha>1$, we have that
	\begin{align*}
	    e_{N,s,\alpha,\bsgamma}(\bsz) - T_\alpha(N, \bsz)
	    =
		\sum_{\bszero \ne \bsm \in \Z^s} \frac{\delta_N(\bsm \cdot \bsz)}{r_{\alpha,\bsgamma}(\bsm)}
		-
		\sum_{\bszero \ne \bsm \in M_{N,s}} \frac{\delta_N(\bsm \cdot \bsz)}{r_{\alpha,\bsgamma}(\bsm)} 
		\le
		\frac{1}{N^\alpha} \sum_{\emptyset\neq \setu \subseteq \{1 \mcol s\}} \gamma_\setu \, (4\zeta (\alpha))^{\abs{\setu}}
		.
	\end{align*}
\end{proposition}
\begin{proof}
	For a subset $\emptyset\neq\setu \subseteq \{1 \mcol s\}$ and $i \in \setu$, we write for short $\bsm_{\setu\setminus\{i\}}, \bsz_{\setu\setminus\{i\}} \in \Z^{\abs{\setu}-1}$ 
	to denote the projections on those components in $\setu\setminus\{i\}$. The difference $e_{N,s,\alpha,\bsgamma} - T_\alpha$ can be rewritten as 
	\begin{align*}
		\sum_{\emptyset\neq \setu \subseteq \{1 \mcol s\}} 
		\left(
		\sum_{\bsm_\setu \in \Z_\ast^{\abs{\setu}}} \frac{\delta_N (\bsm_\setu \cdot \bsz_\setu)}{r_{\alpha,\gamma_\setu} (\bsm_\setu)}
		- \sum_{\bsm_\setu \in M_{N,|\setu|}^{\ast}} \frac{\delta_N (\bsm_\setu \cdot \bsz_\setu)}{r_{\alpha,\gamma_\setu} (\bsm_\setu)}
		\right), 
	\end{align*}
	motivating us to define, for $\emptyset \ne \setu\subseteq\{1 \mcol s\}$,
	\begin{equation*}
		T_\setu:= \sum_{\bsm_\setu \in \Z_\ast^{\abs{\setu}}} \frac{\delta_N (\bsm_\setu \cdot \bsz_\setu)}{r_{\alpha,\gamma_\setu} (\bsm_\setu)}
		- \sum_{\bsm_\setu \in M_{N,|\setu|}^{\ast}} \frac{\delta_N (\bsm_\setu \cdot \bsz_\setu)}{r_{\alpha,\gamma_\setu} (\bsm_\setu)} .
	\end{equation*}
	In the following we distinguish two cases. \\
	
	\noindent \textbf{Case 1:} Suppose that $\abs{\setu}=1$ such that $\setu = \{j\}$ for some $j \in \{1 \mcol s\}$.
	Then we have
	\begin{align*}
		T_\setu
		=
	    T_{\{j\}}
		&=
		\sum_{m_j \in\Z_\ast} \frac{\delta_N (m_j z_j)}{r_{\alpha,\gamma_{\{j\}}}(m_j)} - 
		\sum_{m_j\in M_{N,1}^\ast}\frac{\delta_N (m_j z_j)}{r_{\alpha,\gamma_{\{j\}}}(m_j)}
		=
		\sum_{\abs{m_j} \ge N} \frac{\delta_N (m_j z_j)}{r_{\alpha,\gamma_{\{j\}}}(m_j)} \\
		&=
		\sum_{\abs{m_j} \ge N} \gamma_{\{j\}} \frac{\delta_N (m_j z_j)}{\abs{m_j}^\alpha} 
		= 
		2\gamma_{\{j\}} \sum_{t=1}^\infty \frac{1}{(tN)^\alpha}
		= 
		\frac{2\zeta(\alpha)}{N^\alpha} \gamma_{\{j\}}
		,
	\end{align*}
	which follows since $\gcd(z_j,N)=1$ and thus $m_j z_j \equiv 0 \pmod{N}$ if and only if $m_j$ is a multiple of~$N$. 
	Here $\zeta(\alpha) := \sum_{n=1}^{\infty} \frac{1}{n^\alpha}$, $\alpha > 1$, denotes the Riemann zeta function.\\
	
	\noindent \textbf{Case 2:} Suppose that $\abs{\setu}\ge 2$. In this case, we estimate
	\begin{equation*}
		T_\setu 
		\le 
		\sum_{i\in\setu} \sum_{\bsm_{\setu\setminus \{i\}} \in \Z_\ast^{\abs{\setu}-1}}
		\sum_{\abs{m_i}\ge N}\frac{\delta_N (m_i z_i + \bsm_{\setu \setminus \{i\}}\cdot \bsz_{\setu \setminus \{i\}})}{r_{\alpha,\gamma_\setu}(\bsm_\setu )}.
	\end{equation*}
	For $\bsm_{\setu\setminus \{ i\}} \in \Z_\ast^{\abs{\setu}-1}$, we write $b = \bsm_{\setu \setminus \{i\}}\cdot \bsz_{\setu \setminus \{i\}}$, and consider the expression
	\begin{align*}
		\sum_{|m_i|\ge N}\frac{\delta_N (m_i z_i + b)}{r_{\alpha,\gamma_\setu}(\bsm_\setu )}
		&=
		\gamma_\setu \sum_{|m_i|\ge N}\frac{\delta_N (m_i z_i + b)}{\prod_{j\in\setu}\abs{m_j}^\alpha}
		=
		\gamma_\setu \bigg(\prod_{\substack{j\in\setu\\ j\ne i}} \abs{m_j}^{-\alpha}\bigg)
		\sum_{|m_i| \ge N} \frac{\delta_N (m_i z_i + b)}{\abs{m_i}^\alpha} \\
		&=
		\gamma_\setu \bigg(\prod_{\substack{j\in\setu\\ j\ne i}} \abs{m_j}^{-\alpha}\bigg)
		\sum_{t=1}^\infty \sum_{m_i=tN}^{(t+1)N-1} 
		\left[ \frac{\delta_N(m_i z_i + b)}{\abs{m_i}^\alpha} + \frac{\delta_N(m_i z_i - b)}{\abs{m_i}^\alpha} \right]
		\\
		&\le
		\gamma_\setu \bigg(\prod_{\substack{j\in\setu\\ j\ne i}} \abs{m_j}^{-\alpha}\bigg)
		\sum_{t=1}^\infty \frac{1}{(tN)^\alpha}
		\sum_{m_i=tN}^{(t+1)N-1}
		\big[ \delta_N(m_i z_i + b) + \delta_N(m_i z_i - b) \big] \\
		&=
		\gamma_\setu \bigg(\prod_{\substack{j\in\setu\\ j\ne i}} \abs{m_j}^{-\alpha}\bigg)
		\frac{1}{N^\alpha} \sum_{t=1}^\infty \frac{1}{t^\alpha} 
		\sum_{m_i=tN}^{(t+1)N-1}
		\big[ \delta_N(m_i z_i + b) + \delta_N(m_i z_i - b) \big]
		\\
		&=
		\frac{2\zeta(\alpha)}{N^\alpha} \gamma_\setu \prod_{\substack{j\in\setu\\ j\ne i}} \abs{m_j}^{-\alpha}
		,
	\end{align*}
	where the last equality follows since
	\begin{align*}
		\sum_{m = tN}^{(t+1)N - 1}
		\delta_N(m z + b)
		&=
		\sum_{m = 0}^{N - 1}
		\delta_N(m z + b)
		=
		1
		,
	\end{align*}
	which holds since for any $b, z \in \Z$ with $\gcd(z, N) = 1$ the congruence $m z + b \equiv 0 \pmod{N}$ has the unique solution $m \equiv -z^{-1} b \pmod{N} \in \Z_N$, 
	see also \cite[Corollary of Proposition 1]{Kor63}. Hence, we can estimate $T_{\setu}$, for $|\setu| \ge 2$, by
	\begin{align*}
		T_\setu
		&\le
		\gamma_\setu \frac{2\zeta (\alpha)}{N^\alpha}\ \sum_{i\in \setu}\ \ \sum_{\bsm_{\setu\setminus \{i\}} 
		\in \Z_\ast^{\abs{\setu}-1}} 
		\prod_{\substack{j\in\setu\\ j\neq i}} \abs{m_j}^{-\alpha}
		=
		\gamma_\setu \frac{2\zeta (\alpha)}{N^\alpha}\ \sum_{i\in \setu} \left(2\sum_{m=1}^\infty \frac{1}{m^\alpha}\right)^{\abs{\setu}-1} \\
		&=
		\gamma_\setu \frac{1}{N^\alpha} \sum_{i\in \setu} (2\zeta(\alpha))^{\abs{\setu}}
		=
		\gamma_\setu \frac{1}{N^\alpha} (2\zeta(\alpha))^{\abs{\setu}} \abs{\setu}
		\le
		\gamma_\setu \frac{1}{N^\alpha} (4\zeta(\alpha))^{\abs{\setu}}.
	\end{align*}
	In summary, we obtain, using the results for both cases from above,
	\begin{align*}
		\sum_{\emptyset\neq \setu \subseteq \{1 \mcol s\}} T_\setu 
		\le
		\sum_{\emptyset\neq \setu \subseteq \{1 \mcol s\}}  \gamma_\setu \frac{1}{N^\alpha} (4\zeta (\alpha))^{\abs{\setu}}
		,
	\end{align*}
	which yields the claimed result.
\end{proof}

The result on the truncation error in Proposition~\ref{prop:trunc_error} will be a key ingredient for the worst-case error analysis in the subsequent sections.
Furthermore, combining Proposition~\ref{prop:trunc_error} and Theorem~\ref{thm:existence-T}, we immediately obtain the following existence result 
for prime $N$.

\begin{theorem}[Existence of good rules w.r.t.\ the worst-case error] \label{thm:ex_res}
	Let $N$ be a prime number and let $\bsgamma=\{\gamma_\setu\}_{\setu \subseteq \{1 \mcol s\}}$ be positive weights with $\gamma_\emptyset =1$. 
	Then there exists a generating vector $\bsz \in \{1,\ldots,N-1\}^s$ such that, for all $\alpha > 1$, the worst-case error $e_{N,s,\alpha,\bsgamma^\alpha}(\bsz)$ satisfies
	\begin{align*}
		e_{N,s,\alpha,\bsgamma^{\alpha}}(\bsz)
		&\le
		\frac{1}{N^{\alpha}} \left( \sum_{\emptyset \neq \setu \subseteq\{1 \mcol s\}} \gamma_\setu^\alpha \, (4\zeta (\alpha))^{\abs{\setu}} + 
		\left( \sum_{\emptyset \neq \setu \subseteq\{1 \mcol s\}} 2\gamma_\setu \, (2(1 + \ln N))^{\abs{\setu}}\right)^{\!\!\alpha} \right)
	\end{align*}  
	with weight sequence $\bsgamma^{\alpha}=\{\gamma_\setu^{\alpha}\}_{\setu \subseteq \{1 \mcol s\}}$.
\end{theorem}
\begin{proof}
    Since $N$ is prime we have in particular that $\gcd(z_j,N)=1$ for all $j=1,\ldots,s$.
	Therefore, by Proposition~\ref{prop:trunc_error} the worst-case error satisfies
	\begin{align*}
		e_{N,s,\alpha,\bsgamma^{\alpha}}(\bsz)
		&\le
		\frac{1}{N^\alpha} \sum_{\emptyset\neq \setu \subseteq \{1 \mcol s\}} \gamma_u^\alpha \, (4\zeta (\alpha))^{\abs{\setu}} 
		+ 
		\sum_{\bszero \ne \bsm \in M_{N,s}} \frac{\delta_N(\bsm \cdot \bsz)}{r_{\alpha,\bsgamma^{\alpha}}(\bsm)} 
		.
	\end{align*} 
	Then, using the fact that for $\alpha \ge 1$ we have $\sum_{i \in I} x_i^{\alpha} \le (\sum_{i \in I} x_i)^\alpha$ for $x_i \ge 0$ and
	countable index set $I$, Theorem~\ref{thm:existence-T} guarantees the existence of a $\bsz \in \{1,\ldots,N-1\}^s$ such that
	\begin{align*}
		\sum_{\bszero \ne \bsm \in M_{N,s}} \frac{\delta_N(\bsm \cdot \bsz)}{r_{\alpha,\bsgamma^{\alpha}}(\bsm)}
		\le
		\left( \sum_{\bszero \ne \bsm \in M_{N,s}} \frac{\delta_N(\bsm \cdot \bsz)}{r_{1,\bsgamma}(\bsm)} \right)^{\!\!\alpha}
		\le \frac{1}{N^{\alpha}} \left(\sum_{\emptyset \neq \setu \subseteq\{1 \mcol s\}} 2\gamma_\setu \, (2(1 + \ln N))^{\abs{\setu}}\right)^{\!\!\alpha}
		.
	\end{align*} 
	Combining both estimates yields the claim.
\end{proof}

Theorem \ref{thm:ex_res} shows the existence of rank-$1$ lattice rules which achieve the almost optimal error convergence rate of $\calO(N^{-\alpha})$ in the space $E_{s,\bsgamma}^{\alpha}$ for prime $N$, see \cite[Theorem~1]{Kor63} for a lower bound for general cubature rules with arbitrary integration weights, provided certain conditions 
on the weights $\gamma_{\setu}$ are met. In \cite{K03} and \cite{DSWW06} it is shown that if the weights sequence, which is $\bsgamma^{\alpha}$ in our case, 
satisfies $\sum_{|\setu|<\infty} \gamma_\setu^{\alpha/\alpha} = \sum_{|\setu|<\infty} \gamma_\setu < \infty$ then this rate can be achieved with the implied constant independent of the dimension. The weights in our space of smoothness $\alpha$ are expressed as $\bsgamma^\alpha$ in compliance with this summability condition.
It is not known if this is a necessary condition. If our weights in the space would be just $\bsgamma$ then this would show a much weaker condition on the summability 
than is currently known. Exactly the same approach with weights of the form $\bsgamma^\alpha$ in the space $E_{s,\bsgamma^\alpha}^\alpha$ occurs in \cite{HN03}. \\

In this context, Korobov called good rules with respect to the $T$ criterion  \emph{optimal coefficients}, see, e.g., \cite{Kor63,Kor82}. We extend this terminology to the weighted setting. (We remark that Korobov defined the concept of optimal coefficients in slightly different ways in different papers.) Similar terminology appears in the work by Hlawka \cite{H62} which coins this \emph{the method of good lattice points}.

\begin{definition}[Optimal coefficients modulo~$N$] \label{def:opt_coeff}
	For $N \in \N$, let $z_1 = z_1(N),\ldots,z_s = z_s(N) \in \Z$ be integers and let $\bsgamma=\{\gamma_\setu\}_{\setu \subseteq \{1 \mcol s\}}$ be a sequence of positive weights.
	If there exists a positive constant $C(\bsgamma,\delta)$, independent of $N$, such that for infinitely many values of $N$ we have 
	\begin{equation} \label{eq:opt_coeff}
		T(N,\bsz)\le C(\bsgamma,\delta) \, N^{-1+\delta} 
		\quad\text{for any}\quad \delta > 0 
		,
	\end{equation}
	then the numbers $z_1 = z_1(N), \ldots, z_s = z_s(N)$ (or, equivalently, the vectors $\bsz(N)= (z_1(N), \ldots,$ $z_s(N))$, 
	for those values of $N\in\N$ for which \eqref{eq:opt_coeff} holds) are called \emph{optimal coefficients modulo~$N$}.
\end{definition}

In the following section, we will introduce construction algorithms which devise generating vectors for lattice rules which are optimal coefficients modulo $N$ 
according to Definition \ref{def:opt_coeff}.

\section{Construction methods for rank-$1$ lattice rules} \label{sec:constr}

In this section, we introduce construction methods to design rank-$1$ lattice rules exhibiting the desired worst-case error behavior. At first, we establish some auxiliary statements which will be needed in the further analysis. We note that $T(N, \bsz)$ can be written as
\begin{align}\label{eq:TN-subsets}
	T(N, \bsz)
	&=
	\sum_{\bszero \ne \bsm \in M_{N,s}} \frac{\delta_N(\bsm\cdot\bsz)}{r_{1,\bsgamma}(\bsm)}
	=
	\sum_{\emptyset \ne \setu \subseteq \{1 \mcol s\}} \frac{\gamma_\setu}{N} \sum_{k=0}^{N-1}
	\left[
	\prod_{j\in\setu} \sum_{m_j \in M_{N,1}^\ast} \frac{\rme^{2\pi\icomp k m_j z_j/N }}{|m_j|}
	\right]
	.
\end{align}
The expression in square brackets can be considered as a function which has non-zero Fourier coefficients only for indices in the truncated box $M_{N,|\setu|}^\ast$, 
and likewise we can define a similar function which is not truncated. So, for $x \in (0,1)$, we define two functions,
\begin{equation}\label{eq:vartheta}
	\vartheta_N(x)
	:=
	\sum_{m \in M_{N,1}^\ast}
	\frac{\rme^{2\pi\icomp m x}}{|m|}
	\qquad\text{and}\qquad
	\vartheta(x)
	:=
	\sum_{m \in \ZZ_\ast}
	\frac{\rme^{2\pi\icomp m x}}{|m|}
	.
\end{equation}
In Lemma~\ref{lem:ln-sin} we will first show that $\vartheta(x)$ can equivalently be written as the function $-2\ln(a \sin(\pi x))$ modulo a constant depending on the parameter~$a \ge 1$.
Then in Lemma~\ref{lemma:expr_target_function} we show how $\vartheta_N(x)$ can be approximated by  $-2\ln(a \sin(\pi x))$ with $a=1$.
(Later we will use $-2\ln(a \sin(\pi x))$ as a substitute, in Theorems~\ref{thm:T_target_CBCDBD} and~\ref{thm:T_target_CBC}.)
Finally, we show in Lemma~\ref{lemma:diff_prod} how to deal with the difference of products of these functions.

\begin{lemma}\label{lem:ln-sin}
	For $x \in (0,1)$ and $a \ge 1$ we have
	\begin{equation} \label{eq:Fourier_log_sin}
		-2 \ln (a \sin (\pi x))
		=
		\ln(4) - 2\ln(a) + \sum_{m \in \Z_\ast} \frac{\rme^{2\pi \icomp m x}}{\abs{m}} 
		=
		\ln(4) - 2\ln(a) + \vartheta(x) 
		.
	\end{equation}
\end{lemma}

\begin{proof}
For $\sigma \in \{-1,1\}$, Euler's formula yields the identity
\begin{align*}
	\ln (a \sin (\pi x))
	&=
	\ln(a) + \ln\left(\frac{\rme^{\icomp \pi x} - \rme^{-\icomp \pi x}}{2\icomp}\right)
	=
	\ln(a) + \ln\left(\rme^{\sigma \icomp \pi x} \sigma \left(\frac{1 - \rme^{-2 \sigma \icomp \pi x}}{2\icomp}\right)\right) \\
	&=
	\ln(a) + \sigma \icomp \pi x - \ln(2\sigma\icomp) + \ln\left(1 - \rme^{-2 \sigma \icomp \pi x}\right) \\
	&=
	\ln(a) + \sigma \icomp \pi x - \ln(2) - \sigma \frac{\icomp \pi}{2} + \ln\left(1 - \rme^{-2 \sigma \icomp \pi x}\right) .
\end{align*}	
We recall that the Maclaurin series of $\ln(1-y)$ equals $-\sum_{m=1}^\infty \frac{y^m}{m}$ which converges to $\ln(1-y)$
for $y=\rme^{2\pi \icomp x}$ provided that $x \not\in \Z$. Then, averaging over both choices of $\sigma \in \{-1,1\}$ yields  
\begin{align*}
	\ln (a \sin (\pi x))
	&=
	\ln(a) - \ln(2) - \frac12 \sum_{m=1}^\infty \left( \frac{\rme^{-2 \icomp \pi m x} + \rme^{2 \icomp \pi m x}}{m} \right)
\end{align*}
and we note that the series is convergent for $x \in (0,1)$.
\end{proof}

Figure~\ref{fig:qual-func} depicts the function $-2 \ln (a \sin (\pi x))$ for $a=1$ and $a=2$ and illustrates its divergence 
towards infinity on the boundaries of the interval $[0,1]$. We therefore bear in mind that it cannot be evaluated in $x=0$ and $x=1$.

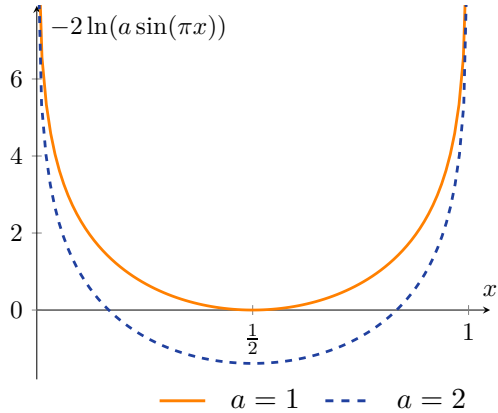
\begin{figure}[H]
	\centering
	\begin{tikzpicture}[scale=.87]
	\begin{axis}
		[
		xlabel=$x$,
		ylabel=$-2 \ln (a \sin(\pi x))$,
		ymin=-1.8,
		ymax=7.9,
		axis lines=middle,
		xmin=0,
		xmax=1.05,
		xtick=\empty,
		extra x ticks={0.5,1},
		extra x tick labels={$\frac12$,1},
		extra y ticks={0},
		extra y tick labels={0},
		tick label style={font=\normalsize},
		x label style={at={(axis description cs:1,0.2)},anchor=south},
		y label style={at={(axis description cs:0.01,1)},anchor=north west},
		]
		\addplot[orange,very thick,domain=0.002:0.998,samples=100,draw=orange] {-2*ln(sin(deg(pi*\x)))};
		\addplot[blue,dashed,very thick,domain=0.001:0.999,samples=100,draw={rgb:red,1;green,2;blue,5}] {-2*ln(2*sin(deg(pi*\x)))};
	\end{axis}
	\end{tikzpicture}
	\vskip\baselineskip
	\vspace{-15pt}
	\begin{tikzpicture}
		\hspace{0.05\linewidth}
		\begin{customlegend}[
			legend columns=2,legend style={align=left,draw=none,column sep=1.5ex},
			legend entries={$a=1$, $a=2$}
			]
			\addlegendimage{color=orange, solid, line width=1.1pt}
			\addlegendimage{color={rgb:red,1;green,2;blue,5}, dashed, line width=1.1pt}
		\end{customlegend}
	\end{tikzpicture}
	\caption{Behavior of the function $-2 \ln (a \sin(\pi x))$ for $a=1$ and $a=2$ on the interval $[0,1]$.}
	\label{fig:qual-func}
\end{figure}

We can now use this function to approximate the truncated series.
\begin{lemma}\label{lemma:expr_target_function}
	Let $N \in \N$, then for any $x \in (0,1)$ there exists a $\tau(x) \in \R$ with $|\tau(x)| \le 1$ such that 
	\begin{equation*}
		\ln (\sin^{-2}(\pi x) )
		=
		\ln 4 + \sum_{m \in M_{N,1}^\ast} \frac{\rme^{2\pi \icomp m x}}{|m|} + \frac{\tau(x)}{N \|x\|}
		=
		\sum_{m = -(N-1)}^{N-1} \frac{\rme^{2\pi \icomp m x}}{b(m)} + \frac{\tau(x)}{N \|x\|},
	\end{equation*}
	with coefficients
	\begin{align} \label{eq:def-b-dnint}
		b(m)
		&:= 
		\left\{\begin{array}{cc}
		\abs{m}, & {\text{for }} m \ne 0 , \\ 
		1/(\ln 4), & {\text{for }} m = 0 ,
		\end{array}\right.
	\end{align}
   where $\|x\|$ denotes the distance to the nearest integer of $x$, i.e.,
   \[
		\|x\|
		:=
		\min\{\{x\},1-\{x\}\}.
   \]
\end{lemma}
\begin{proof}
	The proof of the statement is given in the appendix, see also \cite{Kor63}.
\end{proof}

The following lemma provides a result for the difference of two products. Different variants of such a result can be found in the literature, but, as this lemma 
is crucial in showing our main result, we provide a proof for the sake of completeness.

\begin{lemma} \label{lemma:diff_prod}
	For $1\le j \le s$, let $u_j, v_j$ and $r_j$ be real numbers which satisfy:
	\begin{align*}
		(a) \quad u_j = v_j + r_j, \quad
		(b) \quad |u_j| \le \bar{u}_j, \quad
		(c) \quad \bar{u}_j \geq 1 .
	\end{align*}
	Then, for any subset $\emptyset \ne \setu \subseteq \{1,\ldots,s\}$ there exists a $\theta_{\setu}$ with $|\theta_{\setu}| \le 1$ such that the following identity holds,
	\begin{align}\label{eq:diff_prod} 
		\prod_{j \in \setu} u_j
		&=
		\prod_{j \in \setu} v_j + \theta_{\setu} \left(\prod_{j \in \setu} (\bar{u}_j+|r_j|) \right) \sum_{j \in \setu} |r_j| .
	\end{align}
\end{lemma}

\begin{proof}
	Note that for a subset $\emptyset \ne \setu \subseteq \{1,\ldots,s\}$ the expansion of $\prod_{j\in\setu} (v_j + r_j)$ 
	equals the sum over all possible $2^{|\setu|}$ products
	which select either $v_j$ or $r_j$ as a factor for each $j \in \setu$. Therefore
	\[
	\prod_{j\in\setu} u_j
	=
	\prod_{j\in\setu} (v_j + r_j)
	=
	\sum_{\setv \subseteq \setu} \prod_{\substack{j\notin\setv\\j\in\setu}} v_j \prod_{j\in\setv} r_j
	=
	\prod_{j\in\setu} v_j
	+
	\sum_{\emptyset \ne \setv \subseteq \setu} \prod_{\substack{j\notin\setv\\j\in\setu}} v_j \prod_{j\in\setv} r_j
	.
	\]
	We note that the last sum comprises all possible factor selections of either $v_j$ or $r_j$ for all $j\in\setu$ with the condition of always choosing at least one
	$r_j$ (by $\setv \ne \emptyset$). We can group all these selections by the index $i \in \setu$ which is the highest index in $\setu$ for which the factor $r_i$ was selected.
	Then, for all $j$ with $j > i$, factors $v_j$ are chosen, that is, $\prod_{j\in\setu,j>i} v_j$, and for all $j < i$ we have all possible factor selections of either $r_j$ or $v_j$, 
	that is, $\prod_{j\in\setu,j<i} (v_j+r_j)$. Hence, we obtain
	\begin{align} \label{eq:diff-expressions}
		\prod_{j\in\setu} u_j - \prod_{j\in\setu} v_j 
		=
		\sum_{\emptyset \ne \setv \subseteq \setu} \prod_{\substack{j\notin\setv\\j\in\setu}} v_j \prod_{j\in\setv} r_j
		&=
		\sum_{i\in\setu} r_i \prod_{\substack{j\in\setu \\ j > i}} v_j \prod_{\substack{j \in \setu \\ j < i}} (v_j + r_j)
		,
	\end{align}
	from which the result follows by bounding the absolute value of \eqref{eq:diff-expressions} using $|v_j| = |u_j - r_j| \le \bar{u}_j + |r_j|$ and 
	$|v_j + r_j| =  |u_j| \le \bar{u}_j \le \bar{u}_j + |r_j|$, and by multiplying with the factor $\bar{u}_i + |r_i| \ge 1$ for $j = i$.
\end{proof}

We also require the following lemma, which is inspired by \cite[Lemma 3]{HN03}. We will use the summability condition on our weight sequence 
to bound the constants in our error bounds independently of the dimension.

\begin{lemma}\label{lem:logN}
	 Let $\{\gamma_\setu\}_{\setu \subset \N}$ be a sequence of positive weights with $\gamma_\emptyset = 1$ which satisfies
	\begin{equation} \label{eq:summability-gamma-tilde}
		\sum_{j \ge 1} \tilde\gamma_j
		<
		\infty
		,
		\quad\text{where}\quad
		\tilde\gamma_j := \max_{\setv\subseteq\{1 \mcol j-1\}} \frac{\gamma_{\setv\cup\{j\}}}{\gamma_\setv}
		.
	\end{equation}
	Then, for $a > 0$ and any $\delta > 0$, there exists a constant $C_\delta > 0$ such that for all $N \ge 1$ we have
	\begin{equation*}
		\sum_{1 \le |\setu| < \infty} \gamma_\setu \, (a \ln N)^{|\setu|}
		<
		C_\delta \, N^\delta
		.
	\end{equation*}
\end{lemma}
\begin{proof}
	We start with a finite sum over all subsets of $\{1\mcol s\}$ (including the empty set for ease of manipulation of the expressions). At first, we prove by an inductive argument that
	\begin{equation} \label{eq:prod_induction_lnN}
		\sum_{\setu \subseteq \{1 \mcol s\}} \gamma_\setu \, (a \ln N)^{|\setu|}
		\le
		\prod_{j=1}^s \left(1 + a \max_{\setv\subseteq\{1 \mcol j-1\}} \frac{\gamma_{\setv\cup\{j\}}}{\gamma_\setv} \, \ln N\right)
		.
	\end{equation}
	
	For the base case $s=1$, we obtain that
	\begin{equation*}
		\sum_{\setu \subseteq \{1\}} \gamma_\setu \, (a \ln N)^{|\setu|}
		=
		\gamma_{\emptyset} + a \, \gamma_{\{1\}} \ln N
		=
		1 + a \max_{\setv = \emptyset} \frac{\gamma_{\setv\cup\{1\}}}{\gamma_\setv} \, \ln N
	\end{equation*}
	and so inequality \eqref{eq:prod_induction_lnN} holds for $s=1$. 
	
	Consider then $s\ge2$ and assume the estimate in \eqref{eq:prod_induction_lnN} holds for $s-1$, i.e.,
	\begin{equation*}
		\sum_{\setu \subseteq \{1\mcol s-1\}} \gamma_\setu \, (a \ln N)^{|\setu|}
		\le
		\prod_{j=1}^{s-1} \left(1 + a \max_{\setv\subseteq\{1 \mcol j-1\}} \frac{\gamma_{\setv\cup\{j\}}}{\gamma_\setv} \, \ln N\right)
		.
	\end{equation*}
	We can rewrite the term on the left-hand side of \eqref{eq:prod_induction_lnN} as
	\begin{align*}
		\sum_{\setu \subseteq \{1 \mcol s\}} \gamma_\setu \, (a \ln N)^{|\setu|}
		&=
		\sum_{s \notin \setu \subseteq \{1 \mcol s\}} \gamma_\setu \, (a \ln N)^{|\setu|}
		+
		\sum_{s \in \setu \subseteq \{1 \mcol s\}} \gamma_\setu \, (a \ln N)^{|\setu|}
		\\
		&=
		\sum_{\setu \subseteq \{1 \mcol s-1\}} \gamma_\setu \, (a \ln N)^{|\setu|}
		+
		\sum_{\setu \subseteq \{1 \mcol s-1\}} \gamma_{\setu \cup \{s\}} \, (a \ln N)^{|\setu|+1}
		\\
		&=
		\sum_{\setu \subseteq \{1 \mcol s-1\}} \gamma_\setu \, (a \ln N)^{|\setu|} \left(1 + a \, \frac{\gamma_{\setu \cup \{s\}}}{\gamma_\setu} \, \ln N\right)
		\\
		&\le
		\left(1 + a \max_{\setv\subseteq\{1 \mcol s-1\}} \frac{\gamma_{\setv\cup\{s\}}}{\gamma_\setv} \, \ln N\right) \sum_{\setu \subseteq \{1 \mcol s-1\}} \gamma_\setu \, (a \ln N)^{|\setu|}
		\\
		&\le
		\prod_{j=1}^s \left(1 + a \max_{\setv\subseteq\{1 \mcol j-1\}} \frac{\gamma_{\setv\cup\{j\}}}{\gamma_\setv} \, \ln N\right)
		,
	\end{align*}
	where we used the induction hypothesis in the last step to obtain the claimed result. Therefore, we get
	\begin{equation*}
		\sum_{\setu \subseteq \{1 \mcol s\}} \gamma_\setu \, (a \ln N)^{|\setu|}
		\le
		\prod_{j=1}^s \left(1 + a \, \tilde\gamma_j \, \ln N\right)
		\le
		\prod_{j=1}^\infty \left(1 + a \, \tilde\gamma_j \, \ln N\right)
		,
	\end{equation*}
	such that an application of~\cite[Lemma~3]{HN03} yields the result.
\end{proof}

We note that Condition \eqref{eq:summability-gamma-tilde} in Lemma \ref{lem:logN} can be interpreted as stating the influence of the component~$j$ with respect to all previous subsets of components. It is natural that this quantity has to decay if we want $s$ to go to infinity. In particular, in the case of product weights $\gamma_\setu = \prod_{j\in\setu} \gamma_j$, the condition in Lemma \ref{lem:logN} recovers the standard condition for product weights $\sum_{j \ge 1} \gamma_j < \infty$, and for product-and-order-dependent (POD) weights of the form $\gamma_\setu = |\setu|! \, \prod_{j\in\setu} \gamma_j$, as they appear for PDE applications, the condition becomes $\sum_{j \ge 1} j \, \gamma_j < \infty$.

\subsection{The component-by-component digit-by-digit construction} \label{sec:dbd}

The base ingredient of the algorithms in this paper is to construct generating vectors which are optimal coefficients as in Definition~\ref{def:opt_coeff}.
That is, given $N$ we want to find $\bsz(N)$ such that $T(N,\bsz) \le C(\bsgamma,\delta) \, N^{-1+\delta}$ holds for all $\delta > 0$.

In this section, we study the first of two algorithms, namely a \textit{component-by-component digit-by-digit} (CBC-DBD) algorithm. To this end, we assume throughout the section
that $N$ is of the form $N=2^n$ for some positive integer $n$; this choice of $N$ is natural as the components of the generating vector will be constructed digit by digit, i.e., bit by bit 
as we consider $N$ to be a power of~$2$. We first show the following estimate on the quantity $T(N,\bsz)$, which already indicates the target function to be minimized in the CBC-DBD algorithm below.

\begin{theorem} \label{thm:T_target_CBCDBD}
Let $N=2^n$, with $n\ge1$, and let $\bsgamma=\{\gamma_\setu\}_{\setu \subseteq \{1 \mcol s\}}$ be positive weights. Furthermore, let $\bsz = (z_1, \ldots, z_s)\in\{1,\ldots,N-1\}^s$
with $\gcd (z_j,N)=1$ for $1\le j\le s$. Then the following estimate holds:
\begin{align*}
	T(N,\bsz) 
	&\le
	 \sum_{\emptyset \ne \setu \subseteq \{1 \mcol s\}} \frac{\gamma_\setu}{N} (\ln 4 + 2(1 + \ln N))^{\abs{\setu}} 
	-
	\sum_{\emptyset \ne \setu \subseteq \{1 \mcol s\}} \gamma_\setu \, (\ln 4)^{|\setu|} 
	\\
	&\quad+
	\sum_{\emptyset \ne \setu \subseteq \{1 \mcol s\}} \frac{\gamma_\setu}{N} \, 2 |\setu| \, \left(1 + 2 \ln N\right)^{|\setu|} \, (1+\ln N) + \frac{1}{N}\, H_{s,n,\bsgamma} (\bsz)
	,
\end{align*}
with $b(m)$ defined as in~\eqref{eq:def-b-dnint}, and where 
\begin{align*}
	H_{s,n,\bsgamma}(\bsz)
	&:=
	\sum_{k=1}^{2^n-1}
	\sum_{\emptyset \ne \setu \subseteq \{1 \mcol s\}} \gamma_\setu
	\prod_{j\in\setu} \ln \left( \frac1{\sin^{2}(\pi k z_j / 2^n)} \right)
	.
\end{align*}
\end{theorem}

\begin{proof}
We rewrite $T(N, \bsz)$ such that products of function values of $\vartheta_N$, defined in~\eqref{eq:vartheta}, are replaced 
by products of values of $\ln(\sin^{-2}(\pi\, \cdot))$.
Note that we need to take care of not evaluating $\ln(\sin^{-2}(\pi x))$ in $x=0$ or $x=1$. Recalling the definition of $b(m)$ in \eqref{eq:def-b-dnint}, we can estimate
\begin{align}
	&\hspace{-5mm}T(N, \bsz) \notag
	=
	\sum_{\emptyset \ne \setu \subseteq \{1 \mcol s\}} \gamma_\setu \sum_{\bsm_\setu \in M_{N,|\setu|}^\ast} \frac{\delta_N(\bsz_\setu\cdot\bsm_\setu)}{\prod_{j\in\setu} |m_j|}
	\notag
	=
	\sum_{\emptyset \ne \setu \subseteq \{1 \mcol s\}} \gamma_\setu \sum_{\bsm_\setu \in M_{N,|\setu|}^\ast} \frac{\delta_N(\bsz_\setu\cdot\bsm_\setu)}{\prod_{j\in\setu} b(m_j)}
	\\ \label{eq:increase-at-0}
	&\le
	\sum_{\emptyset \ne \setu \subseteq \{1 \mcol s\}} \!\!\! \gamma_\setu \!\! \sum_{\bszero \ne \bsm_\setu \in M_{N,|\setu|}} \!\! \frac{\delta_N(\bsz_\setu\cdot\bsm_\setu)}{\prod_{j\in\setu} b(m_j)}
	= \!\!
	\sum_{\emptyset \ne \setu \subseteq \{1 \mcol s\}} \! \frac{\gamma_\setu}{N} \sum_{k=0}^{N-1} 
	\left[ \sum_{\bsm_\setu \in M_{N,|\setu|}} \!\! \frac{\rme^{2\pi\icomp k \,\bsz_\setu\cdot\bsm_\setu/N}}{\prod_{j\in\setu} b(m_j)} - (\ln 4)^{|\setu|} \right]
	\\\notag
	&=
	\sum_{\emptyset \ne \setu \subseteq \{1 \mcol s\}} \! \frac{\gamma_\setu}{N} \! \left[ \sum_{\bsm_\setu \in M_{N,|\setu|}} \! \frac{1}{\prod_{j\in\setu} b(m_j)}
	+
	\sum_{k=1}^{N-1} \prod_{j\in\setu} \left(\ln 4 + \!\!\! \sum_{m \in M_{N,1}^\ast} \!\!\! \frac{\rme^{2\pi\icomp k m z_j / N}}{|m|} \right) \right]
	\notag
	\! - \!\!
	\sum_{\emptyset \ne \setu \subseteq \{1 \mcol s\}} \!\!\!\! \gamma_\setu \, (\ln 4)^{|\setu|}
	\\
	\label{eq:apply-prod-diff}
	&=
	\sum_{\emptyset \ne \setu \subseteq \{1 \mcol s\}} \frac{\gamma_\setu}{N} \left( \sum_{\bsm_\setu \in M_{N,|\setu|}} \frac{1}{\prod_{j\in\setu} b(m_j)}
	+
	\sum_{k=1}^{N-1} \left[ \prod_{j\in\setu} v_j(k) - \prod_{j\in\setu} u_j(k) + \prod_{j\in\setu} u_j(k) \right]
	\right)
	\\ \notag
	&\quad -
	\sum_{\emptyset \ne \setu \subseteq \{1 \mcol s\}} \gamma_\setu \, (\ln 4)^{|\setu|}
	\\
	\notag
	&=
	\sum_{\emptyset \ne \setu \subseteq \{1 \mcol s\}}
	\frac{\gamma_\setu}{N} \sum_{\bsm_\setu \in M_{N,|\setu|}} \frac{1}{\prod_{j\in\setu} b(m_j)}
	+
	\sum_{\emptyset \ne \setu \subseteq \{1 \mcol s\}} \frac{\gamma_\setu}{N} \sum_{k=1}^{N-1} \prod_{j\in\setu} u_j(k)
	\\
	\notag
	&\quad+
	\sum_{\emptyset \ne \setu \subseteq \{1 \mcol s\}} \frac{\gamma_\setu}{N} \sum_{k=1}^{N-1}
	\theta_\setu(k) \left(\prod_{j\in\setu} \left(\bar{u}_j(k) + |r_j(k)|\right)\right) \sum_{j\in\setu} |r_j(k)|
	-
	\sum_{\emptyset \ne \setu \subseteq \{1 \mcol s\}} \gamma_\setu \, (\ln 4)^{|\setu|}
	\\
	\label{eq:T-bound-ln-sin}
	&=
	\sum_{\emptyset \ne \setu \subseteq \{1 \mcol s\}}
	\frac{\gamma_\setu}{N}
	\sum_{\bsm_\setu \in M_{N,|\setu|}} \frac{1}{\prod_{j\in\setu} b(m_j)}
	-
	\sum_{\emptyset \ne \setu \subseteq \{1 \mcol s\}} \gamma_\setu \, (\ln 4)^{|\setu|}
	\\
	\notag
	&\quad+
	\sum_{\emptyset \ne \setu \subseteq \{1 \mcol s\}} \frac{\gamma_\setu}{N} \sum_{k=1}^{N-1}
	\theta_\setu(k) \left(\prod_{j\in\setu} \left(\bar{u}_j(k) + |r_j(k)|\right)\right) \sum_{j\in\setu} |r_j(k)|
	+
	\frac{1}{N}\, H_{s,n,\bsgamma} (\bsz)
	,
\end{align}
where in~\eqref{eq:increase-at-0} we added terms with $m_j=0$ and in~\eqref{eq:apply-prod-diff} we used Lemma~\ref{lemma:diff_prod} with
\begin{align*}
	u_j=u_j(k) &:= \ln \left( \frac1{\sin^{2}(\pi z_j k / N)} \right) ,
	&
	\bar{u}_j=\bar{u}_j(k) &:= 2 \ln N,
	\\
	v_j=v_j(k) &:= \ln 4 + \sum_{m \in M_{N,1}^\ast} \frac{\rme^{2\pi\icomp k m z_j / N}}{|m|},
	&
	r_j=r_j(k) &:= \frac{\tau_j(k)}{N \, \|z_j k / N\|}
	,
\end{align*}
and all $|\theta_\setu(k)| \le 1$ and $|\tau_j(k)| \le 1$. Due to Lemma \ref{lemma:expr_target_function}, Condition (a) in Lemma \ref{lemma:diff_prod} is fulfilled. 
Furthermore, we have for $1 \le k \le N-1$ that 
\begin{equation*}
	\sin^2 \left(\pi \frac{z_j k}{N} \right)
	\ge
	\sin^2 \left( \frac{\pi}{N} \right)
	=
	\sin^2 \left(2 \frac{\pi}{2N} \right)
	\ge
	\left(\frac{\pi}{2N} \right)^2
	\ge
	\left(\frac1N \right)^2 ,
\end{equation*}
where we used that for $x \in [0,\tfrac{\pi}{4}]$ the estimate $\sin(2x) \ge x$ holds. This implies that Conditions (b) and (c) in Lemma \ref{lemma:diff_prod} are fulfilled since 
\begin{equation*}
	\abs{u_j}
	=
	\ln \left(\frac{1}{\sin^{2}\left(\pi z_j k / N\right)} \right)
	\le
	\ln(N^2)
	=
	2 \ln N
	=
	\bar{u}_j ,
\end{equation*}
and the latter expression is not smaller than one as long as $N\ge 2$. Next, we show how to bound the last sum in~\eqref{eq:T-bound-ln-sin} independently 
of the choice of $\bsz$, which can be done as follows:
\begin{align*}
	&\sum_{\emptyset \ne \setu \subseteq \{1 \mcol s\}} \frac{\gamma_\setu}{N} \sum_{k=1}^{N-1}
	\theta_\setu(k) \left(\prod_{j\in\setu} \left(\bar{u}_j(k) + |r_j(k)|\right)\right) \sum_{j\in\setu} |r_j(k)|
	\\
	&\qquad=
	\sum_{\emptyset \ne \setu \subseteq \{1 \mcol s\}} \frac{\gamma_\setu}{N} \sum_{k=1}^{N-1}
	\theta_\setu(k) \left(\prod_{j\in\setu} \left(2 \ln N + \frac{|\tau_j(k)|}{N \, \|z_j k / N\|}\right)\right) \sum_{j\in\setu} \frac{|\tau_j(k)|}{N \, \|z_j k / N\|}
	\\
	&\qquad\le
	\sum_{\emptyset \ne \setu \subseteq \{1 \mcol s\}} \frac{\gamma_\setu}{N} \sum_{k=1}^{N-1}
	\theta_\setu(k) \left(\prod_{j\in\setu} \left(1 + 2 \ln N \right)\right) \sum_{j\in\setu} \frac{|\tau_j(k)|}{N \, \|z_j k / N\|}
	\\
	&\qquad=
	\sum_{\emptyset \ne \setu \subseteq \{1 \mcol s\}} \frac{\gamma_\setu}{N}
	\left(\prod_{j\in\setu} \left(1 + 2 \ln N \right)\right) \sum_{j\in\setu} \sum_{k=1}^{N-1} \frac{\theta_\setu(k) |\tau_j(k)|}{N \, \|z_j k / N\|}
	\\
	&\qquad\le
	\sum_{\emptyset \ne \setu \subseteq \{1 \mcol s\}} \frac{\gamma_\setu}{N}
	\left(\prod_{j\in\setu} \left(1 + 2 \ln N \right)\right) \sum_{j\in\setu} \sum_{k=1}^{N-1} \frac{1}{N \, \|z_j k / N\|}
	\\
	&\qquad\le
	\sum_{\emptyset \ne \setu \subseteq \{1 \mcol s\}} \frac{\gamma_\setu}{N}
	\left(\prod_{j\in\setu} \left(1 + 2 \ln N \right)\right) \sum_{j\in\setu} 2 (1+\ln N)
	\\
	&\qquad\le
	\sum_{\emptyset \ne \setu \subseteq \{1 \mcol s\}} \frac{\gamma_\setu}{N}
	 \, 2 |\setu| \, \left(1 + 2 \ln N \right)^{|\setu|} \, (1+\ln N)
	,
\end{align*}
where we used that $N \norm{\frac{z_j k}{N}} \ge N \norm{\frac{1}{N}} = 1$ and the fact that if $\gcd(z_j,N)=1$, then
\begin{equation*}
	\sum_{k=1}^{N-1} \frac{1}{N \, \|z_j k / N\|}
	\le
	2(1 + \ln N)
	,
\end{equation*}
see \cite[Corollary of Proposition~4]{Kor63}. Finally, using \eqref{eq:estimate_ln}, we obtain the estimate
\begin{equation*}
	\sum_{\bsm_\setu \in M_{N,|\setu|}^\ast} \frac{1}{\prod_{j\in\setu} b(m_j)}
	=
	\prod_{j \in \setu} \left(\ln 4 + 2 \sum_{m=1}^{N-1} \frac{1}{m} \right)
	\le
	(\ln 4 + 2 (1 + \ln N))^{\abs{\setu}}
	,
\end{equation*}
such that using the last two estimates for the third and the first term in \eqref{eq:T-bound-ln-sin}, respectively, yields the claim.

\end{proof}

Theorem \ref{thm:T_target_CBCDBD} implies that if we can find a generating vector $\bsz$ of a lattice rule with $N=2^n$ points for which a good bound 
on $H_{s,n,\bsgamma}(\bsz)$ holds, then for this generating vector we also have a good bound on $T(N,\bsz)$. We remark that here we use the function 
$-2 \ln( \sin(\pi x)) = \ln( \sin^{-2}(\pi x))$, opposed to the function $-2 \ln(2 \sin(\pi x))$ which is used in Section \ref{sec:cbc}. For this choice, certain 
steps in the proof of the result in Theorem \ref{theorem:H-induction} simplify, additionally we can exploit the non-negativity of the function in estimates.

The algorithm which we are considering here is a component-by-component digit-by-digit construction. The CBC-DBD algorithm introduced
below does not directly optimize the function $H_{s,n,\bsgamma}(\bsz)$ in terms of $\bsz$, but, digit-by-digit, the relevant part of this quantity.
To this end, we prove the following lemma which allows us to rewrite the average of $\ln\left( \sin^{-2} (\pi x) \right)$, with the 
base-$2$ digits of $x$ depending on $z \in \{0,1\}$, where we consider the average with respect to the choice of $z$.

\begin{lemma} \label{lemma:average}
	Let $a$ and $k$ be odd integers and $v \ge 2$. Then it holds true that
	\begin{equation*}
	\sum_{z \in \{0,1\}} \ln \left(\frac{1}{\sin^2(\pi k (a + 2^{v-1} z) / 2^v)}\right)
	=
	\ln 4 +  \ln \left( \frac{1}{\sin^2(\pi k a / 2^{v-1})} \right)
	.
	\end{equation*}
\end{lemma}

\begin{proof}
	Since $k$ is odd, we have that 
	\begin{equation*}
		\sin^2 \left( \frac{\pi k (a + 2^{v-1})}{2^v} \right) 
		= 
		\sin^2 \left( \frac{\pi k a}{2^v} + \frac{\pi (k-1)}{2} + \frac{\pi}{2} \right) 
		= 
		\cos^2 \left( \frac{\pi k a}{2^v} \right)
	\end{equation*}
	and therefore, using that $\sin(x) \cos(x) = \sin(2x)/2$, we obtain
	\begin{align*}
		&\sum_{z \in \{0,1\}} \ln \left( \frac{1}{\sin^2(\pi k (a + 2^{v-1} z) / 2^v)} \right)
		=
		\ln \left( \frac{1}{\sin^2(\pi k a / 2^v)} \right) + \ln \left( \frac{1}{\cos^2(\pi k a / 2^v)} \right)
		\\
		&\quad=
		\ln \left( \frac{1}{\sin^2(\pi k a / 2^v) \cos^2(\pi k a / 2^v)} \right)
		=
		\ln \left( \frac{4}{\sin^2(\pi k a / 2^{v-1})} \right)
		=
		\ln 4 + \ln \left( \frac{1}{\sin^2(\pi k a / 2^{v-1})} \right)
	\end{align*}
	which yields the claimed equality.
\end{proof}

Furthermore, we note that the function $\sin^2(\pi x)$ is periodic with period $1$ such that 
\begin{equation} \label{eq:prop-ln-sin}
	\sin^2\left( \pi \frac{k z}{2^t} \right)
	=
	\sin^2\left( \pi \frac{k z \bmod 2^t}{2^t} \right)
\end{equation}
for integers $k,z$ and $t \ge 1$.

The next lemma motivates the choice of our quality function and, furthermore, shows that we can decide on the bits of the components of our generating vector $\bsz$ 
in a digit-by-digit fashion. In particular, assuming we have already fixed the first $v-1$ bits of $z_s$ and want to calculate how good a specific choice for the $v$-th bit is,
in terms of  $H_{s,n,\bsgamma}$, we can average over all remaining $n-v$ bits, and still claim that this is at least as good as the average over all bits.

\begin{lemma} \label{lemma:bit-average-H}
	For an integer $v \in \{1,\ldots,n\}$, with $n \in \N$, let $z \in \{0,1\}$ and $\bsz \in \Z^s$, with all components $z_j$ odd, and where the first $v-1$ bits of $z_s$ have been selected, that is, $z_s \in \Z_{2^{v-1}}$. Then the average of $H_{s,n,\bsgamma}$ over all $n-v$ remaining bit choices for $z_s$ is given by
	\begin{align} \label{eq:average-zs}
		&\frac1{2^{n-v}} \sum_{\bar{z} \in \Z_{2^{n-v}}} H_{s,n,\bsgamma}\left(z_1,\ldots,z_{s-1}, z_s + z \, 2 ^{v-1} + \bar{z} \, 2^v\right) \notag \\
		&= \!
		\sum_{t=v}^n \frac1{2^{t-v}} \!\!\!\!\!\!\!\! \sum_{\substack{k=1\\k\equiv1\tpmod{2}}}^{2^t-1} \!
		\sum_{\setu \subseteq \{1 \mcol s-1\}} \!\!\!\!\!\gamma_{\setu\cup\{s\}}
		\ln \!\left( \frac1{\sin^{2}(\pi k (z_s + z \, 2 ^{v-1}) / 2^v)} \right) \!\prod_{j\in\setu} \ln\!\left( \frac1{\sin^{2}(\pi k z_j / 2^t)} \right) \!+ S_{n,v,\bsgamma}(\bsz)
		,
	\end{align}
	where the term $S_{n,v,\bsgamma}(\bsz)$, which is independent of $z$ and $\bar{z}$, is given by
	\begin{align} \label{eq:S_n_v_gamma}
		S_{n,v,\bsgamma}(\bsz)
		&= \notag
		\sum_{t=1}^{v-1} \sum_{\substack{k=1\\k\equiv1\tpmod{2}}}^{2^t-1}
		\sum_{\emptyset \ne \setu \subseteq \{1 \mcol s\}} \gamma_{\setu} \prod_{j\in\setu} \ln \left(\frac1{\sin^{2}(\pi k z_j / 2^t)} \right) \notag
		\\
		&\quad\!\!+
		\sum_{t=v}^n \sum_{\substack{k=1\\k\equiv1\tpmod{2}}}^{2^t-1}
		\sum_{\emptyset \ne \setu \subseteq \{1 \mcol s-1\}} \gamma_\setu \prod_{j\in\setu} \ln \left(\frac1{\sin^{2}(\pi k z_j / 2^t)}\right)
		\\
		&\quad\!\!+ \notag 
		\sum_{t=v}^n \frac{(2^{t-v} - 1) \ln 4}{2^{t-v}} \!\!\!\!\!\!\! \sum_{\substack{k=1\\k\equiv1\tpmod{2}}}^{2^t-1} \!
		\sum_{\setu \subseteq \{1 \mcol s-1\}} \!\!\!\!\! \gamma_{\setu\cup\{s\}} \prod_{j\in\setu} \ln \left( \frac1{\sin^{2}(\pi k z_j / 2^t)}\right)
		.
	\end{align}
\end{lemma}

\begin{proof}
	Firstly, note that due to the general identity
	\begin{equation*}
		\sum_{k=1}^{2^p-1} f(k/2^p)
		=
		\sum_{t=1}^p \sum_{\substack{k=1\\k\equiv1\tpmod{2}}}^{2^t-1} f(k/2^t)
		,
	\end{equation*}
	for $f: \R \to \R$ and $p \in \N$, the quantity $H_{s,n,\bsgamma}$ can be rewritten as
	\begin{equation} \label{eq:rewrite-H}
		H_{s,n,\bsgamma}(\bsz)
		=
		\sum_{t=1}^n \sum_{\substack{k=1\\k\equiv1\tpmod{2}}}^{2^t-1} 
		\sum_{\emptyset \ne \setu \subseteq \{1 \mcol s\}} \gamma_\setu \prod_{j\in\setu} \ln \left(\frac1{\sin^{2}(\pi k z_j / 2^t)}\right)
		.
	\end{equation}
	The first summand in the definition of $S_{n,v,\bsgamma}(\bsz)$ in \eqref{eq:S_n_v_gamma} originates from relation \eqref{eq:prop-ln-sin} since 
	$z_s + z \, 2 ^{v-1} + \bar{z} \, 2^v \bmod 2^t = z_s$ for all $t=1,\ldots,v-1$. The remaining terms in \eqref{eq:average-zs} are obtained by repeated use of the 
	one bit averaging result in Lemma~\ref{lemma:average}. Writing $\tilde{z}_j = z_j$ for $j \in \{1 \mcol s-1\}$ and $\tilde{z}_s = z_s + z \, 2 ^{v-1} + \bar{z} \, 2^v$, we obtain
	\begin{align*}
		&\frac1{2^{n-v}} \sum_{\bar{z} \in \Z_{2^{n-v}}} \sum_{t=v}^n \sum_{\substack{k=1\\k\equiv1\tpmod{2}}}^{2^t-1}
		\sum_{\setu \subseteq \{1 \mcol s\}} \gamma_\setu \prod_{j\in\setu} \ln \left(\frac1{\sin^{2}(\pi k \tilde{z}_j / 2^t)}\right) \\
		&\quad=
		\sum_{t=v}^n \sum_{\substack{k=1\\k\equiv1\tpmod{2}}}^{2^t-1}
		\sum_{\emptyset \ne \setu \subseteq \{1 \mcol s-1\}} \gamma_\setu \prod_{j\in\setu} \ln \left(\frac1{\sin^{2}(\pi k z_j / 2^t)}\right) 
		\\
		&\quad+ 
		\frac1{2^{n-v}} \sum_{\bar{z} \in \Z_{2^{n-v}}} \sum_{t=v}^n \!\!\!\! \sum_{\substack{k=1\\k\equiv1\tpmod{2}}}^{2^t-1} \!\!
		\sum_{\setu \subseteq \{1 \mcol s-1\}} \!\!\!\!\!\! \gamma_{\setu \cup \{s\}} \!\!
		\left[\prod_{j\in\setu} \ln \!\left(\frac1{\sin^{2}(\pi k z_j / 2^t)} \right) \right] \! \ln \!\left( \frac1{\sin^{2}(\pi k \tilde{z}_s / 2^t)} \right)
		.
	\end{align*}
	For the last expression, we can pull the average over $\bar{z} \in \Z_{2^{n-v}}$ inside and then obtain, for $t \in \{v+1,\ldots,n\}$, 
	by \eqref{eq:prop-ln-sin} and repeated use of Lemma~\ref{lemma:average}, 
	\begin{align*}
		&\frac1{2^{n-v}} \sum_{\bar{z} \in \Z_{2^{n-v}}} \ln \left(\frac1{\sin^{2}(\pi k (z_s + z \, 2 ^{v-1} + \bar{z} \, 2^v) / 2^t)} \right) \\
		&\quad=
		\frac{2^{n-t}}{2^{n-v}} \sum_{\bar{z} \in \Z_{2^{t-v}}} \ln \left( \frac1{\sin^{2}(\pi k (z_s + z \, 2 ^{v-1} + \bar{z} \, 2^v) / 2^t)} \right) \\
		&\quad=
		2^{v-t} \sum_{\bar{z} \in \Z_{2^{t-1-v}}} \left[ \ln 4 + \ln \left(\frac1{\sin^{2}(\pi k (z_s + z \, 2 ^{v-1} + \bar{z} \, 2^v) / 2^{t-1})} \right) \right] \\
		&\quad=
		2^{-1} \ln 4 + 2^{v-t} \sum_{\bar{z} \in \Z_{2^{t-1-v}}} \ln \left( \frac1{\sin^{2}(\pi k (z_s + z \, 2 ^{v-1} + \bar{z} \, 2^v) / 2^{t-1})} \right) \\
		&\quad=
		\ln 4 \left( \sum_{\ell=1}^{t-v} 2^{-\ell} \right) + 2^{v-t} \ln \left( \frac1{\sin^{2}(\pi k (z_s + z \, 2 ^{v-1}) / 2^v)} \right) \\
		&\quad=
		(1 - 2^{v-t}) \ln 4 + 2^{v-t} \ln \left( \frac1{\sin^{2}(\pi k (z_s + z \, 2 ^{v-1}) / 2^v)} \right)
		,
	\end{align*}
	which together with the previous identity yields the claim.
\end{proof}

We note that in Lemma \ref{lemma:bit-average-H} only the first term in \eqref{eq:average-zs} depends on the $v$-th bit $z$ of $z_s$, while $S_{n,v,\bsgamma}(\bsz)$ is independent of it. 
This leads us to introducing the following digit-wise quality function which is based on the first term in \eqref{eq:average-zs} and additionally the second term in \eqref{eq:S_n_v_gamma}, which is independent of $z_s$, modified by an additional factor. We include the second term in \eqref{eq:S_n_v_gamma} in order to obtain a more easily computable quality function for the implementation of the algorithm in Section \ref{sec:fast_impl}.

\begin{definition}[Digit-wise quality function] \label{def:h_rv}
	Let $x \in \N$ be an odd integer, $n, s \in \N$ be positive integers, and let $\bsgamma=\{\gamma_\setu\}_{\setu\subseteq \{1 \mcol s\}}$
	be a sequence of positive weights. For $1 \le v \le n$ and $1 \le r \le s$ and odd integers $z_1,\ldots,z_{r-1} \in \Z$ we define the 
	quality function $h_{r,n,v,\bsgamma}: \Z \to \R$ as
	\begin{align*} 
		h_{r,n,v,\bsgamma}(x) 
		&:=
		\sum_{t=v}^{n} \frac{1}{2^{t-v}} \sum_{\substack{k=1 \\ k \equiv 1 \tpmod{2}}}^{2^t-1} 
		\left[ \sum_{\emptyset \ne \setu\subseteq \{1 \mcol r-1\}} \gamma_{\setu} 
		\prod_{j\in\setu} \ln \left( \frac{1}{\sin^2(\pi k z_j / 2^t)} \right) \right.\nonumber\\ 
		&\qquad\qquad\qquad\qquad\qquad+
		\left.\sum_{\setu\subseteq \{1 \mcol r-1\}} \!\!\!\! \gamma_{\setu\cup \{r\}} 
		\ln \left( \frac{1}{\sin^2(\pi k x / 2^v)} \right) \prod_{j\in\setu} \ln \left( \frac{1}{\sin^2(\pi k z_j / 2^t)}\right) \right]
		.
	\end{align*} 
\end{definition}

Note that while the quantity $h_{r,n,v,\bsgamma}$ depends on the integers $z_1,\ldots,z_{r-1}$, this dependency is not explicitly visible in our notation. Nevertheless, in the following these integers will always be the components of the generating vector which were selected in the previous steps of our algorithm. Based on $h_{r,n,v,\bsgamma}$ the component-by-component digit-by-digit algorithm is formulated as follows.

\begin{algorithm}[H] 
	\caption{Component-by-component digit-by-digit construction}	
	\label{alg:cbc-dbd}
	\vspace{5pt}
	\textbf{Input:} Integer $n \in \N$, dimension $s$ and positive weights $\bsgamma=\{\gamma_\setu\}_{\setu\subseteq \{1 \mcol s\}}$. \\
	\vspace{-7pt}
		\begin{algorithmic}
			\STATE Set $z_{1,n} = 1$ and $z_{2,1} = \cdots = z_{s,1} = 1$.
			\vspace{5pt}
			\FOR{$r=2$ \TO $s$}
			\FOR{$v=2$ \TO $n$}
			\STATE $z^{\ast} = \underset{z \in \{0,1\}}{\argmin} \; h_{r,n,v,\bsgamma}(z_{r,v-1} + z \, 2^{v-1})$
			\STATE $z_{r,v} = z_{r,v-1} + z^{\ast} \, 2^{v-1}$
			\ENDFOR
			\ENDFOR
			\vspace{5pt}
			\STATE Set $\bsz = (z_1,\ldots,z_s)$ with $z_r := z_{r,n}$ for $r=1,\ldots,s$.
		\end{algorithmic}
	\vspace{5pt}
	\textbf{Return:} Generating vector $\bsz = (z_1,\ldots,z_s)$ for $N=2^n$.
\end{algorithm}
\noindent
The resulting integer vector $\bsz = (z_1,\ldots,z_s) \in \Z^s$ with $z_j < 2^n$ for $j=1,\ldots,s$ is then used as the generating vector of a lattice rule
with $N=2^n$ points in $s$ (or fewer) dimensions.

\subsubsection{Error convergence behavior of the constructed lattice rules}

In the following, we show that, under certain conditions on the weights $\bsgamma$, Algorithm \ref{alg:cbc-dbd} constructs generating vectors which are optimal coefficients modulo $N$ according to Definition \ref{def:opt_coeff}. As indicated by Theorem \ref{thm:T_target_CBCDBD}, we need to show that the quantity $H_{s,n,\bsgamma}(\bsz)$ is sufficiently small for generating vectors constructed by Algorithm \ref{alg:cbc-dbd}.

\begin{theorem} \label{theorem:H-induction}
	Let $n, s \in \N$, $N=2^n$, and let $\bsgamma=\{\gamma_\setu\}_{\setu\subseteq \{1 \mcol s\}}$ be a sequence of positive weights with $\gamma_\emptyset=1$. 
	Furthermore, let the generating vector $\bsz \in \Z_N^s$ be constructed by \RefAlg{alg:cbc-dbd}. Then the following estimate holds
	\begin{equation} \label{eq:estimate-H-induction}
	H_{s,n,\bsgamma}(\bsz) 
	\le 
	H_{s-1,n,\bsgamma}(\bsz_{\{1 \mcol s-1\}}) + \ln 4 \left[ \gamma_{\{s\}} N + H_{s-1,n,\bsgamma \cup \{s\}}(\bsz_{\{1 \mcol s-1\}}) \right]
	\end{equation}
	with weight sequence $\bsgamma \cup \{s\} = \{\gamma_{\setu \cup \{s\}}\}_{\setu \subseteq \{1 \mcol s-1\}}$.
\end{theorem}

\begin{proof}
	We will prove the stated estimate via an inductive argument over the selection of the $n-1$ bits of the component $z_s$. We start with the most significant bit of $z_s$, i.e., the $n$-th bit. According to \RefAlg{alg:cbc-dbd}, this bit was selected by minimizing $h_{s,n,n,\bsgamma}(z_{s,n-1} + z \, 2^{n-1})$ with respect to the choices $z \in \{0,1\}$, where $z_{s,n-1}$ has been determined in the previous steps of the algorithm. By Lemma \ref{lemma:bit-average-H} (with $v=n$) and Definition \ref{def:h_rv}, this is equivalent to minimizing 
	\begin{equation*}
		H_{s,n,\bsgamma}(z_1,\ldots,z_{s-1}, z_{s,n-1} + z \, 2 ^{n-1})
	\end{equation*}
	with respect to $z \in \{0,1\}$. By the standard averaging argument, this yields
	\begin{align} \label{eq:target_quantity}
		H_{s,n,\bsgamma}(\bsz)
		&=
		\argmin_{\bar{z} \in \{0,1\}} H_{s,n,\bsgamma}(\bsz_{\{1 \mcol s-1\}},z_{s,n-1} + \bar{z} \, 2^{n-1}) \notag
		\\
		&\le
		\frac1{2} \sum_{\bar{z} \in \Z_{2}} H_{s,n,\bsgamma}\left(\bsz_{\{1 \mcol s-1\}}, z_{s,n-1} + \bar{z} \, 2^{n-1} \right) \notag
		\\
		&=
		\frac1{2} \sum_{\bar{z} \in \Z_{2}} H_{s,n,\bsgamma}\left(\bsz_{\{1 \mcol s-1\}}, z_{s,n-2} + z \, 2^{n-2} + \bar{z} \, 2^{n-1} \right)
		,
	\end{align}
	where $z_{s,n-1}$ has been split up into $z_{s,n-2}$ and $z \, 2^{n-2}$ in accordance with \RefAlg{alg:cbc-dbd} such that $z \in \{0,1\}$
	is the previously selected $(n-1)$-th bit of $z_s$.
	
	Similarly, we see that the $(n-1)$-th bit of $z_s$ has been selected in \RefAlg{alg:cbc-dbd} by minimizing $h_{s,n,n-1,\bsgamma}(z_{s,n-2} + z \, 2^{n-2})$ with respect to $z \in \{0,1\}$. Again by Lemma \ref{lemma:bit-average-H} (with $v=n-1$) and Definition \ref{def:h_rv}, this is equivalent to minimizing 
	\begin{equation*}
		\frac1{2} \sum_{\bar{z} \in \Z_{2}} H_{s,n,\bsgamma}\left(\bsz_{\{1 \mcol s-1\}}, z_{s,n-2} + z \, 2^{n-2} + \bar{z} \, 2^{n-1} \right)
	\end{equation*}
	with respect to $z \in \{0,1\}$. By the averaging argument, we obtain that
	\begin{align*}
		&\argmin_{z \in \{0,1\}} \frac1{2} \sum_{\bar{z} \in \Z_{2}} H_{s,n,\bsgamma} \left(\bsz_{\{1 \mcol s-1\}}, z_{s,n-2} + z \, 2^{n-2} + \bar{z} \, 2^{n-1} \right)
		\\
		&\quad\le
		\frac{1}{2^2} \sum_{z \in \Z_{2}} \sum_{\bar{z} \in \Z_{2}} H_{s,n,\bsgamma} (\bsz_{\{1 \mcol s-1\}}, z_{s,n-2} + z \, 2^{n-2} + \bar{z} \, 2^{n-1} )
		\\
		&\quad=
		\frac{1}{2^2} \sum_{\bar{z} \in \Z_{2^2}} H_{s,n,\bsgamma} (\bsz_{\{1 \mcol s-1\}}, z_{s,n-3} + \tilde{z} \, 2^{n-3} + \bar{z} \, 2^{n-2} )
		,
	\end{align*}
	where again we split up $z_{s,n-2}$ according to \RefAlg{alg:cbc-dbd} such that $\tilde{z}$ is the $(n-2)$-th bit of $z_s$, selected in the previous step of the algorithm. Inductively repeating this argument and combining the result with the estimate in \eqref{eq:target_quantity}, we obtain the inequality
	\begin{align} \label{eq:H-induction-bound}
		H_{s,n,\bsgamma}(\bsz)
		\le
		\frac1{2^{n-1}} \sum_{\bar{z} \in \Z_{2^{n-1}}} H_{s,n,\bsgamma}\left(z_1,\ldots,z_{s-1}, 1 + \bar{z} \, 2 \right)
		,
	\end{align}
	where we used that by \RefAlg{alg:cbc-dbd} we have $z_{s,1}=1$. Then, using Lemma \ref{lemma:bit-average-H} with $v=1$ to equate the right-hand side term in \eqref{eq:H-induction-bound}, we finally obtain
	\begin{align*}
		H_{s,n,\bsgamma}(\bsz) 
		&\le
		\sum_{t=1}^n \sum_{\substack{k=1\\k\equiv1\tpmod{2}}}^{2^t-1}
		\sum_{\emptyset \ne \setu \subseteq \{1 \mcol s-1\}} \gamma_\setu \prod_{j\in\setu} \ln \left(\frac1{\sin^{2}(\pi k z_j / 2^t)}\right)
		\\
		&\quad+
		\sum_{t=1}^n \frac1{2^{t-1}} \sum_{\substack{k=1\\k\equiv1\tpmod{2}}}^{2^t-1}
		\sum_{\setu \subseteq \{1 \mcol s-1\}} \gamma_{\setu\cup\{s\}}
		\ln\left( \frac1{\sin^{2}(\pi k / 2)} \right) \prod_{j\in\setu} \ln\left( \frac1{\sin^{2}(\pi k z_j / 2^t)} \right) \notag
		\\
		&\quad+
		\sum_{t=1}^n \frac{(2^{t-1} - 1) \ln 4}{2^{t-1}} \sum_{\substack{k=1\\k\equiv1\tpmod{2}}}^{2^t-1}
		\sum_{\setu \subseteq \{1 \mcol s-1\}} \gamma_{\setu\cup\{s\}} \prod_{j\in\setu} \ln \left( \frac1{\sin^{2}(\pi k z_j / 2^t)}\right) \notag
		.
	\end{align*}
	Noting that for odd $k$ we have $\ln (\sin^{-2}(\pi k / 2)) = \ln 1 = 0$, this yields
	\begin{align*}
		H_{s,n,\bsgamma}(\bsz) 
		&\le
		\sum_{t=1}^n \sum_{\substack{k=1\\k\equiv1\tpmod{2}}}^{2^t-1}
		\sum_{\emptyset \ne \setu \subseteq \{1 \mcol s-1\}} \gamma_\setu \prod_{j\in\setu} \ln \left(\frac1{\sin^{2}(\pi k z_j / 2^t)}\right)
		\\
		&\quad+
		\ln 4 \sum_{t=1}^n \sum_{\substack{k=1\\k\equiv1\tpmod{2}}}^{2^t-1}
		\sum_{\setu \subseteq \{1 \mcol s-1\}} \gamma_{\setu\cup\{s\}} \prod_{j\in\setu} \ln \left( \frac1{\sin^{2}(\pi k z_j / 2^t)}\right) \notag
		\\
		&=
		H_{s-1,n,\bsgamma}(\bsz_{\{1 \mcol s-1\}}) + \ln 4 \left[ \gamma_{\{s\}} (N-1) + H_{s-1,n,\bsgamma\cup \{s\}}(\bsz_{\{1 \mcol s-1\}}) \right]
		,
	\end{align*}
	where in the last step we used the identity in \eqref{eq:rewrite-H} and separated the contribution of $\setu=\emptyset$ in the second term,
	which gives the claimed result.
\end{proof}

The result in Theorem \ref{theorem:H-induction} relates the quantity $H_{s,n,\bsgamma}(\bsz)$ to the quantities 
$H_{s-1,n,\bsgamma}(\bsz_{\{1 \mcol s-1\}})$ and $H_{s-1,n,\bsgamma \cup \{s\}}(\bsz_{\{1 \mcol s-1\}})$ of dimension $s-1$. Intuitively, since we are analyzing a component-by-component construction, this suggests an inductive argument over the dimension. However, note that the components $z_1,\ldots,z_{s-1}$ have been selected with respect to the function $h_{s-1,n,v,\bsgamma}$, and thus $H_{s-1,n,\bsgamma}$, with weights $\bsgamma$ while for the next step in an inductive argument, we require a statement for $H_{s-1,n,\bsgamma \cup \{s\}}$ with modified weights $\bsgamma \cup \{s\}$.

In the following, we will therefore restrict ourselves to weights which have more structure than general weights, in particular, we will assume product weights $\gamma_{\setu} = \prod_{j \in \setu} \gamma_j$ for a sequence of positive reals $\{\gamma_j\}_{j \ge 1}$. For product weights, we obtain the following estimate for $H_{s,n,\bsgamma}(\bsz)$.

\begin{theorem} \label{theorem:upper_bound_H}
	Let $n, s \in \N$, $N=2^n$, and let $\bsgamma=\{\gamma_\setu\}_{\setu\subseteq \{1 \mcol s\}}$, with $\gamma_{\setu} = \prod_{j \in \setu} \gamma_j$ and positive reals 
	$\{\gamma_j\}_{j \ge 1}$, be a sequence of product weights. Furthermore, let the generating vector $\bsz \in \Z_N^s$ be constructed by \RefAlg{alg:cbc-dbd}. Then for $H_{s,n,\bsgamma}(\bsz)$ the following upper bound holds:
	\begin{equation*}
		H_{s,n,\bsgamma}(\bsz) \notag
		\le
		N \left[ -1 + \prod_{j=1}^s (1 + \gamma_j \ln 4) \right]
		.
	\end{equation*}
\end{theorem}

\begin{proof}
	Firstly, note that due to the formulation of Algorithm \ref{alg:cbc-dbd}, the estimate
	\begin{equation} \label{eq:estimate-H-induction-recall}
		H_{r,n,\bsgamma}(\bsz_{\{1 \mcol r\}})
		\le 
		H_{r-1,n,\bsgamma}(\bsz_{\{1 \mcol r-1\}}) + \ln 4 \left[ \gamma_{\{r\}} N + H_{r-1,n,\bsgamma \cup \{r\}}(\bsz_{\{1 \mcol r-1\}}) \right]
	\end{equation}
	as in Theorem \ref{theorem:H-induction} holds for any $2 \le r \le s$. 
	
	For product weights $\gamma_{\setu} = \prod_{j \in \setu} \gamma_j$, we can rewrite $H_{r,n,\bsgamma}(\bsz_{\{1 \mcol r\}})$ as
	\begin{align} \label{eq:H-prod-weights}
		H_{r,n,\bsgamma}(\bsz_{\{1 \mcol r\}}) \notag
		&=
		\sum_{k=1}^{2^n-1} \sum_{\emptyset \ne \setu \subseteq \{1 \mcol r\}} \gamma_\setu \prod_{j\in\setu} \ln \left(\frac1{\sin^{2}(\pi k z_j / 2^n)}\right)
		\\
		&=
		-(N-1) + \sum_{k=1}^{N-1} \prod_{j=1}^r \left( 1 + \gamma_j \ln \left(\frac1{\sin^{2}(\pi k z_j / N)}\right) \right)
	\end{align}
	with $1 \le r \le s$. Furthermore, note that for any $2 \le r \le s$ we obtain
	\begin{align} \label{eq:identity-H_gamma-cup-r}
		H_{r-1,n,\bsgamma \cup \{r\}}(\bsz_{\{1 \mcol r-1\}}) \notag
		&=
		\sum_{k=1}^{2^n-1} \sum_{\emptyset \ne \setu \subseteq \{1 \mcol r-1\}} \gamma_{\setu\cup \{r\}} \prod_{j\in\setu} 
		\ln \left(\frac1{\sin^{2}(\pi k z_j / 2^n)}\right)
		\\
		&=
		\sum_{k=1}^{N-1} \sum_{\emptyset \ne \setu \subseteq \{1 \mcol r-1\}} \gamma_r \prod_{j\in\setu} 
		\gamma_j \ln \left(\frac1{\sin^{2}(\pi k z_j / N)}\right)
		=
		\gamma_r H_{r-1,n,\bsgamma}(\bsz_{\{1 \mcol r-1\}})
		.
	\end{align}
	Using this identity, we can apply the estimate in Theorem \ref{theorem:H-induction} which yields
	\begin{align*}
		H_{s,n,\bsgamma}(\bsz) 
		&\le 
		H_{s-1,n,\bsgamma}(\bsz_{\{1 \mcol s-1\}}) + \ln 4 \left[ \gamma_s N + H_{s-1,n,\bsgamma \cup \{s\}}(\bsz_{\{1 \mcol s-1\}}) \right]
		\\
		&=
		H_{s-1,n,\bsgamma}(\bsz_{\{1 \mcol s-1\}}) + \ln 4 \left[ \gamma_s N + \gamma_s H_{s-1,n,\bsgamma}(\bsz_{\{1 \mcol s-1\}}) \right]
		\\
		&=
		(1 + \gamma_s \ln 4) H_{s-1,n,\bsgamma}(\bsz_{\{1 \mcol s-1\}}) + \gamma_s N \ln 4
		.
	\end{align*}
	Combining \eqref{eq:estimate-H-induction-recall} with identity \eqref{eq:identity-H_gamma-cup-r}, this estimate can be applied recursively for dimensions $s-1$ to $1$ to obtain
	\begin{align} \label{eq:bound_H_H_1}
		&H_{s,n,\bsgamma}(\bsz) \notag
		\le
		(1 + \gamma_s \ln 4) H_{s-1,n,\bsgamma}(\bsz_{\{1 \mcol s-1\}}) + \gamma_s N \ln 4
		\\
		&\le \notag
		(1 + \gamma_s \ln 4) \left[ (1 + \gamma_{s-1} \ln 4) H_{s-2,n,\bsgamma}(\bsz_{\{1 \mcol s-2\}}) + \gamma_{s-1} N \ln 4 \right] + \gamma_s N \ln 4
		\\
		&= \notag
		H_{s-2,n,\bsgamma}(\bsz_{\{1 \mcol s-2\}}) \prod_{j=s-1}^s (1 + \gamma_j \ln 4) + N \left[ -1 + \prod_{j=s-1}^s (1 + \gamma_j \ln 4) \right] 
		\\
		&\le
		H_{1,n,\bsgamma}(z_1) \prod_{j=2}^s (1 + \gamma_j \ln 4) + N \left[ -1 + \prod_{j=2}^s (1 + \gamma_j \ln 4) \right]
		.
	\end{align}
	Comparing to \eqref{eq:H-prod-weights}, we see that $H_{1,n,\bsgamma}(z_1)$ equals
	\begin{align*}
		H_{1,n,\bsgamma}(z_1)
		&=
		-(N-1) + \sum_{k=1}^{N-1} \left( 1 + \gamma_1 \ln \left(\frac1{\sin^{2}(\pi k z_1 / N)}\right) \right)
		=
		-2 \gamma_1 \sum_{k=1}^{N-1} \ln \left( \sin \left( \frac{\pi k}{N} \right)\right)
		\\
		&=
		-2 \gamma_1 \ln \left(\prod_{k=1}^{N-1} \sin \left( \frac{\pi k}{N} \right)\right)
		=
		-2 \gamma_1 \ln \left( \frac{N}{2^{N-1}} \right)
		=
		-2 \gamma_1 \left[ \ln N - (N-1) \ln 2 \right]
		\\
		&=
		2 \gamma_1 \left[ (N-1) \ln 2 - n \ln 2\right]
		=
		\gamma_1 (N - n - 1) \ln 4
		,
	\end{align*}
	where we used that $z_1=1$ and the identity
	\begin{equation} \label{eq:prod_sin}
		\prod_{k=1}^{N-1} \left(2\sin\left(\frac{\pi k}{N}\right)\right) 
		= 
		N
		,
	\end{equation}
	see, e.g., \cite[Proposition 25]{Kor63}. Combining the obtained expression with \eqref{eq:bound_H_H_1} finally gives
	\begin{equation*}
		H_{s,n,\bsgamma}(\bsz) \notag
		\le
		N \gamma_1 \ln 4 \prod_{j=2}^s (1 + \gamma_j \ln 4) + N \left[ -1 + \prod_{j=2}^s (1 + \gamma_j \ln 4) \right]
		=
		N \left[ -1 + \prod_{j=1}^s (1 + \gamma_j \ln 4) \right],
	\end{equation*}
	which is the claim.
\end{proof}

We are now able to show the main result regarding the component-by-component digit-by-digit construction. 

\begin{theorem}\label{thm:optcoeff-dbd}
	Let $N=2^n$, with $n \in \N$, and let $\bsgamma=\{\gamma_\setu\}_{\setu\subseteq \{1 \mcol s\}}$, with $\gamma_{\setu} = \prod_{j \in \setu} \gamma_j$ and positive reals 
	$\{\gamma_j\}_{j \ge 1}$, be a sequence of product weights. Furthermore, denote by $\bsz = (z_1, \ldots, z_s)$ the corresponding generating vector constructed by \RefAlg{alg:cbc-dbd}. Then the following estimate holds: 
	\begin{equation} \label{eq:optcoeff-dbd}
		T(N,\bsz)
		\le
		\frac1N \left[\prod_{j=1}^s \left( 1 + \gamma_j (\ln 4 + 2(1+\ln N)) \right) + 2 (1+\ln N) \prod_{j=1}^s \left( 1 + \gamma_j (2(1 + 2\ln N)) \right) \right]
	\end{equation}
	and the $z_1, \ldots, z_s$ are optimal coefficients modulo $N$. Moreover, if the weights satisfy	 	
	\begin{align*}
		\sum_{j \ge 1} \gamma_j
		&<
		\infty 
		,
	\end{align*}
	then $T(N,\bsz)$ can be bounded independently of the dimension.
\end{theorem}

\begin{proof}
	Combining the bound on $T(N,\bsz)$ in Theorem \ref{thm:T_target_CBCDBD} for product weights, and inserting for $\bsz$ the generating vector obtained by \RefAlg{alg:cbc-dbd}, with the estimate on $H_{s,n,\bsgamma}(\bsz)$ from Theorem \ref{theorem:upper_bound_H} gives 
	\begin{align} \label{eq:T_bound_main}
		T(N,\bsz) \notag
		&\le
		\frac1N \left[-1 + \prod_{j=1}^{s} \left( 1 + \gamma_j (\ln 4 + 2(1+\ln N)) \right) \right]
		- \prod_{j=1}^s (1 + \gamma_j \ln 4) + 1
		\\
		&\quad+ \notag
		\frac{2 (1+\ln N)}{N} \sum_{\emptyset\neq \setu\subseteq \{1 \mcol s\}} \abs{\setu} \prod_{j \in \setu} \gamma_j (1 + 2 \ln N) + \frac1N H_{s,n,\bsgamma}(\bsz)
		\\
		&\le
		\frac1N \prod_{j=1}^{s} \left( 1 + \gamma_j (\ln 4 + 2(1+\ln N)) \right)
		+ 
		\frac{2 (1+\ln N)}{N} \prod_{j=1}^s \left( 1 + \gamma_j (2(1 + 2 \ln N)) \right)
		,
	\end{align}
	where we used that $\abs{\setu} \le 2^{\abs{\setu}}$.
	Note that the condition on $\bsz$ in Theorem \ref{thm:T_target_CBCDBD} is fulfilled since by the formulation of Algorithm \ref{alg:cbc-dbd} all $z_j$ are odd such that
	$\gcd(z_j,N) = \gcd(z_j,2^n) = 1$ is satisfied. This estimate yields the result in \eqref{eq:optcoeff-dbd} and we can deduce that 
	\begin{equation*}
		T(N,\bsz) 
		\le 
		\frac{C_s}{N} (\ln N)^{s+1} \prod_{j=1}^s \left( 1 + \gamma_j \right)
	\end{equation*}
	for some constant $C_s$. As any power of $N$ grows asymptotically faster than $\ln N$, we see that \RefAlg{alg:cbc-dbd} yields optimal coefficients modulo $N$
	in the sense of Definition \ref{def:opt_coeff} for the values $N=2^n$ with $n \in \N$. Additionally, we obtain from \eqref{eq:T_bound_main} that
	\begin{align*}
		N \, T(N,\bsz) 
		&\le
		\prod_{j=1}^{s} \left( 1 + \gamma_j (8 \ln N) \right) + 2 (1+\ln N) \prod_{j=1}^s \left( 1 + \gamma_j (8 \ln N) \right)
		\\
		&=
		(3 + 2 \ln N) \!\prod_{j=1}^s \!\left( 1 + \gamma_j (8 \ln N) \right)
		\le 
		\widetilde{C}(\delta/2) N^{\delta/2} \!\prod_{j=1}^s \!\left( 1 + \gamma_j (8 \ln N) \right)
	\end{align*}	
	for an arbitrary $\delta>0$, where $\widetilde{C}(\delta/2)$ is a constant depending only on $\delta$. We can then estimate
	\begin{equation*}
		\prod_{j=1}^s \!\left( 1 + \gamma_j (8 \ln N) \right)
		\le
		\prod_{j=1}^\infty \!\left( 1 + \gamma_j (8 \ln N) \right)
	\end{equation*}
	and, due to the imposed condition on the weights, employ \cite[Lemma~3]{HN03} to see that this last expression is of order $\calO (N^{\delta/2})$ with implied 
	constant independent of the dimension. This yields the claimed result.
\end{proof}

\begin{corollary} \label{cor:main-result-dbd}
	Let $N=2^n$, with $n \in \N$, and let $\bsgamma=\{\gamma_\setu\}_{\setu\subseteq \{1 \mcol s\}}$ be a sequence of product weights, i.e., 
	$\gamma_{\setu} = \prod_{j \in \setu} \gamma_j$,  with positive reals $\{\gamma_j\}_{j \ge 1}$ satisfying
	\begin{align*}
		\sum_{j \ge 1} \gamma_j
		&<
		\infty
		.
	\end{align*}
	Denote by $\bsz=(z_1,\ldots,z_s)$ the generating vector constructed by \RefAlg{alg:cbc-dbd}. 
	Then, for any $\delta>0$ and each $\alpha>1$, the worst-case error $e_{N,s,\alpha,\bsgamma^{\alpha}}(\bsz)$ satisfies
	\begin{equation*}
		e_{N,s,\alpha,\bsgamma^{\alpha}}(\bsz)
		\le
		\frac{1}{N^{\alpha}} \left( \prod_{j=1}^{s} \left( 1 + \gamma_j^\alpha (4 \zeta(\alpha)) \right) + C(\bsgamma,\delta) N^{\alpha \delta} \right)
	\end{equation*}  
	with weight sequence $\bsgamma^{\alpha}=\{\gamma_{\setu}^{\alpha}\}_{\setu \subseteq \{1 \mcol s\}}$ and a positive constant $C(\bsgamma,\delta)$ independent of $s$ and $N$.
\end{corollary}

\begin{proof}
	The proof works analogously to that of \RefThm{thm:ex_res} by using the estimate obtained in Theorem \ref{thm:optcoeff-dbd}, namely
	\begin{equation*}
		T(N,\bsz)
		=
		\sum_{\bszero \ne \bsm \in M_{N,s}} \frac{\delta_N(\bsm \cdot \bsz)}{r_{1,\bsgamma}(\bsm)}
		\le
		\bar{C}(\bsgamma,\delta) N^{-1+\delta}
		,
	\end{equation*}
	with a positive constant $\bar{C}(\bsgamma,\delta)$ independent of $s$ and $N$.
\end{proof}

While for the previous results we restricted ourselves to product weights, we expect that analogous statements can be derived for general weights. 
In fact the only point where the current theory does not seem to work for general weights is the proof of Theorem \ref{theorem:upper_bound_H}.
A generalization would be of great interest and will be the subject of future research.

\begin{remark}
	Note that the proof technique used for showing the result in Corollary \ref{cor:main-result-dbd} is such that we do not need any knowledge about the smoothness 
	parameter $\alpha$ when running Algorithm \ref{alg:cbc-dbd}. We stress that this is not the case for all formulations of component-by-component algorithms 
	in the literature. On the other hand, similar ideas are to be found in \cite[Theorem 5.5]{N92b} and \cite[Lemma 4.20]{LP14}. However, the results in \cite{N92b} and \cite{LP14} 
	are formulated for different algorithms than the one considered in the present paper. A similar comment applies to Corollary \ref{cor:main-result-cbc} below.
\end{remark}

\subsection{The weighted component-by-component algorithm by Korobov} \label{sec:cbc}

In this section, we study a component-by-component (CBC) algorithm that is motivated by earlier work of Korobov, see, e.g., \cite{Kor63} for a reference.
In contrast to Section \ref{sec:dbd}, we will assume that $N$ is prime throughout the present section. The following theorem is a key ingredient in our main result 
on the CBC algorithm, and is analogous to Theorem~\ref{thm:T_target_CBCDBD}.

\begin{theorem}\label{thm:T_target_CBC}
	Let $N$ be prime and let $\bsgamma=\{\gamma_\setu\}_{\setu \subseteq \{1 \mcol s\}}$ be a sequence of positive weights. 
	Furthermore, let $\bsz = (z_1, \ldots, z_s)\in\{1,\ldots,N-1\}^s$. Then the following estimate holds:
	\begin{align} \label{eq:T_target_CBC}
		T(N,\bsz) 
		&\le
		\sum_{\emptyset\neq\setu\subseteq\{1 \mcol s\}}\frac{\gamma_\setu}{N} \sum_{k=1}^{N-1} K_\setu \left(\left\{ \frac{k \bsz_\setu}{N} \right\}\right) 
		+ \sum_{\emptyset\neq\setu\subseteq\{1 \mcol s\}}\frac{\gamma_\setu}{N} (2(1 + \ln N))^{\abs{\setu}} \nonumber \\
		&\quad+
		\sum_{\emptyset\neq\setu\subseteq\{1 \mcol s\}}\frac{\gamma_\setu}{N} \, 2 {\abs{\setu}} \, (1+ 2\ln N)^{\abs{\setu}} \, (1+\ln N)
		,
	\end{align}
	where, for a subset $\emptyset\neq\setu\subseteq \{1 \mcol s\}$, we define the function $K_{\setu}: \R^{|\setu|} \to \R$ via
	\begin{equation} \label{eq:def_K_u}
		K_\setu(\bsx_\setu)
		:=
		\prod_{j\in\setu} (-2\ln (2 \sin \pi x_j))
		=
		\sum_{\bsm_\setu \in (\Z\setminus \{0\})^{\abs{\setu}}} \frac{\rme^{2 \pi \icomp \bsm_\setu \cdot \bsx_\setu}}{\prod_{j\in\setu} \abs{m_j}} .
	\end{equation}
\end{theorem}

\begin{proof}
In a similar fashion as in the proof of Theorem~\ref{thm:T_target_CBCDBD}, we obtain from \eqref{eq:TN-subsets}
\begin{align}
	\notag
	&\hspace{-4mm} T(N, \bsz)
	=
	\sum_{\bszero \ne \bsm \in M_{N,s}} \frac{\delta_N(\bsm\cdot\bsz)}{r_{1,\bsgamma}(\bsm)}
	=
	\sum_{\emptyset \ne \setu \subseteq \{1 \mcol s\}} \frac{\gamma_\setu}{N} \sum_{k=0}^{N-1}
	\left[
	\sum_{\bsm_\setu \in M_{N,|\setu|}^\ast} \frac{\rme^{2\pi\icomp k\,\bsz_\setu\cdot\bsm_\setu/N}}{\prod_{j\in\setu} |m_j|}
	\right]
	\\\notag
	&=
	\sum_{\emptyset \ne \setu \subseteq \{1 \mcol s\}} \frac{\gamma_\setu}{N} \left(
	\sum_{\bsm_\setu \in M_{N,|\setu|}^\ast} \frac{1}{\prod_{j\in\setu} |m_j|}
	+
	\sum_{k=1}^{N-1} \prod_{j\in\setu} \left(\sum_{m \in M_{N,1}^\ast} \frac{\rme^{2\pi\icomp k m z_j / N}}{|m|} \right)
	\right)
	\\ \notag
	&=
	\sum_{\emptyset \ne \setu \subseteq \{1 \mcol s\}} \frac{\gamma_\setu}{N} \left(
	\sum_{\bsm_\setu \in M_{N,|\setu|}^\ast} \frac{1}{\prod_{j\in\setu} |m_j|}
	+
	\sum_{k=1}^{N-1}
	\left[ \prod_{j\in\setu} v_j(k) - \prod_{j\in\setu} u_j(k) + \prod_{j\in\setu} u_j(k) \right]
	\right)
	\\ \notag
	&=
	\sum_{\emptyset \ne \setu \subseteq \{1 \mcol s\}} \frac{\gamma_\setu}{N}
	\sum_{\bsm_\setu \in M_{N,|\setu|}^\ast} \frac{1}{\prod_{j\in\setu} |m_j|}
	+
	\sum_{\emptyset \ne \setu \subseteq \{1 \mcol s\}} \frac{\gamma_\setu}{N} \sum_{k=1}^{N-1} \prod_{j\in\setu} u_j(k)
	\\
	\notag
	&\quad+
	\sum_{\emptyset \ne \setu \subseteq \{1 \mcol s\}} \frac{\gamma_\setu}{N} \sum_{k=1}^{N-1}
	\theta_\setu(k) \left(\prod_{j\in\setu} \left(\bar{u}_j(k) + |r_j(k)|\right)\right) \sum_{j\in\setu} |r_j(k)|
	\\
	\label{eq:T-bound-CBC}
	&=
	\sum_{\emptyset\neq\setu\subseteq\{1 \mcol s\}}\frac{\gamma_\setu}{N} \sum_{k=1}^{N-1} K_\setu \left(\left\{ \frac{k \bsz_\setu}{N} \right\}\right)
	+
	\sum_{\emptyset \ne \setu \subseteq \{1 \mcol s\}} \frac{\gamma_\setu}{N}
	\sum_{\bsm_\setu \in M_{N,|\setu|}^\ast} \frac{1}{\prod_{j\in\setu} |m_j|}
	\\
	\notag
	&\quad+
	\sum_{\emptyset \ne \setu \subseteq \{1 \mcol s\}} \frac{\gamma_\setu}{N} \sum_{k=1}^{N-1}
	\theta_\setu(k) \left(\prod_{j\in\setu} \left(\bar{u}_j(k) + |r_j(k)|\right)\right) \sum_{j\in\setu} |r_j(k)|
	,
\end{align}
where we used Lemma \ref{lemma:diff_prod} with
\begin{align*}
  u_j=u_j(k) &:= -2\ln(2 \sin (\pi z_j k / N)),
  &
  \bar{u}_j=\bar{u}_j(k) &:= 2 \ln N,
  \\
  v_j = v_j(k) &:= \sum_{m \in M_{N,1}^\ast} \frac{\rme^{2\pi\icomp k m z_j / N}}{|m|},
  &
  r_j = r_j(k) &:= \frac{\tau_j(k)}{N \, \|z_j k / N\|}
  ,
\end{align*}
and all $|\theta_\setu(k)| \le 1$ and $|\tau_j(k)| \le 1$. Due to Lemmas \ref{lem:ln-sin} and \ref{lemma:expr_target_function}, Condition (a) of Lemma \ref{lemma:diff_prod} holds, and Conditions (b) and (c) in the same lemma are satisfied since, as in the proof of Theorem~\ref{thm:T_target_CBCDBD}, we have
\begin{equation*}
	|u_j| \le \abs{-2 \ln\left(2\sin \frac{\pi}{N}\right)} 
	= 
	\abs{-\ln 4 + \ln\left(\sin^{-2} \left( \frac{\pi}{N} \right) \right)} 
	\le 
	\abs{\ln\left(\sin^{-2} \left( \frac{\pi}{N} \right) \right)} 
	\le 
	2 \ln N 
	= 
	\bar{u}_j \ge 1
\end{equation*}
for $j=1,\ldots,s$, as long as $N \ge 2$. Similar to the proof of Theorem \ref{thm:T_target_CBCDBD}, we see that the second sum in \eqref{eq:T-bound-CBC} can be bounded as
\begin{equation*}
	\sum_{\emptyset \ne \setu \subseteq \{1 \mcol s\}} \frac{\gamma_\setu}{N} \sum_{\bsm_\setu \in M_{N,|\setu|}^\ast} \frac{1}{\prod_{j\in\setu} |m_j|}
	\le
	\sum_{\emptyset\neq\setu\subseteq\{1 \mcol s\}} \frac{\gamma_\setu}{N} (2(1 + \ln N))^{\abs{\setu}} .
\end{equation*}
Finally, the third sum in~\eqref{eq:T-bound-CBC} can be bounded independently of the choice of $\bsz$,
in exactly the same way as in the proof of Theorem~\ref{thm:T_target_CBCDBD}. Using the estimates for the second sum and the third sum in
\eqref{eq:T-bound-CBC} yields \eqref{eq:T_target_CBC}.
\end{proof}

Based on Theorem \ref{thm:T_target_CBC}, it is obvious that is desirable to find generating vectors such that 
\begin{equation*}
	\sum_{\emptyset\neq\setu\subseteq\{1 \mcol s\}}\frac{\gamma_\setu}{N} \sum_{k=1}^{N-1} K_\setu \left(\left\{ \frac{k \bsz_\setu}{N} \right\}\right) 
\end{equation*}
is small. This motivates the definition of the following quality function for the component-by-component algorithm.

\begin{definition}[Quality function] \label{def:V_N,s}
	For a generating vector $\bsz=(z_1,\ldots,z_s) \in \Z^s$, a number $N \in \N$ and positive weights 
	$\bsgamma=\{\gamma_\setu\}_{\setu\subseteq \{1 \mcol s\}}$ we define the quality function $V_{N,s,\bsgamma}: \Z^s \to \R$ as
	\begin{equation*}
		V_{N,s,\bsgamma}(\bsz) 
		:=
		\sum_{\emptyset\neq\setu\subseteq\{1 \mcol s\}}\gamma_\setu \sum_{k=1}^{N-1} K_\setu \left(\left\{ \frac{k \bsz_\setu}{N} \right\}\right) .
	\end{equation*}  
\end{definition}

Using this quality function, we formulate the component-by-component construction.

\begin{algorithm}[H] 
	\caption{Component-by-component construction}	
	\label{alg:cbc}
	\vspace{5pt}
	\textbf{Input:} Prime number $N$, dimension $s$ and positive weights $\bsgamma=\{\gamma_\setu\}_{\setu\subseteq \{1 \mcol s\}}$. \\
	\vspace{-7pt}
	\begin{algorithmic}
		\STATE Set $z_1=1$.
		\vspace{1pt}
		\FOR{$d=2$ \bf{to} $s$}
		\STATE $z_d = \underset{z \in \{1,\ldots,N-1\}}{\operatorname*{argmin}} V_{N,d,\bsgamma}(z_1,\ldots,z_{d-1},z)$
		\ENDFOR
	\end{algorithmic}
	\vspace{5pt}
	\textbf{Return:} Generating vector $\bsz=(z_1,\ldots, z_s) \in \{1,\ldots,N-1\}^s$ for $N$.
\end{algorithm}

In order to derive an error bound for lattice rules based on generating vectors constructed by Algorithm \ref{alg:cbc}, we first show the following theorem.
\begin{theorem} \label{thm:cbc}
	Let $N > 2$ be prime and $\bsgamma=\{\gamma_\setu\}_{\setu\subseteq \{1 \mcol s\}}$ be a sequence of positive weights with $\gamma_\emptyset=1$.
	Denote by $\bsz$ the corresponding generating vector constructed by Algorithm \ref{alg:cbc}. Then $\bsz \in \{1,\ldots,N-1\}^s$ satisfies 
	\begin{equation*}
		V_{N,s,\bsgamma}(\bsz) 
		\le 
		\sum_{\emptyset\neq\setu\subseteq\{1 \mcol s\}}\gamma_\setu (2\ln N)^{\abs{\setu}}.
	\end{equation*}
\end{theorem}

\begin{proof}
	We prove the statement by induction on $d\in\{1,\ldots,s\}$. For $d=1$ we obtain that
	\begin{align*}
		V_{N,1,\bsgamma}
		&=
		\sum_{\emptyset \neq \setu \subseteq\{1\}}\gamma_\setu \sum_{k=1}^{N-1} K_\setu 
		\left(\left\{ \frac{k z_{\setu}}{N} \right\}\right)
		=
		\gamma_{\{1\}}\sum_{k=1}^{N-1} K_{\{ 1\}} \left(\left\{ \frac{k z_1}{N} \right\}\right) \\
		&=
		\gamma_{\{1\}}\sum_{k=1}^{N-1} \left(-2\ln \left(2\sin \left( \pi\frac{k}{N} \right)\right)\right)
		=
		-2 \gamma_{\{1\}} \sum_{k=1}^{N-1} \ln \left(2\sin \left( \pi \frac{k}{N} \right)\right)\\
		&=
		-2 \gamma_{\{1\}} \ln\left(\prod_{k=1}^{N-1} 2\sin \left(\pi \frac{k}{N}\right) \right)
		=
		-2 \gamma_{\{1\}} \ln N
		\le
		\gamma_{\{1\}} (2\ln N) 
		,
	\end{align*}
	where the penultimate equality follows from \eqref{eq:prod_sin}, see proof of Theorem \ref{theorem:upper_bound_H}.
	
	Consider then $d \ge 2$ and assume that the statement holds for $d-1$, that is, 
	\begin{equation*}
		V_{N,d-1,\bsgamma}(z_1,\ldots,z_{d-1}) 
		\le 
		\sum_{\emptyset\neq\setu\subseteq\{1 \mcol d-1\}}\gamma_\setu (2\ln N)^{\abs{\setu}}
		.
	\end{equation*}
	By the standard averaging argument and again using $\sum_{k=1}^{N-1} \ln \left(2\sin\left(\frac{\pi k}{N}\right)\right) = \ln N$, we obtain
	\begin{align*}
		\lefteqn{V_{N,d,\bsgamma}(z_1,\ldots,z_{d-1},z_d)
		\le
		\frac{1}{N-1} \sum_{z=1}^{N-1} V_{N,d,\bsgamma}(z_1,\ldots,z_{d-1},z)} \\
		&=
		\frac{1}{N-1} \sum_{z=1}^{N-1} \left(\sum_{\emptyset\neq \setu\subseteq \{1 \mcol d-1\}} \!\!\! \gamma_\setu 		
		\sum_{k=1}^{N-1} K_\setu \left(\left\{ \frac{k \bsz_\setu}{N} \right\}\right)
		+ \sum_{\setv\subseteq \{1 \mcol d-1\}} \!\!\! \gamma_{\setv \cup \{d\}}\sum_{k=1}^{N-1} K_{\setv \cup \{d\}}
		\left(\left\{ \frac{k (\bsz_\setv,z)}{N} \right\}\right) \right) \\
		&=
		V_{N,d-1,\bsgamma} (\bsz_{\{1 \mcol d-1\}}) +
		\frac{1}{N-1}\sum_{\setv\subseteq \{1 \mcol d-1\}} \!\!\!\! \gamma_{\setv \cup \{d\}} 
		\sum_{k=1}^{N-1} K_\setv \left(\left\{ \frac{k \bsz_\setv}{N} \right\}\right)
		\sum_{z=1}^{N-1} \left[ -2\ln\left(2\sin \left(\pi\left\{\frac{kz}{N}\right\}\right)\right)\right]\\
		&=
		V_{N,d-1,\bsgamma} (\bsz_{\{1 \mcol d-1\}}) + \frac{1}{N-1}\sum_{\setv\subseteq \{1 \mcol d-1\}} \gamma_{\setv \cup \{d\}} 
		\sum_{k=1}^{N-1} K_\setv \left(\left\{ \frac{k \bsz_\setv}{N} \right\}\right) (-2\ln N) \\
		&\le
		\sum_{\emptyset\neq\setu\subseteq\{1 \mcol d-1\}}\gamma_\setu (2\ln N)^{\abs{\setu}}
		+ \frac{2\ln N}{N-1}\sum_{k=1}^{N-1} \sum_{\setv\subseteq \{1 \mcol d-1\}} \gamma_{\setv \cup \{d\}} 
		\prod_{j\in\setv} \abs{K_\setv \left(\left\{ \frac{k \bsz_\setv}{N} \right\}\right)}
		.
	\end{align*}
	Considering the term
	\begin{equation*}
		\abs{K_\setv \left(\left\{ \frac{k \bsz_\setv}{N} \right\}\right)}
		=
		\prod_{j\in\setv} \abs{-2\ln\left(2\sin \left(\pi\left\{\frac{kz_j}{N}\right\}\right)\right)}
		=
		\prod_{j\in\setv} \abs{-\ln 4 +\ln\left(\sin^{-2} \left(\pi\left\{\frac{kz_j}{N}\right\}\right)\right)}
		,
	\end{equation*}
	we see that, as in the proof of Theorem~\ref{thm:T_target_CBCDBD}, we have
	\begin{equation*}
		0 
		\le 
		\ln \left(\sin^{-2} \left(\pi\left\{\frac{kz_j}{N}\right\}\right)\right) 
		\le 
		2 \ln N
		,
	\end{equation*}
	and hence,
	\begin{equation*}
		\abs{K_\setv \left(\left\{ \frac{k \bsz_\setv}{N} \right\}\right)}
		\le
		\prod_{j\in\setv} (2 \ln N)
		= 
		(2\ln N)^{\abs{\setv}} 
		.
	\end{equation*}
	This finally yields, using the previous estimate,
	\begin{align*}
		V_{N,d,\bsgamma}(z_1,\ldots,z_{d-1},z_d) 
		&\le 
		\sum_{\emptyset\neq\setu\subseteq\{1 \mcol d-1\}} \gamma_\setu (2\ln N)^{\abs{\setu}}
		+ \sum_{\setv\subseteq \{1 \mcol d-1\}} \gamma_{\setv \cup \{d\}} (2\ln N)^{\abs{\setv}} \, (2\ln N) \\ 
		&=
		\sum_{\emptyset\neq\setu \{1 \mcol d\}} \gamma_\setu (2\ln N)^{\abs{\setu}}
	\end{align*}
	as claimed. By induction, the result follows for dimension $s$ and $\bsz$.
\end{proof}

We are now able to show the main result regarding the component-by-component construction in Algorithm \ref{alg:cbc}.

\begin{theorem} \label{thm:optcoeff-cbc}
	Let $N$ be prime and let $\bsgamma=\{\gamma_\setu\}_{\setu \subseteq \{1 \mcol s\}}$ be positive weights.
	Furthermore, let $\bsz=(z_1,\ldots,z_s)$ be the generating vector constructed by Algorithm \ref{alg:cbc}.
		Then the following estimate holds:
		\begin{align} \label{eq:estimate_T_cbc}
			T(N,\bsz)
			\le 
			\frac2N \sum_{\emptyset\neq\setu\subseteq\{1 \mcol s\}} \gamma_\setu 
			\left( (4\ln N)^{\abs{\setu}} + (2 + 4\ln N)^{\abs{\setu}} (1+\ln N) \right),
		\end{align}
		and the $z_1,\ldots,z_s$ are optimal coefficients modulo $N$. Moreover, if the weights satisfy
		\begin{align*}
			\sum_{j \ge 1} \max_{\setv \subseteq \{1 \mcol j-1\}} \frac{\gamma_{\setv \cup \{j\}}}{\gamma_\setv}
			<
			\infty
			,
		\end{align*}
		then $T(N,\bsz)$ can be bounded independently of the dimension. 
\end{theorem}

\begin{proof}
     Combining the bound on $T(N,\bsz)$ in Theorem \ref{thm:T_target_CBC}, inserting for $\bsz$ the generating vector obtained from Algorithm \ref{alg:cbc}, 
     with the bound on $V_{N,s,\bsgamma}(\bsz)$ from Theorem \ref{thm:cbc} yields 
	\begin{eqnarray*}
		T(N,\bsz) 
		&\le& 
		\sum_{\emptyset\neq\setu\subseteq\{1 \mcol s\}}\frac{\gamma_\setu}{N} (2\ln N)^{\abs{\setu}}
		+ \sum_{\emptyset\neq\setu\subseteq\{1 \mcol s\}}\frac{\gamma_\setu}{N} (2(1+\ln N))^{\abs{\setu}} \\
		&&+ 
		\sum_{\emptyset\neq\setu\subseteq\{1 \mcol s\}}\frac{\gamma_\setu}{N} 2 {\abs{\setu}}(1+ 2\ln N)^{\abs{\setu}} (1+\ln N) \\
		&\le&
		2 \sum_{\emptyset\neq\setu\subseteq\{1 \mcol s\}}\frac{\gamma_\setu}{N} (6 \ln N)^{\abs{\setu}}
		+ 2 \sum_{\emptyset\neq\setu\subseteq\{1 \mcol s\}}\frac{\gamma_\setu}{N} (2 + 4\ln N)^{\abs{\setu}} (1+\ln N),
	\end{eqnarray*}
	which gives the claimed inequality in \eqref{eq:estimate_T_cbc}. From this, we can deduce that 
	\begin{equation*}
		T(N,\bsz) 
		\le 
		\frac{C_s}{N} (\ln N)^{s+1} \sum_{\emptyset \ne \setu \subseteq \{1 \mcol s\}} \gamma_\setu
	\end{equation*}
	for some constant $C_s$. As any power of $N$ grows asymptotically faster than $\ln N$, we see that Algorithm \ref{alg:cbc} indeed yields optimal coefficients
	in the sense of Definition \ref{def:opt_coeff} for prime $N$. Furthermore, from \eqref{eq:estimate_T_cbc}, we easily find that 
	\begin{align*}
		N \, T(N,\bsz) 
		&\le
		2 \sum_{\emptyset\neq\setu\subseteq\{1 \mcol s\}} \gamma_\setu (6 \ln N)^{\abs{\setu}}
		+ 2 \sum_{\emptyset\neq\setu\subseteq\{1 \mcol s\}} \gamma_\setu (2 + 4\ln N)^{\abs{\setu}} (1+\ln N)
		\\
		&\le
		2 \sum_{\emptyset\neq\setu\subseteq\{1 \mcol s\}}\gamma_\setu (7 \ln N)^{\abs{\setu}} 
		+ 6 \ln N \sum_{\emptyset\neq\setu\subseteq\{1 \mcol s\}}\gamma_\setu (7 \ln N)^{\abs{\setu}} \\
		&=
		(2 + 6 \ln N) \!\!\! \sum_{\emptyset\neq\setu\subseteq\{1 \mcol s\}} \!\!\!\! \gamma_\setu (7 \ln N)^{\abs{\setu}}
		\le 
		\widetilde{C}(\delta/2) N^{\delta/2} \!\!\!\! \sum_{\emptyset\neq\setu\subseteq\{1 \mcol s\}} \!\!\!\! \gamma_\setu (7 \ln N)^{\abs{\setu}}	
	\end{align*}	
	for an arbitrary $\delta>0$, where $\widetilde{C}(\delta/2)$ is a constant depending only on $\delta$. We can now directly use Lemma \ref{lem:logN} with $a=7$
	to see that the sum in the last expression is of order $\calO (N^{\delta/2})$. This yields the claimed result.
\end{proof}

\begin{corollary} \label{cor:main-result-cbc}
	Let $N$ be prime and let $\bsgamma=\{\gamma_\setu\}_{\setu \subseteq \{1 \mcol s\}}$ be positive weights satisfying
	\begin{align*}
		\sum_{j \ge 1} \max_{\setv \subseteq \{1 \mcol j-1\}} \frac{\gamma_{\setv \cup \{j\}}}{\gamma_\setv}
		&<
		\infty
		.
	\end{align*}
	Denote by $\bsz=(z_1,\ldots,z_s)$ the generating vector constructed by Algorithm \ref{alg:cbc}. 
	Then, for any $\delta>0$ and each $\alpha>1$, the worst-case error $e_{N,s,\alpha,\bsgamma^{\alpha}}(\bsz)$ satisfies
	\begin{align*}
		e_{N,s,\alpha,\bsgamma^{\alpha}}(\bsz)
		&\le
		\frac{1}{N^{\alpha}} \left( \sum_{\emptyset\neq \setu \subseteq \{1 \mcol s\}} \gamma_u^\alpha (4\zeta (\alpha))^{\abs{\setu}} 
		+ C(\bsgamma,\delta) N^{\alpha \delta} \right)
	\end{align*}  
	with weight sequence $\bsgamma^{\alpha}=\{\gamma_\setu^{\alpha}\}_{\setu \subseteq \{1 \mcol s\}}$ and positive constant $C(\bsgamma,\delta)$ independent of $s$ and~$N$.
\end{corollary}
\begin{proof}
	The proof works analogously to that of Corollary \ref{cor:main-result-dbd}.
\end{proof}

\section{Efficient implementation of the construction schemes} \label{sec:fast_impl}

In this section we discuss the efficient implementation of the two introduced algorithms and analyze their complexity. In the previous sections the two construction methods were formulated for general weights. Here, we consider the implementation for the special case of product weights $\gamma_{\setu} = \prod_{j\in\setu} \gamma_j$ for a sequence of positive 
reals $\{\gamma_j\}_{j\ge1}$. 

\subsection{Implementation and cost analysis of the CBC-DBD algorithm}

Recalling the definition of the quality function in Definition \ref{def:h_rv}, we see that for product weights $h_{r,n,v,\bsgamma}$ can be rewritten as follows.
Let $x \in \N$ be an odd integer, $n, s \in \N$ be positive integers, and let $\bsgamma = \{\gamma_j\}_{j\ge 1}$ be a sequence of positive weights. 
For $1 \le v \le n$, $1 \le r \le s$ and odd integers $z_1,\ldots,z_{r-1}$ the quality function $h_{r,n,v,\bsgamma}$ reads
\begin{align} \label{eq:h_rv_prod}
	h_{r,n,v,\bsgamma}(x)
	&=
	\sum_{t=v}^{n} \frac{1}{2^{t-v}} \sum_{\substack{k=1 \\ k \equiv 1 \tpmod{2}}}^{2^t-1} 
	\left[ \sum_{\emptyset \ne \setu\subseteq \{1 \mcol r-1\}} \gamma_{\setu} \prod_{j\in\setu} \ln \left(\frac{1}{\sin^2(\pi k z_j / 2^t)}\right) \right.\nonumber \\ 
	&\phantom{\qquad\qquad\qquad\qquad}+ 
	\left.\sum_{\setu\subseteq \{1 \mcol r-1\}} \gamma_{\setu\cup \{r\}} \left(\prod_{j\in\setu} \ln \left(\frac{1}{\sin^2(\pi k z_j / 2^t)}\right) \right) 
	\ln \left( \frac{1}{\sin^2(\pi k x / 2^v)} \right) \right] \nonumber \\
	&=
	\underbrace{-\sum_{t=v}^{n} \frac{1}{2^{t-v}} \sum_{\substack{k=1 \\ k \equiv 1 \tpmod{2}}}^{2^t-1} 1}_{=:C_{n,v}} 
	+ \sum_{t=v}^{n} \frac{1}{2^{t-v}} \sum_{\substack{k=1 \\ k \equiv 1 \tpmod{2}}}^{2^t-1} 
	\left[ \sum_{\setu\subseteq \{1 \mcol r-1\}} \gamma_{\setu} \prod_{j\in\setu} \ln \left( \frac{1}{\sin^2(\pi k z_j / 2^t)} \right) \right. \nonumber \\ 
	&\phantom{=}+ 
	\left. \prod_{j=1}^{r-1} \left(1+\gamma_j \ln \left( \frac{1}{\sin^2(\pi k z_j / 2^t)}\right) \right) \gamma_r \ln \left( \frac{1}{\sin^2(\pi k x / 2^v)}\right)\right] \nonumber \\
	&=
	C_{n,v} \!+ \underbrace{\sum_{t=v}^{n} \frac{1}{2^{t-v}} \!\!\!\!\!\!\!\! \sum_{\substack{k=1 \\ k \equiv 1 \tpmod{2}}}^{2^t-1} 
	\prod_{j=1}^{r-1} \left[1+\gamma_j \ln\! \left(\frac{1}{\sin^2(\pi k z_j / 2^t)} \right) \right] \!
	\left( 1 + \gamma_r \ln\! \left(\frac{1}{\sin^2(\pi k x / 2^v)}\right)\right)}_{=:\bar{h}_{r,n,v,\bsgamma}(x)}
\end{align}
with $C_{n,v} = 2^{v-1} (2^{v-1} - 2^n)$. As in each minimization step of Algorithm \ref{alg:cbc-dbd} the variable $v$ is fixed, 
$C_{n,v}$ is constant and we can therefore equivalently minimize the expression $\bar{h}_{r,n,v,\bsgamma}(x)$ in \eqref{eq:h_rv_prod}. 

We observe that a single evaluation of $\bar{h}_{r,n,v,\bsgamma}$ requires $\calO(r \sum_{t=v}^n 2^{t-1})$ operations. Therefore, in a naive
implementation of Algorithm \ref{alg:cbc-dbd}, the number of calculations for each inner loop over the $v=2,\ldots,n$, with $N=2^n$, is
\begin{equation*}
	\calO\left(r \sum_{v=2}^n 2 \sum_{t=v}^n 2^{t-1}\right) 
	=
	\calO\left(r \sum_{v=2}^n \sum_{t=v}^n 2^{t}\right)
	=
	\calO\left(r \, (2^n n - 2(2^n - 1))\right)
	=
	\calO\left(r \, 2^n n \right)
	=
	\calO\left(r  N \ln N \right) .
\end{equation*}
Since there is an inner loop for each $r=2,\ldots,s$, the computational cost of a naive implementation of the component-by-component digit-by-digit construction in 
Algorithm \ref{alg:cbc-dbd} is $\calO\left(s^2 N \ln N \right)$. For large $s$ this cost is prohibitive such that we aim for a more efficient implementation. \\

For $1 \le r < s$, let $z_1,\ldots,z_r$ be constructed by Algorithm \ref{alg:cbc-dbd}. 
For integers $t \in \{2,\ldots,n\}$ and odd $k \in \{1,\ldots,2^t - 1\}$, we introduce the term $q(r,t,k)$ as
\begin{equation*}
	q(r,t,k)
	=
	\prod_{j=1}^{r} \left(1+\gamma_j \ln \left(\frac{1}{\sin^2(\pi k z_j / 2^t)}\right) \right)
\end{equation*}
and note that for the evaluation of $\bar{h}_{r,n,v,\bsgamma}(x)$ in \eqref{eq:h_rv_prod} we can compute and store $q(r-1,t,k)$ as it is independent of $v$ and $x$. 
This way $\bar{h}_{r,n,v,\bsgamma}(x)$ can be rewritten as
\begin{equation} \label{eq:fast-eval-h_rv}
	\bar{h}_{r,n,v,\bsgamma}(x)
	=
	\sum_{t=v}^{n} \frac{1}{2^{t-v}} \sum_{\substack{k=1 \\ k \equiv 1 \tpmod{2}}}^{2^t-1} 
	q(r-1,t,k) \left( 1 + \gamma_r\ln \left( \frac{1}{\sin^2(\pi k x / 2^v)}\right) \right) ,
\end{equation}
where, after determining $z_r$, the value of $q(r,t,k)$ is computed via the recurrence relation
\begin{equation} \label{eq:rec_rel}
	q(r,t,k) 
	=
	q(r-1,t,k) \left(1+\gamma_r \ln \left( \frac{1}{\sin^2(\pi k z_r / 2^t)}\right) \right) .
\end{equation}
For an algorithmic realization of this finding, we introduce the vector $\bsp=(p(1),\ldots,p(N-1))$ of length $N-1$ whose components, 
for the current $r \in \{1,\ldots,s\}$, are given by 
\begin{equation*}
	p(k \, 2^{n-t}) 
	= 
	\prod_{j=1}^{r} \left(1+\gamma_j \ln \left( \frac{1}{\sin^2(\pi k z_j / 2^t)}\right) \right)
	=
	q(r,t,k)
\end{equation*}
for each $t=1,\ldots,n$ and corresponding odd index $k$ in $\{1,3,\ldots,2^t - 1\}$. Furthermore, we note that for the evaluation of $\bar{h}_{r,n,v,\bsgamma}$ we do not 
require the values of $q(r,t,k)$ (or $\bsp$) for $t=2,\ldots,v-1$. Additionally, due to the way the $z_{r,v}$ are constructed in Algorithm \ref{alg:cbc-dbd}, we have that 
$z_{r,n} \bmod{2^v} = z_{r,v}$ for $1 \le v \le n$ and thus, by the periodicity of $\sin^2(\pi x)$,
\begin{equation*}
	\sin^2 \left(\pi \frac{k z_r}{2^v}\right) 
	= 
	\sin^2 \left(\pi \frac{k z_{r,n} \bmod{2^v}}{2^v}\right)
	=
	\sin^2 \left(\pi \frac{k z_{r,v}}{2^v}\right) 
	.
\end{equation*} 
Hence, we can perform the update as in \eqref{eq:rec_rel} for $k=1,3,\ldots,2^v-1$ with $z_{r,n}$ replaced by $z_{r,v}$ 
immediately after each $z_{r,v}$ has been determined.

These observations give rise to a fast implementation of Algorithm \ref{alg:cbc-dbd}.
\begin{algorithm}[H]
	\caption{Fast component-by-component digit-by-digit algorithm}
	\label{alg:fast-cbc-dbd}
	\vspace{5pt}
	\textbf{Input:} Integer $n \in \N$, dimension $s$ and positive weights $\bsgamma=\{\gamma_\setu\}_{\setu\subseteq \{1 \mcol s\}}$. \\
	\vspace{-10pt}
	\begin{algorithmic}
		\FOR{$t=2$ \TO $n$}
		\FOR{$k=1$ \TO $2^{t}-1$ \textbf{in steps of} $2$}
		\STATE $p(k \, 2^{n-t}) = \left(1+\gamma_1 \ln\left( \frac{1}{\sin^2(\pi k / 2^t)}\right) \right)$
		\ENDFOR
		\ENDFOR
		\vspace{5pt}
		\STATE Set $z_{1,n} = 1$ and $z_{2,1} = \ldots = z_{s,1} = 1$.
		\vspace{5pt}
		\FOR{$r=2$ \TO $s$}
		\FOR{$v=2$ \TO $n$}
		\STATE $z^{\ast} = \underset{z \in \{0,1\}}{\argmin} \; \bar{h}_{r,n,v,\bsgamma}(z_{r,v-1} + z \, 2^{v-1})$, 
		where $\bar{h}_{r,n,v,\bsgamma}$ is evaluated using \eqref{eq:fast-eval-h_rv}.
		\STATE $z_{r,v} = z_{r,v-1} + z^{\ast} \, 2^{v-1}$
		\FOR{$k=1$ \TO $2^{v}-1$ \textbf{in steps of} $2$}
		\STATE $p(k \, 2^{n-v}) = p(k \, 2^{n-v}) \left(1+\gamma_r \ln \left( \frac{1}{\sin^2(\pi k z_{r,v} / 2^v)}\right) \right)$
		\ENDFOR
		\ENDFOR
		\ENDFOR
		\vspace{5pt}
		\STATE Set $\bsz = (z_1,\ldots,z_s)$ with $z_r := z_{r,n}$ for $r=1,\ldots,s$.
	\end{algorithmic}
	\vspace{5pt}
	\textbf{Return:} Generating vector $\bsz = (z_1,\ldots,z_s)$ for $N=2^n$.
\end{algorithm}

The computational cost of Algorithm \ref{alg:fast-cbc-dbd} is summarized in the following proposition.
\begin{proposition} \label{prop:cost-alg-cbc-dbd}
	Let $n,s \in \N$ and $N=2^n$. For a given positive weight sequence $\bsgamma = \{\gamma_j\}_{j=1}^s$, a generating vector $\bsz=(z_1,\ldots,z_s)$ 
	can be computed via Algorithm \ref{alg:fast-cbc-dbd} using $\calO(s N \ln N)$ operations and requiring $\calO(N)$ memory.
\end{proposition}

\begin{proof} 
		Due to the relation in \eqref{eq:fast-eval-h_rv}, the cost of evaluating $\overline{h}_{r,n,v,\bsgamma}(x)$ can be reduced to $\calO(\sum_{t=v}^{n} 2^{t-1})$. 
		Thus, the number of calculations in the inner loop over $v = 2,\dots, n$ of Algorithm \ref{alg:fast-cbc-dbd} equals  
		\begin{equation*}
			\calO\left(\sum_{v=2}^n 2 \sum_{t=v}^n 2^{t-1}\right) 
			=
			\calO\left(\sum_{v=2}^n \sum_{t=v}^n 2^{t}\right)
			=
			\calO\left(2^n n - 2(2^n - 1)\right)
			=
			\calO\left(2^n n \right)
			=
			\calO\left(N \ln N \right) .
		\end{equation*}
		Hence, the outer loop over $r=2,\ldots,s$, which is the main cost of the algorithm, can be executed in $\calO\left(s N \ln N \right)$ operations. Furthermore, we observe that initialization and update of the vector $\bsp \in \R^{N-1}$ can both be executed in
		$\calO(N)$ operations. To store the vector $\bsp$ itself, we require $\calO(N)$ of memory.  
\end{proof}

We note that the running time of Algorithm \ref{alg:fast-cbc-dbd} can be further reduced by precomputing and storing the $N$ values 
\begin{equation*}
	\ln \left(\frac{1}{\sin^2(\pi k / N)}\right)
	\quad \text{for} \quad
	k=1,\ldots,N-1 .
\end{equation*}
Proposition \ref{prop:cost-alg-cbc-dbd} reveals that the fast implementation of the component-by-component digit-by-digit construction achieves the
same computational complexity as the state-of-the-art component-by-component methods, see, e.g., \cite{NC06b,NC06}. In these
constructions, the speed-up of the algorithm is achieved by exploiting the special (block-) circulant structure of the involved matrices
and by employing a fast matrix-vector product which uses fast Fourier transformations (FFTs). We refer to \cite{NC06b,NC06} for details. 
In contrast, our method does not rely on the use of FFTs and its low complexity is the result of the smaller search space for the components $z_j$ of $\bsz$.  

\subsection{Fast implementation and cost analysis of the CBC algorithm} \label{sec:impl-cbc}

As before, for product weights $\gamma_{\setu} = \prod_{j\in\setu} \gamma_j$ with $\{\gamma_j\}_{j \ge 1} \in \R_{+}^{\N}$, the quality function $V_{N,s,\bsgamma}$ in Definition \ref{def:V_N,s} can be written in a special form. For this purpose, let $N \in \N$ and consider an integer vector $\bsz \in \Z^s$. Then $V_{N,s,\bsgamma}(\bsz)$ equals
\begin{align*}
	V_{N,s,\bsgamma}(\bsz) 
	&=
	\sum_{\emptyset\neq\setu\subseteq\{1 \mcol s\}}\gamma_\setu \sum_{k=1}^{N-1} K_\setu \left(\left\{ \frac{k \bsz_\setu}{N} \right\}\right)
	=
	\sum_{k=1}^{N-1} \left[ -1 + \prod_{j=1}^{s} \left( 1 + \gamma_j K_{\{j\}} \left(\left\{ \frac{k z_j}{N} \right\}\right) \right) \right] \\
	&=
	-(N-1) + \underbrace{\sum_{k=1}^{N-1} \prod_{j=1}^{s} \left( 1 - 2\gamma_j \ln \left( 2 \sin\left(\pi \left\{ \frac{k z_j}{N} \right\} 
	\right) \right) \right)}_{=: \bar{V}_{N,s,\bsgamma}(\bsz)} .
\end{align*}
As the term $N-1$ is constant, we can equivalently minimize the function $\bar{V}_{N,s,\bsgamma}(\bsz)$ in each step of Algorithm \ref{alg:cbc}. Then, observing that the quantity $\bar{V}_{N,s,\bsgamma}(\bsz)$, with symmetric function $\omega(x)=-2\ln(2\sin(\pi x))$, has the same structure as the worst-case error expression which is minimized in the common CBC algorithm, see, e.g., \cite{NC06b,NC06}, we can employ the same machinery to obtain a fast implementation of Algorithm \ref{alg:cbc}. The computational cost of such a fast implementation is summarized in the following proposition.

\begin{proposition} \label{prop:cost-alg-cbc}
	Let $N,s \in \N$. For a given positive weight sequence $\bsgamma = \{\gamma_j\}_{j=1}^s$, a generating vector 
	$\bsz=(z_1,\ldots,z_s)$ can be computed via Algorithm \ref{alg:cbc} using $\calO(s N \ln N)$ operations.
\end{proposition}

Since the mentioned fast implementation of Algorithm \ref{alg:cbc} can be achieved entirely analogously as for the standard CBC construction,
we omit further implementation details and refer the reader to \cite{DKS13,NC06b,NC06}.

\section{Numerical results} \label{sec:num}

In this section, we illustrate the error convergence rates and the computational cost of the two introduced algorithms, discussed in Sections \ref{sec:constr} and \ref{sec:fast_impl}, 
by means of numerical experiments. As before we consider the construction of rank-$1$ lattice rules in the weighted space $E_{s,\bsgamma}^{\alpha}$ of smoothness $\alpha>1$, and, as in Section \ref{sec:fast_impl}, assume product weights $\gamma_{\setu} = \prod_{j \in \setu} \gamma_j, \setu \subseteq \{1 \mcol s\}$. \\

In order to test the competitiveness of our construction methods, we compare the worst-case error of the constructed lattice rules as well as the algorithms' computation times with those of a state-of-the-art fast component-by-component algorithm, see, e.g., \cite{NC06}. Based on \eqref{eq:wce-infty}, we observe that for product weights the worst-case error of an $N$-point rank-$1$ lattice rule in the space $E_{s,\bsgamma}^{\alpha}$ equals
\begin{align} \label{eq:wce_prod}
	e_{N,s,\alpha,\bsgamma}(\bsz)
	=
	-1 + \frac1N \sum_{k=0}^{N-1} \prod_{j=1}^s \left(1 + \gamma_j \sum_{m \in \Z_\ast} \frac{\rme^{2\pi\icomp k m z_j / N}}{|m|^{\alpha}} \right)
	,
\end{align}    
see, e.g., \cite[Section 5.1]{DKS13}, where the character property of lattice points can be used to involve the sums of complex exponentials. Note that for an even smoothness parameter $\alpha$, the sum of exponentials in \eqref{eq:wce_prod} simplifies to the Bernoulli polynomial $B_{\alpha}(\{k z_j / N\})$ modulo a constant, see, e.g., \cite{DKS13}. For other real $\alpha>1$, the occurring sum can be numerically approximated.   

The different algorithms have all been implemented in double-precision and arbitrary-precision floating-point arithmetic using Python 3.6.3. 
The arbitrary-precision is provided by the multi-precision Python library mpmath.

\subsection{Error convergence behavior -- CBC-DBD algorithm} \label{sec:err-cbc-dbd}

We consider the convergence behavior of $e_{N,s,\alpha,\bsgamma^{\alpha}}(\bsz)$ for generating vectors $\bsz$ constructed by the component-by-component digit-by-digit construction in Algorithm \ref{alg:cbc-dbd} for different sequences of product weights $\bsgamma = \{\gamma_j\}_{j \ge 1}$. As a benchmark we use the lattice rules designed by the CBC construction for $N=2^n$ as introduced in \cite{NC06b} using the error $e_{N,s,\alpha,\bsgamma^{\alpha}}$ as the quality function. According to Corollary \ref{cor:main-result-dbd}, the constructed lattice rules will asymptotically achieve the almost optimal error convergence rate of $\calO(N^{-\alpha + \delta})$. We stress that in certain cases the asymptotic rates promised in Corollary \ref{cor:main-result-dbd} will not be visible for the ranges of $N$ considered in our numerical experiments. Therefore, the presented graphs are to be understood as a demonstration of the pre-asymptotic error behavior.

\begin{figure}
	\centering
	\textbf{Error convergence in the space $E_{s,\bsgamma}^{\alpha}$ with $s=100, \alpha=2,3,4$.} \par\medskip 
	\hspace{-0.25cm}
	\centering
	\begin{subfigure}[b]{0.5\textwidth}
		\centering
		\begin{tikzpicture}
		\begin{axis}[%
		width=0.8\textwidth,
		height=0.8\textwidth,
		at={(0\textwidth,0\textwidth)},
		scale only axis,
		xmode=log,
		xmin=10,
		xmax=1000000,
		xminorticks=true,
		xlabel={Number of points $N=2^n$},
		xmajorgrids,
		ymode=log,
		ymin=1e-20,
		ymax=0.05,
		yminorticks=true,
		ylabel={Worst-case error $e_{N,s,\alpha,\mathbf{\gamma^{\alpha}}}(\bsz)$},
		ymajorgrids,
		axis background/.style={fill=white},
		legend style={at={(0.03,0.03)},anchor=south west,legend cell align=left,align=left,draw=white!15!black}
		]
		\addplot [color=mycolor1-fig,solid,line width=0.6pt,mark=o,mark options={solid},forget plot]
		table[row sep=crcr]{%
			64	0.00456848398492677\\
			128	0.00180527770620017\\
			256	0.000527341603738657\\
			512	0.000154711610983483\\
			1024	4.4138016980327e-05\\
			2048	1.7936054185462e-05\\
			4096	5.55536375055292e-06\\
			8192	1.31495247412994e-06\\
			16384	2.89641597990014e-07\\
			32768	9.51220898455531e-08\\
			65536	4.37485710629889e-08\\
			131072	1.37797577659967e-08\\
		};
		\addplot [color=mycolor2-fig,solid,line width=0.6pt,mark=triangle,mark options={solid},forget plot]
		table[row sep=crcr]{%
			64	0.00409976594175709\\
			128	0.00119251461366021\\
			256	0.000361277437484666\\
			512	0.000103860481837796\\
			1024	3.04693603577786e-05\\
			2048	9.01145193308296e-06\\
			4096	2.50414674271563e-06\\
			8192	7.42491276265867e-07\\
			16384	2.09891082841716e-07\\
			32768	5.8218190108069e-08\\
			65536	1.73653800368653e-08\\
			131072	4.74437606413374e-09\\
		};
		\addplot [color=mycolor3-fig,dashed,line width=0.8pt]
		table[row sep=crcr]{%
			64	0.0022548712679664\\
			128	0.00065236158475316\\
			256	0.000188736112481254\\
			512	5.46036446459596e-05\\
			1024	1.5797496141171e-05\\
			2048	4.57040708451647e-06\\
			4096	1.32227415860916e-06\\
			8192	3.8254993881153e-07\\
			16384	1.10676333445583e-07\\
			32768	3.2020004559438e-08\\
			65536	9.26377537154802e-09\\
			131072	2.68012248328068e-09\\
		};
		\addlegendentry{\footnotesize $\mathcal{O}(N^{-1.79})$};
		
		\addplot [color=mycolor4-fig,solid,line width=0.6pt,mark=o,mark options={solid},forget plot]
		table[row sep=crcr]{%
			64	3.61521988706428e-05\\
			128	1.26119132134169e-05\\
			256	1.74687949187365e-06\\
			512	3.26128857995221e-07\\
			1024	3.51880305969369e-08\\
			2048	1.29997860262164e-08\\
			4096	1.73942755276091e-09\\
			8192	1.49916193812972e-10\\
			16384	1.06552243870806e-11\\
			32768	2.37503292181002e-12\\
			65536	1.08740063823396e-12\\
			131072	2.08218406458688e-13\\
		};
		\addplot [color=mycolor5-fig,solid,line width=0.6pt,mark=triangle,mark options={solid},forget plot]
		table[row sep=crcr]{%
			64	3.20857003931377e-05\\
			128	4.24405134978819e-06\\
			256	6.10693274610581e-07\\
			512	9.15064949064824e-08\\
			1024	1.33241000855469e-08\\
			2048	1.85746132533158e-09\\
			4096	2.49635054331026e-10\\
			8192	4.11434405620622e-11\\
			16384	4.98523852212555e-12\\
			32768	6.7523457633164e-13\\
			65536	1.08502304781415e-13\\
			131072	1.68766816480382e-14\\
		};
		\addplot [color=darkgray,dashed,line width=0.8pt]
		table[row sep=crcr]{%
			64	1.76471352162258e-05\\
			128	2.47565497857671e-06\\
			256	3.4730099236257e-07\\
			512	4.87216435003276e-08\\
			1024	6.83498923865687e-09\\
			2048	9.58856773627537e-10\\
			4096	1.34514668601304e-10\\
			8192	1.88705931548723e-11\\
			16384	2.64728962067458e-12\\
			32768	3.71379016982403e-13\\
			65536	5.20994654977233e-14\\
			131072	7.30885209186999e-15\\
		};
		\addlegendentry{\footnotesize $\mathcal{O}(N^{-2.83})$};
		
		\addplot [color=mycolor6-fig,solid,line width=0.6pt,mark=o,mark options={solid},forget plot]
		table[row sep=crcr]{%
			64	4.10993574421105e-07\\
			128	1.350902328208e-07\\
			256	8.94859363946374e-09\\
			512	1.26929904643267e-09\\
			1024	4.26011381546609e-11\\
			2048	1.60612135236372e-11\\
			4096	8.32289135024346e-13\\
			8192	2.60469786257831e-14\\
			16384	5.90643710070039e-16\\
			32768	9.62079731462703e-17\\
			65536	4.29885460703069e-17\\
			131072	5.07347175236328e-18\\
		};
		\addplot [color=mycolor7-fig,solid,line width=0.6pt,mark=triangle,mark options={solid},forget plot]
		table[row sep=crcr]{%
			64	3.50313263877513e-07\\
			128	2.0225519419585e-08\\
			256	1.54397259050068e-09\\
			512	1.24700153969556e-10\\
			1024	8.69354851101778e-12\\
			2048	6.05479560258685e-13\\
			4096	3.74058544629636e-14\\
			8192	3.27874642107338e-15\\
			16384	1.81718581558055e-16\\
			32768	1.5507262173846e-17\\
			65536	3.13844694235168e-18\\
			131072	1.00277368938886e-18\\
		};
		\addplot [color=black,dashed,line width=0.8pt]
		table[row sep=crcr]{%
			64	1.92672295132632e-07\\
			128	1.36264759451582e-08\\
			256	9.63713265345987e-10\\
			512	6.81572595542452e-11\\
			1024	4.82032591745738e-12\\
			2048	3.40910742340199e-13\\
			4096	2.41104307536635e-14\\
			8192	1.70517616176231e-15\\
			16384	1.20596175669763e-16\\
			32768	8.52899419561529e-18\\
			65536	6.03201068233197e-19\\
			131072	4.26605435966558e-20\\
		};
		\addlegendentry{\footnotesize $\mathcal{O}(N^{-3.82})$};
		
		\end{axis}
		\end{tikzpicture}
		\caption{Weight sequence $\bsgamma=\{\gamma_j\}_{j=1}^s$ with $\gamma_j = 1/j^2$.}    
	\end{subfigure}
	\begin{subfigure}[b]{0.5\textwidth}  
		\centering 
		\begin{tikzpicture}
		
		\begin{axis}[%
		width=0.8\textwidth,
		height=0.8\textwidth,
		at={(0\textwidth,0\textwidth)},
		scale only axis,
		xmode=log,
		xmin=10,
		xmax=1000000,
		xminorticks=true,
		xlabel={Number of points $N=2^n$},
		xmajorgrids,
		ymode=log,
		ymin=1e-21,
		ymax=0.1,
		yminorticks=true,
		ylabel={Worst-case error $e_{N,s,\alpha,\mathbf{\gamma^{\alpha}}}(\bsz)$},
		ymajorgrids,
		axis background/.style={fill=white},
		legend style={at={(0.03,0.03)},anchor=south west,legend cell align=left,align=left,draw=white!15!black}
		]
		\addplot [color=mycolor1-fig,solid,line width=0.6pt,mark=o,mark options={solid},forget plot]
		table[row sep=crcr]{%
			64	0.00136397342871476\\
			128	0.00041495206968083\\
			256	0.000113777898495654\\
			512	3.57186608991275e-05\\
			1024	8.96735413484793e-06\\
			2048	3.12393880750252e-06\\
			4096	5.94557183683682e-07\\
			8192	1.73173640260874e-07\\
			16384	4.81706937364865e-08\\
			32768	1.07434643818022e-08\\
			65536	3.40729981226737e-09\\
			131072	1.28708617753193e-09\\
		};
		\addplot [color=mycolor2-fig,solid,line width=0.6pt,mark=triangle,mark options={solid},forget plot]
		table[row sep=crcr]{%
			64	0.00135443541980486\\
			128	0.000352782520955296\\
			256	9.47182153016234e-05\\
			512	2.54418866678807e-05\\
			1024	6.74205804207878e-06\\
			2048	1.77453409528167e-06\\
			4096	4.60002661249674e-07\\
			8192	1.22592457360222e-07\\
			16384	3.15483574591104e-08\\
			32768	8.19217011937004e-09\\
			65536	2.17930226511727e-09\\
			131072	5.67109158519495e-10\\
		};
		\addplot [color=mycolor3-fig,dashed,line width=0.8pt]
		table[row sep=crcr]{%
			64	0.000744939480892675\\
			128	0.000196021310122118\\
			256	5.15805041987398e-05\\
			512	1.3572750900087e-05\\
			1024	3.57149604986438e-06\\
			2048	9.39793570816593e-07\\
			4096	2.47294675233293e-07\\
			8192	6.50724353706765e-08\\
			16384	1.71229802707082e-08\\
			32768	4.50569356565352e-09\\
			65536	1.18561571564153e-09\\
			131072	3.11979633033988e-10\\
		};
		\addlegendentry{\footnotesize $\mathcal{O}(N^{-1.93})$};
		
		\addplot [color=mycolor4-fig,solid,line width=0.6pt,mark=o,mark options={solid},forget plot]
		table[row sep=crcr]{%
			64	1.13368834740657e-05\\
			128	1.8166485482477e-06\\
			256	2.43348499136029e-07\\
			512	5.21642131851584e-08\\
			1024	5.25024469357468e-09\\
			2048	1.4589656944201e-09\\
			4096	7.57672752806919e-11\\
			8192	1.28068368928801e-11\\
			16384	1.97500282913633e-12\\
			32768	1.59493905804064e-13\\
			65536	3.37040686213196e-14\\
			131072	9.82935412695632e-15\\
		};
		\addplot [color=mycolor5-fig,solid,line width=0.6pt,mark=triangle,mark options={solid},forget plot]
		table[row sep=crcr]{%
			64	1.13348899973189e-05\\
			128	1.41526367230331e-06\\
			256	1.85148368058231e-07\\
			512	2.45570266330019e-08\\
			1024	3.14334593096102e-09\\
			2048	4.04508861032469e-10\\
			4096	5.04793453124573e-11\\
			8192	6.64228727495044e-12\\
			16384	8.21131912618704e-13\\
			32768	1.04386938639e-13\\
			65536	1.38458999184636e-14\\
			131072	2.05178314766857e-15\\
		};
		\addplot [color=darkgray,dashed,line width=0.8pt]
		table[row sep=crcr]{%
			64	6.23418949852541e-06\\
			128	7.97841345231606e-07\\
			256	1.02106426554975e-07\\
			512	1.30674129714347e-08\\
			1024	1.6723460758279e-09\\
			2048	2.14024107407535e-10\\
			4096	2.73904541731385e-11\\
			8192	3.5053853927877e-12\\
			16384	4.48613472208133e-13\\
			32768	5.741281625145e-14\\
			65536	7.34759804180265e-15\\
			131072	9.40333544124629e-16\\
		};
		\addlegendentry{\footnotesize $\mathcal{O}(N^{-2.97})$};
		
		\addplot [color=mycolor6-fig,solid,line width=0.6pt,mark=o,mark options={solid},forget plot]
		table[row sep=crcr]{%
			64	1.40458470615939e-07\\
			128	1.11780202154902e-08\\
			256	7.2514379736635e-10\\
			512	1.06571700449759e-10\\
			1024	4.03558187327091e-12\\
			2048	9.90733614782783e-13\\
			4096	1.36115603650878e-14\\
			8192	1.30935087549699e-15\\
			16384	1.16559426831773e-16\\
			32768	3.45818234387286e-18\\
			65536	4.65372362137083e-19\\
			131072	9.76207273895936e-20\\
		};
		\addplot [color=mycolor7-fig,solid,line width=0.6pt,mark=triangle,mark options={solid},forget plot]
		table[row sep=crcr]{%
			64	1.40456754552417e-07\\
			128	8.69405110717746e-09\\
			256	5.56518340338504e-10\\
			512	3.61778864043262e-11\\
			1024	2.27876940256736e-12\\
			2048	1.44530448811119e-13\\
			4096	8.91158916171079e-15\\
			8192	5.74738100312171e-16\\
			16384	3.54789570640588e-17\\
			32768	2.48552725505387e-18\\
			65536	3.85554386927719e-19\\
			131072	2.66476059951592e-19\\
		};
		\addplot [color=black,dashed,line width=0.8pt]
		table[row sep=crcr]{%
			64	7.72512150038291e-08\\
			128	4.93429219043921e-09\\
			256	3.15169663278727e-10\\
			512	2.0130935262345e-11\\
			1024	1.28582982994252e-12\\
			2048	8.21302304152064e-14\\
			4096	5.24593114188092e-15\\
			8192	3.35075080226011e-16\\
			16384	2.1402360486991e-17\\
			32768	1.36703999027963e-18\\
			65536	8.73173936192531e-20\\
			131072	5.57725251834073e-21\\
		};
		\addlegendentry{\footnotesize $\mathcal{O}(N^{-3.97})$};
		
		\end{axis}
		\end{tikzpicture}
		\caption{Weight sequence $\bsgamma=\{\gamma_j\}_{j=1}^s$ with $\gamma_j = 1/j^3$.}
	\end{subfigure}
	\vskip\baselineskip
	\hspace{-0.25cm}
	\centering
	\begin{subfigure}[b]{0.5\textwidth}
		\centering
		\begin{tikzpicture}
		
		\begin{axis}[%
		width=0.8\textwidth,
		height=0.8\textwidth,
		at={(0\textwidth,0\textwidth)},
		scale only axis,
		xmode=log,
		xmin=10,
		xmax=1000000,
		xminorticks=true,
		xlabel={Number of points $N=2^n$},
		xmajorgrids,
		ymode=log,
		ymin=0.00002,
		ymax=25000000,
		yminorticks=true,
		ylabel={Worst-case error $e_{N,s,\alpha,\mathbf{\gamma^{\alpha}}}(\bsz)$},
		ymajorgrids,
		axis background/.style={fill=white},
		legend style={at={(0.03,0.03)},anchor=south west,legend cell align=left,align=left,draw=white!15!black}
		]
		\addplot [color=mycolor1-fig,solid,line width=0.6pt,mark=o,mark options={solid},forget plot]
		table[row sep=crcr]{%
			64	4417079.48643413\\
			128	2208539.31795966\\
			256	1104269.19820593\\
			512	552134.113436488\\
			1024	276066.703147798\\
			2048	138033.261541012\\
			4096	69016.3247481713\\
			8192	34508.0564800858\\
			16384	17253.8180482969\\
			32768	8626.7707941766\\
			65536	4313.28896327669\\
			131072	2156.56910773271\\
		};
		\addplot [color=mycolor2-fig,solid,line width=0.6pt,mark=triangle,mark options={solid},forget plot]
		table[row sep=crcr]{%
			64	4349233.97513606\\
			128	2199934.17774299\\
			256	1092062.79327424\\
			512	544975.459994963\\
			1024	271766.502104828\\
			2048	135828.761904194\\
			4096	67322.022192407\\
			8192	33903.6227190315\\
			16384	17012.9251321677\\
			32768	8530.18056490266\\
			65536	4250.92070965355\\
			131072	2127.45418057421\\
		};
		\addplot [color=mycolor3-fig,dashed,line width=0.8pt]
		table[row sep=crcr]{%
			64	2827002.08383844\\
			128	1414157.70966979\\
			256	707407.341243682\\
			512	353867.990128419\\
			1024	177016.192986314\\
			2048	88549.2145474827\\
			4096	44295.1758519762\\
			8192	22157.87699286\\
			16384	11084.0854198531\\
			32768	5544.61736718673\\
			65536	2773.59661027552\\
			131072	1387.44256764378\\
		};
		\addlegendentry{\footnotesize $\mathcal{O}(N^{-1})$};
		
		\addplot [color=mycolor4-fig,solid,line width=0.6pt,mark=o,mark options={solid},forget plot]
		table[row sep=crcr]{%
			64	387.877686853467\\
			128	193.746068340041\\
			256	96.6315665213086\\
			512	48.0340928835433\\
			1024	23.921630108221\\
			2048	12.0387450337173\\
			4096	5.90424827351647\\
			8192	2.92329253254853\\
			16384	1.42634453252555\\
			32768	0.696563053768625\\
			65536	0.339859818456783\\
			131072	0.16330063823923\\
		};
		\addplot [color=mycolor5-fig,solid,line width=0.6pt,mark=triangle,mark options={solid},forget plot]
		table[row sep=crcr]{%
			64	361.432455046776\\
			128	182.738064076475\\
			256	88.819678868358\\
			512	44.4346449019205\\
			1024	22.1652168955361\\
			2048	10.8576167927295\\
			4096	5.37954662071342\\
			8192	2.61661948780476\\
			16384	1.2773475338927\\
			32768	0.608827210513876\\
			65536	0.28927185901447\\
			131072	0.133285353981137\\
		};
		\addplot [color=darkgray,dashed,line width=0.8pt]
		table[row sep=crcr]{%
			64	234.931095780404\\
			128	115.550056092526\\
			256	56.832899955765\\
			512	27.9530674982589\\
			1024	13.7486206611028\\
			2048	6.76221205757388\\
			4096	3.32597087655262\\
			8192	1.63586740218925\\
			16384	0.804595787777658\\
			32768	0.395737686834019\\
			65536	0.194642228010293\\
			131072	0.0957341142510418\\
		};
		\addlegendentry{\footnotesize $\mathcal{O}(N^{-1.02})$};
		
		\addplot [color=mycolor6-fig,solid,line width=0.6pt,mark=o,mark options={solid},forget plot]
		table[row sep=crcr]{%
			64	14.2709132849861\\
			128	7.08109889271043\\
			256	3.46532136703737\\
			512	1.59647671572642\\
			1024	0.785733512985636\\
			2048	0.422086357174344\\
			4096	0.175270755775307\\
			8192	0.0795068660293254\\
			16384	0.0336611338085015\\
			32768	0.0146743292221481\\
			65536	0.00670279436242792\\
			131072	0.00255255675510684\\
		};
		\addplot [color=mycolor7-fig,solid,line width=0.6pt,mark=triangle,mark options={solid},forget plot]
		table[row sep=crcr]{%
			64	13.2055070573877\\
			128	6.3515715101183\\
			256	2.8752289377894\\
			512	1.35210167341946\\
			1024	0.614997296374464\\
			2048	0.27017115967038\\
			4096	0.123929428701347\\
			8192	0.0480435011470263\\
			16384	0.0193523515203181\\
			32768	0.00786358669078902\\
			65536	0.00325179826963299\\
			131072	0.00123784843985397\\
		};
		\addplot [color=black,dashed,line width=0.8pt]
		table[row sep=crcr]{%
			64	8.58357958730201\\
			128	3.76114581946106\\
			256	1.64805577106507\\
			512	0.72214371761052\\
			1024	0.316428338191083\\
			2048	0.138652308077507\\
			4096	0.0607545539224456\\
			8192	0.0266213802964751\\
			16384	0.0116649344474526\\
			32768	0.00511133134901287\\
			65536	0.00223967895208421\\
			131072	0.000981380674797729\\
		};
		\addlegendentry{\footnotesize $\mathcal{O}(N^{-1.19})$};
		
		\end{axis}
		\end{tikzpicture}
		\caption{Weight sequence $\bsgamma=\{\gamma_j\}_{j=1}^s$ with $\gamma_j = (0.95)^j$.}
	\end{subfigure}
	\begin{subfigure}[b]{0.5\textwidth}  
		\centering 
		\begin{tikzpicture}
		
		\begin{axis}[%
		width=0.8\textwidth,
		height=0.8\textwidth,
		at={(0\textwidth,0\textwidth)},
		scale only axis,
		xmode=log,
		xmin=10,
		xmax=1000000,
		xminorticks=true,
		xlabel={Number of points $N=2^n$},
		xmajorgrids,
		ymode=log,
		ymin=1e-17,
		ymax=1,
		yminorticks=true,
		ylabel={Worst-case error $e_{N,s,\alpha,\mathbf{\gamma^{\alpha}}}(\bsz)$},
		ymajorgrids,
		axis background/.style={fill=white},
		legend style={at={(0.03,0.03)},anchor=south west,legend cell align=left,align=left,draw=white!15!black}
		]
		\addplot [color=mycolor1-fig,solid,line width=0.6pt,mark=o,mark options={solid},forget plot]
		table[row sep=crcr]{%
			64	0.0448432984147419\\
			128	0.0209390552045867\\
			256	0.00718223638700622\\
			512	0.00222713082717785\\
			1024	0.000868313446033489\\
			2048	0.000321281296432149\\
			4096	0.000115589430341442\\
			8192	4.13301613669945e-05\\
			16384	2.25768276888366e-05\\
			32768	6.91855926606921e-06\\
			65536	2.25704042110754e-06\\
			131072	8.37331278614467e-07\\
		};
		\addplot [color=mycolor2-fig,solid,line width=0.6pt,mark=triangle,mark options={solid},forget plot]
		table[row sep=crcr]{%
			64	0.0382667538287894\\
			128	0.0142847955914795\\
			256	0.00519004925790926\\
			512	0.0018563028986182\\
			1024	0.000649295569225828\\
			2048	0.000227164438691047\\
			4096	7.60162357546622e-05\\
			8192	2.73102799286214e-05\\
			16384	9.7005546959578e-06\\
			32768	3.29473743980793e-06\\
			65536	1.10760838906525e-06\\
			131072	3.70643753904948e-07\\
		};
		\addplot [color=mycolor3-fig,dashed,line width=0.8pt]
		table[row sep=crcr]{%
			64	0.0210467146058342\\
			128	0.0074390513271694\\
			256	0.00262936452005296\\
			512	0.00092936013951983\\
			1024	0.000328486317641086\\
			2048	0.000116104894420316\\
			4096	4.1037771695203e-05\\
			8192	1.45049759884448e-05\\
			16384	5.12684582359897e-06\\
			32768	1.81210559189436e-06\\
			65536	6.40496474666695e-07\\
			131072	2.26386219376768e-07\\
		};
		\addlegendentry{\footnotesize $\mathcal{O}(N^{-1.50})$};
		
		\addplot [color=mycolor4-fig,solid,line width=0.6pt,mark=o,mark options={solid},forget plot]
		table[row sep=crcr]{%
			64	0.000885234122394915\\
			128	0.000501121180778071\\
			256	7.44226175001177e-05\\
			512	6.31956933205184e-06\\
			1024	1.74556624440887e-06\\
			2048	3.80295256960537e-07\\
			4096	8.21769738670034e-08\\
			8192	1.74391239875537e-08\\
			16384	1.44005809283461e-08\\
			32768	1.75009000978227e-09\\
			65536	3.22917913325766e-10\\
			131072	6.16268130558317e-11\\
		};
		\addplot [color=mycolor5-fig,solid,line width=0.6pt,mark=triangle,mark options={solid},forget plot]
		table[row sep=crcr]{%
			64	0.000381009808069608\\
			128	8.88565335724644e-05\\
			256	1.89462736532981e-05\\
			512	3.39071547457998e-06\\
			1024	5.79750145408817e-07\\
			2048	1.27024925430706e-07\\
			4096	2.20578490193047e-08\\
			8192	4.54266260991007e-09\\
			16384	1.03390851025441e-09\\
			32768	1.99731888256572e-10\\
			65536	3.43249427660016e-11\\
			131072	7.00954801728504e-12\\
		};
		\addplot [color=darkgray,dashed,line width=0.8pt]
		table[row sep=crcr]{%
			64	0.000209555394438284\\
			128	4.20210366323822e-05\\
			256	8.42625657236441e-06\\
			512	1.68967273331375e-06\\
			1024	3.38821150434402e-07\\
			2048	6.79420160592572e-08\\
			4096	1.36240536940449e-08\\
			8192	2.73195954174116e-09\\
			16384	5.47825420049022e-10\\
			32768	1.09852538541115e-10\\
			65536	2.20281494474047e-11\\
			131072	4.41718848304615e-12\\
		};
		\addlegendentry{\footnotesize $\mathcal{O}(N^{-2.32})$};
		
		\addplot [color=mycolor6-fig,solid,line width=0.6pt,mark=o,mark options={solid},forget plot]
		table[row sep=crcr]{%
			64	3.57796295622052e-05\\
			128	2.50942631308733e-05\\
			256	1.68249186544005e-06\\
			512	3.08335303170983e-08\\
			1024	6.28865519885279e-09\\
			2048	7.87616663531856e-10\\
			4096	1.03462710747759e-10\\
			8192	1.37574381061064e-11\\
			16384	1.78938831114578e-11\\
			32768	8.15547186285366e-13\\
			65536	8.44314445353111e-14\\
			131072	7.54378399249423e-15\\
		};
		\addplot [color=mycolor7-fig,solid,line width=0.6pt,mark=triangle,mark options={solid},forget plot]
		table[row sep=crcr]{%
			64	6.08563097749778e-06\\
			128	8.44312772090429e-07\\
			256	8.82747471683997e-08\\
			512	1.00549228520788e-08\\
			1024	8.53309945874362e-10\\
			2048	1.11242835393357e-10\\
			4096	1.15086270387497e-11\\
			8192	1.14714460742172e-12\\
			16384	1.75636707459577e-13\\
			32768	1.67055087934552e-14\\
			65536	1.90656098475711e-15\\
			131072	1.80936308064528e-16\\
		};
		\addplot [color=black,dashed,line width=0.8pt]
		table[row sep=crcr]{%
			64	3.34709703762378e-06\\
			128	3.74452328903001e-07\\
			256	4.18913897759069e-08\\
			512	4.68654726356782e-09\\
			1024	5.24301661299552e-10\\
			2048	5.86555979448711e-11\\
			4096	6.56202225593312e-12\\
			8192	7.34118099483577e-13\\
			16384	8.21285516826921e-14\\
			32768	9.18802983640035e-15\\
			65536	1.02789943990177e-15\\
			131072	1.14994974696808e-16\\
		};
		\addlegendentry{\footnotesize $\mathcal{O}(N^{-3.16})$};
		
		\end{axis}
		\end{tikzpicture}
		\caption{Weight sequence $\bsgamma=\{\gamma_j\}_{j=1}^s$ with $\gamma_j = (0.7)^j$.}
	\end{subfigure}
	\vskip\baselineskip
	\begin{tikzpicture}
	\hspace{0.05\linewidth}
	\begin{customlegend}[
	legend columns=5,legend style={align=left,draw=none,column sep=1.5ex},
	legend entries={CBC-DBD, standard fast CBC \quad, $\alpha=2$, $\alpha=3$, $\alpha=4$}
	]
	\addlegendimage{color=gray, mark=o,solid,line width=1.0pt,line legend}
	\addlegendimage{color=gray, mark=triangle,solid,line width=1.0pt}  
	\addlegendimage{color=mycolor-alpha1, only marks,mark=square*, solid, line width=5pt}
	\addlegendimage{color=mycolor-alpha2, only marks,mark=square*, solid, line width=5pt}
	\addlegendimage{color=mycolor-alpha3, only marks,mark=square*, solid, line width=5pt}
	\end{customlegend}
	\end{tikzpicture}
	\caption{Convergence results of the worst-case error $e_{N,s,\alpha,\bsgamma^{\alpha}}(\bsz)$ in the weighted space $E_{s,\bsgamma}^\alpha$
		for smoothness parameters $\alpha=2,3,4$ with dimension $s=100$. The generating vectors $\bsz$ are constructed via the component-by-component digit-by-digit algorithm 
		and the classic CBC construction for $N=2^n$, respectively.}  
	\label{fig:cbc-dbd}
\end{figure}

The results in Figure \ref{fig:cbc-dbd} illustrate that the CBC-DBD algorithm reliably constructs good lattice rules with
worst-case error values that are comparable to those of the corresponding lattice rules constructed by the common CBC construction. While the asymptotic error behavior of both constructions is identical, we observe that the actual errors of the CBC-DBD algorithm are always slightly higher. However, this behavior was to be expected and can be explained by the fact that the common CBC construction uses the worst-case error of the space $E_{s,\bsgamma}^{\alpha}$ as the quality function and is therefore directly tailored to the respective function space. The CBC-DBD construction, on the other hand, uses a more general quality function which is independent of the smoothness $\alpha$. Note that this also implies that while we only need to run the CBC-DBD algorithm once for a fixed choice of $N$, $s$, and $\bsgamma$, the usual CBC algorithm needs to be run once for each different $\alpha$.
Furthermore, we observe that the pre-asymptotic decay of the used weight sequences determines the error convergence rate. 
The faster the weights decay, the closer the error rate is to the optimal rate $\calO(N^{-\alpha})$. The latter observation is also not surprising, 
as smaller weights can be expected to yield smaller constants in the errors.

\subsection{Computational complexity: CBC-DBD algorithm}

Here, we illustrate the computational complexity of the component-by-component digit-by-digit construction in Algorithm \ref{alg:fast-cbc-dbd} which was stated in Proposition \ref{prop:cost-alg-cbc-dbd}. To this end, let $n,s \in \N$ and $N=2^n$ and use the weight sequence $\bsgamma = \{\gamma_j\}_{j=1}^s$ with $\gamma_j=j^{-2}$. Note that the chosen weights do not influence the computation times. In Table \ref{tab:dbd_times} we report on the timings for the two construction methods considered in Section \ref{sec:err-cbc-dbd}. The displayed times solely measure the duration for constructing the generating vectors but do not comprise the error calculation. The timings were performed on an Intel Core i5 CPU with 2.3 GHz using Python 3.6.3.

\begin{table}[H]
	\captionof{table}{Computation times (in seconds) for constructing the generating vector $\bsz$ of a lattice rule with $N=2^n$ points in $s$ dimensions using the component-by-component digit-by-digit algorithm (\textbf{bold font}) and the standard fast CBC construction (normal font). For the CBC algorithm we construct the lattice rules with smoothness parameter $\alpha=2$.}	
	\label{tab:dbd_times}
	\centering
	\begin{tabular}{p{2cm}p{2cm}p{2cm}p{2cm}p{2cm}p{2cm}} 
		\toprule[1.2pt]
		& $s=50$ & $s=100$ & $s=500$ & $s=1000$ & $s=2000$ \\ 
		\toprule[1.2pt]
		\multirow{2}{6em}{$n=10$}
		& 0.038 & 0.075 & 0.37 & 0.743 & 1.485 \\
		& \textbf{0.061} & \textbf{0.119} & \textbf{0.595} & \textbf{1.184} & \textbf{2.371} \\
		\midrule
		\multirow{2}{6em}{$n=12$}
		& 0.047 & 0.096 & 0.476 & 0.951 & 1.897 \\
		& \textbf{0.093} & \textbf{0.185} & \textbf{0.922} & \textbf{1.843} & \textbf{3.685} \\
		\midrule
		\multirow{2}{6em}{$n=14$}
		& 0.068 & 0.138 & 0.674 & 1.339 & 2.676 \\
		& \textbf{0.155} & \textbf{0.31} & \textbf{1.547} & \textbf{3.081} & \textbf{6.166} \\
		\midrule
		\multirow{2}{6em}{$n=16$}
		& 0.165 & 0.304 & 1.423 & 2.845 & 5.626 \\
		& \textbf{0.344} & \textbf{0.678} & \textbf{3.394} & \textbf{6.804} & \textbf{13.624} \\
		\midrule
		\multirow{2}{6em}{$n=18$}
		& 0.586 & 1.053 & 4.746 & 9.497 & 18.867 \\
		& \textbf{1.145} & \textbf{2.293} & \textbf{11.63} & \textbf{23.1} & \textbf{46.184} \\
		\midrule
		\multirow{2}{6em}{$n=20$}
		& 3.357 & 6.203 & 28.935 & 57.438 & 114.284 \\
		& \textbf{6.31} & \textbf{12.757} & \textbf{64.102} & \textbf{128.897} & \textbf{257.454} \\
		\midrule
	\end{tabular}
\end{table}

The timings in Table \ref{tab:dbd_times} confirm that both considered algorithms have a similar dependence on $N$ and $s$. The observable linear dependence on the dimension $s$ is in accordance with the complexity $\calO(s N \ln N)$ of both algorithms. We assert that the computation times for both algorithms roughly differ by a  factor between $1.6$ and $2.45$. Furthermore, we note that the CBC algorithm, in particular the fast Fourier transformations, are based on compiled and optimized code via Python's Discrete Fourier Transform (numpy.fft) library. It is therefore noteworthy that the CBC-DBD algorithm, which is not based on any compiled libraries, is competitive nonetheless.    

\subsection{Error convergence behavior: CBC algorithm}

Here, we consider the convergence behavior of $e_{N,s,\alpha,\bsgamma^{\alpha}}(\bsz)$ for generating vectors $\bsz$ 
obtained by the CBC construction in Algorithm \ref{alg:cbc} and the common CBC algorithm as in, e.g., \cite{NC06b} 
for prime $N$ and different product weight sequences $\bsgamma = \{\gamma_j\}_{j \ge 1}$. Again, we aim to illustrate 
the almost optimal error convergence rate of $\calO(N^{-\alpha + \delta})$ in $E_{s,\bsgamma}^\alpha$ which can be
achieved according to Corollary \ref{cor:main-result-cbc}, but stress that the presented graphs are to be understood 
as an illustration of the pre-asymptotic error behavior for the considered ranges of $N$.

The graphs in Figure \ref{fig:cbc} show that in most cases the CBC algorithm with smoothness-independent quality function $V_{N,s,\bsgamma}$ yields lattice rules with worst-case error values that are almost identical to those corresponding to rules constructed by the CBC algorithm which uses the worst-case error $e_{N,s,\alpha,\bsgamma^{\alpha}}$ as quality measure. In addition, we observe that for higher smoothness $\alpha$ and pre-asymptotically slowly decaying weights $\gamma_j$ the standard CBC construction with the worst-case error as the quality function excels as it is directly tailored to the function space. As before, we see that the faster the pre-asymptotic decay of the weight sequences, the closer the error rate is to the optimal rate of $\calO(N^{-\alpha})$.

\begin{figure}[H]
	\centering
	\textbf{Error convergence in the space $E_{s,\bsgamma}^{\alpha}$ with $s=100, \alpha=2,3,4$.} \par\medskip 
	\hspace{-0.25cm}
	\centering
	\begin{subfigure}[b]{0.5\textwidth}
		\centering
		\begin{tikzpicture}
		
		\begin{axis}[%
		width=0.8\textwidth,
		height=0.8\textwidth,
		at={(0\textwidth,0\textwidth)},
		scale only axis,
		xmode=log,
		xmin=10,
		xmax=1000000,
		xminorticks=true,
		xlabel={Prime number of points $N$},
		xmajorgrids,
		ymode=log,
		ymin=1e-20,
		ymax=0.05,
		yminorticks=true,
		ylabel={Worst-case error $e_{N,s,\alpha,\mathbf{\gamma^{\alpha}}}(\bsz)$},
		ymajorgrids,
		axis background/.style={fill=white},
		legend style={at={(0.03,0.03)},anchor=south west,legend cell align=left,align=left,draw=white!15!black}
		]
		\addplot [color=mycolor1-fig,solid,line width=0.6pt,mark=o,mark options={solid},forget plot]
		table[row sep=crcr]{%
			61	0.00457643900069773\\
			127	0.00129024975277151\\
			251	0.000370526608114807\\
			509	0.000105877987992033\\
			1021	3.02304172140273e-05\\
			2039	9.09422202926339e-06\\
			4093	2.57164023457242e-06\\
			8191	7.20803591650862e-07\\
			16381	2.16478646837305e-07\\
			32749	5.87915871034376e-08\\
			65537	1.76403411668631e-08\\
			131071	4.93980545929888e-09\\
		};
		\addplot [color=mycolor2-fig,solid,line width=0.6pt,mark=triangle,mark options={solid},forget plot]
		table[row sep=crcr]{%
			61	0.00456232536139465\\
			127	0.00130476226288932\\
			251	0.000368981263023153\\
			509	0.000105705960946351\\
			1021	3.03337003297473e-05\\
			2039	8.99821748552568e-06\\
			4093	2.55498783039824e-06\\
			8191	7.15461120642863e-07\\
			16381	2.11306027619938e-07\\
			32749	5.98416675984414e-08\\
			65537	1.69106150795327e-08\\
			131071	4.7398178633678e-09\\
		};
		\addplot [color=mycolor3-fig,dashed,line width=0.8pt]
		table[row sep=crcr]{%
			61	0.00250927894876706\\
			127	0.00067605212450741\\
			251	0.000199911106791359\\
			509	5.64559373793378e-05\\
			1021	1.62574532824649e-05\\
			2039	4.71870571359881e-06\\
			4093	1.35706184679981e-06\\
			8191	3.92422425412986e-07\\
			16381	1.13613430828501e-07\\
			32749	3.29129171791428e-08\\
			65537	9.51771021278553e-09\\
			131071	2.75536221354429e-09\\
		};
		\addlegendentry{\footnotesize $\mathcal{O}(N^{-1.79})$};
		
		\addplot [color=mycolor4-fig,solid,line width=0.6pt,mark=o,mark options={solid},forget plot]
		table[row sep=crcr]{%
			61	3.60153891517686e-05\\
			127	5.23820215110285e-06\\
			251	6.52136302581965e-07\\
			509	9.3750555581003e-08\\
			1021	1.27915645867856e-08\\
			2039	1.99699854753412e-09\\
			4093	2.67670539584283e-10\\
			8191	3.65633429058261e-11\\
			16381	5.99848485456692e-12\\
			32749	7.08526504070968e-13\\
			65537	1.20098872223416e-13\\
			131071	1.65216607713508e-14\\
		};
		\addplot [color=mycolor5-fig,solid,line width=0.6pt,mark=triangle,mark options={solid},forget plot]
		table[row sep=crcr]{%
			61	3.55385962362886e-05\\
			127	5.24624710662842e-06\\
			251	6.45554536183099e-07\\
			509	9.35604020577442e-08\\
			1021	1.26259939267129e-08\\
			2039	1.9271096846979e-09\\
			4093	2.62248264953695e-10\\
			8191	3.54383459909993e-11\\
			16381	5.44174786322261e-12\\
			32749	7.20695796385969e-13\\
			65537	1.00917659884327e-13\\
			131071	1.45407483077433e-14\\
		};
		\addplot [color=darkgray,dashed,line width=0.8pt]
		table[row sep=crcr]{%
			61	1.95462279299587e-05\\
			127	2.4750420821148e-06\\
			251	3.62912103764839e-07\\
			509	4.94915529058199e-08\\
			1021	6.95996977522084e-09\\
			2039	9.91024478888501e-10\\
			4093	1.39081245165676e-10\\
			8191	1.96878784393541e-11\\
			16381	2.79223064550054e-12\\
			32749	3.96382688012283e-13\\
			65537	5.61130423319878e-14\\
			131071	7.95736881362615e-15\\
		};
		\addlegendentry{\footnotesize $\mathcal{O}(N^{-2.82})$};
		
		\addplot [color=mycolor6-fig,solid,line width=0.6pt,mark=o,mark options={solid},forget plot]
		table[row sep=crcr]{%
			61	4.26391149499209e-07\\
			127	3.40693871202977e-08\\
			251	1.72573267949556e-09\\
			509	1.27243926374671e-10\\
			1021	8.48349833066318e-12\\
			2039	6.76529864227778e-13\\
			4093	4.13286530839397e-14\\
			8191	2.79205692825797e-15\\
			16381	2.5481539542921e-16\\
			32749	1.26462262900138e-17\\
			65537	1.28408894654685e-18\\
			131071	8.34498031107218e-20\\
		};
		\addplot [color=mycolor7-fig,solid,line width=0.6pt,mark=triangle,mark options={solid},forget plot]
		table[row sep=crcr]{%
			61	4.21003649769558e-07\\
			127	3.21381546592023e-08\\
			251	1.70058942725749e-09\\
			509	1.20365703414581e-10\\
			1021	7.93817323589033e-12\\
			2039	6.64882524071393e-13\\
			4093	3.99122540438182e-14\\
			8191	2.58792239758862e-15\\
			16381	2.22686158081971e-16\\
			32749	1.37494984068789e-17\\
			65537	1.51663609650966e-18\\
			131071	1.4364696226651e-19\\
		};
		\addplot [color=black,dashed,line width=0.8pt]
		table[row sep=crcr]{%
			61	2.31552007373257e-07\\
			127	1.38459655881684e-08\\
			251	1.01117013059546e-09\\
			509	6.68964846312654e-11\\
			1021	4.61503206595749e-12\\
			2039	3.23823034172968e-13\\
			4093	2.22773328575489e-14\\
			8191	1.55068974869935e-15\\
			16381	1.08219894709299e-16\\
			32749	7.56222412378339e-18\\
			65537	5.26425449116358e-19\\
			131071	3.67329167390429e-20\\
		};
		\addlegendentry{\footnotesize $\mathcal{O}(N^{-3.84})$};
		
		\end{axis}
		\end{tikzpicture}
		\caption{Weight sequence $\bsgamma=\{\gamma_j\}_{j=1}^s$ with $\gamma_j = 1/j^2$.}    
	\end{subfigure}
	\begin{subfigure}[b]{0.5\textwidth}  
		\centering 
		\begin{tikzpicture}
		
		\begin{axis}[%
		width=0.8\textwidth,
		height=0.8\textwidth,
		at={(0\textwidth,0\textwidth)},
		scale only axis,
		xmode=log,
		xmin=10,
		xmax=1000000,
		xminorticks=true,
		xlabel={Prime number of points $N$},
		xmajorgrids,
		ymode=log,
		ymin=1e-21,
		ymax=0.2,
		yminorticks=true,
		ylabel={Worst-case error $e_{N,s,\alpha,\mathbf{\gamma^{\alpha}}}(\bsz)$},
		ymajorgrids,
		axis background/.style={fill=white},
		legend style={at={(0.03,0.03)},anchor=south west,legend cell align=left,align=left,draw=white!15!black}
		]
		\addplot [color=mycolor1-fig,solid,line width=0.6pt,mark=o,mark options={solid},forget plot]
		table[row sep=crcr]{%
			61	0.00149439958925635\\
			127	0.000378368510339832\\
			251	9.88210549484549e-05\\
			509	2.61153104710953e-05\\
			1021	6.71842516204534e-06\\
			2039	1.80208152247799e-06\\
			4093	4.65026038105247e-07\\
			8191	1.23263487418752e-07\\
			16381	3.26116438429265e-08\\
			32749	8.22047166020136e-09\\
			65537	2.20802294936648e-09\\
			131071	5.77826218270183e-10\\
		};
		\addplot [color=mycolor2-fig,solid,line width=0.6pt,mark=triangle,mark options={solid},forget plot]
		table[row sep=crcr]{%
			61	0.00149310488780952\\
			127	0.000378025163460423\\
			251	9.83549112464933e-05\\
			509	2.58990873041785e-05\\
			1021	6.71067140250752e-06\\
			2039	1.79785995384656e-06\\
			4093	4.56835574955769e-07\\
			8191	1.21082504472805e-07\\
			16381	3.24302067118139e-08\\
			32749	8.21458630886382e-09\\
			65537	2.15209298283157e-09\\
			131071	5.74073503755406e-10\\
		};
		\addplot [color=mycolor3-fig,dashed,line width=0.8pt]
		table[row sep=crcr]{%
			61	0.000821207688295237\\
			127	0.000199923949806106\\
			251	5.38054724710845e-05\\
			509	1.37803280205813e-05\\
			1021	3.60427133388422e-06\\
			2039	9.50753168643576e-07\\
			4093	2.48322347233933e-07\\
			8191	6.52416593241094e-08\\
			16381	1.71631180583029e-08\\
			32749	4.5180224698751e-09\\
			65537	1.18705386164181e-09\\
			131071	3.12255207691986e-10\\
		};
		\addlegendentry{\footnotesize $\mathcal{O}(N^{-1.93})$};
		
		\addplot [color=mycolor4-fig,solid,line width=0.6pt,mark=o,mark options={solid},forget plot]
		table[row sep=crcr]{%
			61	1.31890072035355e-05\\
			127	1.58987039924798e-06\\
			251	1.97469759374849e-07\\
			509	2.51857896906118e-08\\
			1021	3.15196828759722e-09\\
			2039	4.15047377900296e-10\\
			4093	5.07051994499849e-11\\
			8191	6.63884006229567e-12\\
			16381	8.7038610682597e-13\\
			32749	1.04049654058917e-13\\
			65537	1.38384071515551e-14\\
			131071	1.93898467864484e-15\\
		};
		\addplot [color=mycolor5-fig,solid,line width=0.6pt,mark=triangle,mark options={solid},forget plot]
		table[row sep=crcr]{%
			61	1.31872396062089e-05\\
			127	1.57671548407281e-06\\
			251	1.97127579040481e-07\\
			509	2.4978536870921e-08\\
			1021	3.1282237987031e-09\\
			2039	4.14743959377927e-10\\
			4093	5.01230117892207e-11\\
			8191	6.53691277690266e-12\\
			16381	8.69491734535517e-13\\
			32749	1.0426407694528e-13\\
			65537	1.35888101776409e-14\\
			131071	1.93107796178009e-15\\
		};
		\addplot [color=darkgray,dashed,line width=0.8pt]
		table[row sep=crcr]{%
			61	7.25298178341489e-06\\
			127	8.22835242525562e-07\\
			251	1.08942067421865e-07\\
			509	1.33633292830535e-08\\
			1021	1.69312911046295e-09\\
			2039	2.17346886493411e-10\\
			4093	2.74782651095061e-11\\
			8191	3.50566616768412e-12\\
			16381	4.4814394956998e-13\\
			32749	5.73452423199041e-14\\
			65537	7.31641662376563e-15\\
			131071	9.35182187459677e-16\\
		};
		\addlegendentry{\footnotesize $\mathcal{O}(N^{-2.97})$};
		
		\addplot [color=mycolor6-fig,solid,line width=0.6pt,mark=o,mark options={solid},forget plot]
		table[row sep=crcr]{%
			61	1.71570590641982e-07\\
			127	9.79213829503343e-09\\
			251	6.05280448986461e-10\\
			509	3.70799528264642e-11\\
			1021	2.30967568395001e-12\\
			2039	1.49255601155731e-13\\
			4093	8.91018874934075e-15\\
			8191	5.71980813892941e-16\\
			16381	3.72462577724537e-17\\
			32749	2.20558579009817e-18\\
			65537	1.41789736887006e-19\\
			131071	9.208954847244e-21\\
		};
		\addplot [color=mycolor7-fig,solid,line width=0.6pt,mark=triangle,mark options={solid},forget plot]
		table[row sep=crcr]{%
			61	1.71568744921195e-07\\
			127	9.6396413752285e-09\\
			251	6.05067610284948e-10\\
			509	3.69415975569439e-11\\
			1021	2.27775310873231e-12\\
			2039	1.49063910750197e-13\\
			4093	8.87259728715516e-15\\
			8191	5.6834666481556e-16\\
			16381	3.71908768881375e-17\\
			32749	2.20719917867453e-18\\
			65537	1.77401477932725e-19\\
			131071	1.38581868894411e-20\\
		};
		\addplot [color=black,dashed,line width=0.8pt]
		table[row sep=crcr]{%
			61	9.43628097066574e-08\\
			127	5.06144001917286e-09\\
			251	3.34136382293727e-10\\
			509	1.99065222740048e-11\\
			1021	1.2386862842126e-12\\
			2039	7.84463778251665e-14\\
			4093	4.8671635720664e-15\\
			8191	3.05690735825353e-16\\
			16381	1.92509565934653e-17\\
			32749	1.21395954827099e-18\\
			65537	7.6249526084196e-20\\
			131071	4.80110359596462e-21\\
		};
		\addlegendentry{\footnotesize $\mathcal{O}(N^{-3.99})$};
		
		\end{axis}
		\end{tikzpicture}
		\caption{Weight sequence $\bsgamma=\{\gamma_j\}_{j=1}^s$ with $\gamma_j = 1/j^3$.}
	\end{subfigure}
	\vskip\baselineskip
	\hspace{-0.25cm}
	\centering
	\begin{subfigure}[b]{0.5\textwidth}
		\centering
		\begin{tikzpicture}
		
		\begin{axis}[%
		width=0.8\textwidth,
		height=0.8\textwidth,
		at={(0\textwidth,0\textwidth)},
		scale only axis,
		xmode=log,
		xmin=10,
		xmax=1000000,
		xminorticks=true,
		xlabel={Prime number of points $N$},
		xmajorgrids,
		ymode=log,
		ymin=0.00005,
		ymax=75000000,
		yminorticks=true,
		ylabel={Worst-case error $e_{N,s,\alpha,\mathbf{\gamma^{\alpha}}}(\bsz)$},
		ymajorgrids,
		axis background/.style={fill=white},
		legend style={at={(0.03,0.03)},anchor=south west,legend cell align=left,align=left,draw=white!15!black}
		]
		\addplot [color=mycolor1-fig,solid,line width=0.6pt,mark=o,mark options={solid},forget plot]
		table[row sep=crcr]{%
			61	4634568.14275443\\
			127	2225929.04353636\\
			251	1126266.7925863\\
			509	555387.957645679\\
			1021	276870.710531586\\
			2039	138642.365500963\\
			4093	69066.8405598768\\
			8191	34512.1587166647\\
			16381	17257.3222761242\\
			32749	8631.8002590184\\
			65537	4313.34402485901\\
			131071	2156.68123701018\\
		};
		\addplot [color=mycolor2-fig,solid,line width=0.6pt,mark=triangle,mark options={solid},forget plot]
		table[row sep=crcr]{%
			61	4600375.30070987\\
			127	2206558.01945219\\
			251	1091381.68903471\\
			509	545896.541812871\\
			1021	270899.820992591\\
			2039	136768.502878728\\
			4093	67813.3822476945\\
			8191	33980.3294370871\\
			16381	17015.8490051789\\
			32749	8493.16117397436\\
			65537	4220.34563915898\\
			131071	2116.69623873709\\
		};
		\addplot [color=mycolor3-fig,dashed,line width=0.8pt]
		table[row sep=crcr]{%
			61	2990243.94546141\\
			127	1434772.08534694\\
			251	725262.176082432\\
			509	357287.067169935\\
			1021	177943.5924273\\
			2039	89015.6966256517\\
			4093	44301.1139425467\\
			8191	22115.3549132576\\
			16381	11047.5325332233\\
			32749	5520.55476308333\\
			65537	2755.93306728485\\
			131071	1376.64970795731\\
		};
		\addlegendentry{\footnotesize $\mathcal{O}(N^{-1.00})$};
		
		\addplot [color=mycolor4-fig,solid,line width=0.6pt,mark=o,mark options={solid},forget plot]
		table[row sep=crcr]{%
			61	413.93405419742\\
			127	195.206651009568\\
			251	98.7022505324968\\
			509	48.387879678918\\
			1021	24.0374855168552\\
			2039	11.952481492279\\
			4093	5.90796181382728\\
			8191	2.90585785224707\\
			16381	1.50940871695611\\
			32749	0.69616394911822\\
			65537	0.355576604424663\\
			131071	0.172803975278562\\
		};
		\addplot [color=mycolor5-fig,solid,line width=0.6pt,mark=triangle,mark options={solid},forget plot]
		table[row sep=crcr]{%
			61	387.68835749897\\
			127	188.33126796064\\
			251	88.6293935032307\\
			509	44.6947242621373\\
			1021	22.2215313716621\\
			2039	11.2567575958708\\
			4093	5.48235380859927\\
			8191	2.66950455643445\\
			16381	1.26982104997148\\
			32749	0.613355781627567\\
			65537	0.290867478976662\\
			131071	0.134467648845215\\
		};
		\addplot [color=darkgray,dashed,line width=0.8pt]
		table[row sep=crcr]{%
			61	251.99743237433\\
			127	118.754643017961\\
			251	59.033162667634\\
			509	28.580986996412\\
			1021	13.9932149141727\\
			2039	6.8821484833433\\
			4093	3.36697226739308\\
			8191	1.65241414613151\\
			16381	0.81151685104713\\
			32749	0.398681258057918\\
			65537	0.195664494537203\\
			131071	0.0960889973865253\\
		};
		\addlegendentry{\footnotesize $\mathcal{O}(N^{-1.03})$};
		
		\addplot [color=mycolor6-fig,solid,line width=0.6pt,mark=o,mark options={solid},forget plot]
		table[row sep=crcr]{%
			61	16.4817904121816\\
			127	7.1339213237205\\
			251	3.53048216135207\\
			509	1.65587799355142\\
			1021	0.862069575307995\\
			2039	0.369188858912188\\
			4093	0.176484128087849\\
			8191	0.0761284846863732\\
			16381	0.0513233642958822\\
			32749	0.0140345767558548\\
			65537	0.00883360578843512\\
			131071	0.00368749136722201\\
		};
		\addplot [color=mycolor7-fig,solid,line width=0.6pt,mark=triangle,mark options={solid},forget plot]
		table[row sep=crcr]{%
			61	13.6349958910252\\
			127	6.55991367620095\\
			251	2.91143983397834\\
			509	1.34202433843196\\
			1021	0.616963198118981\\
			2039	0.264495412209109\\
			4093	0.11782156357491\\
			8191	0.0454267872304141\\
			16381	0.0205165247023423\\
			32749	0.00807881232233529\\
			65537	0.00339011176106638\\
			131071	0.00126953177703508\\
		};
		\addplot [color=black,dashed,line width=0.8pt]
		table[row sep=crcr]{%
			61	8.86274732916638\\
			127	3.72444749060631\\
			251	1.6644782754234\\
			509	0.721579519135469\\
			1021	0.316876026019911\\
			2039	0.139881495259051\\
			4093	0.0613749580281661\\
			8191	0.0270267435668881\\
			16381	0.0119108041643536\\
			32749	0.00525122800951793\\
			65537	0.00231244424734289\\
			131071	0.00101905832951069\\
		};
		\addlegendentry{\footnotesize $\mathcal{O}(N^{-1.18})$};
		
		\end{axis}
		\end{tikzpicture}
		\caption{Weight sequence $\bsgamma=\{\gamma_j\}_{j=1}^s$ with $\gamma_j = (0.95)^j$.}
	\end{subfigure}
	\begin{subfigure}[b]{0.5\textwidth}  
		\centering 
		\begin{tikzpicture}
		
		\begin{axis}[%
		width=0.8\textwidth,
		height=0.8\textwidth,
		at={(0\textwidth,0\textwidth)},
		scale only axis,
		xmode=log,
		xmin=10,
		xmax=1000000,
		xminorticks=true,
		xlabel={Prime number of points $N$},
		xmajorgrids,
		ymode=log,
		ymin=1e-17,
		ymax=1,
		yminorticks=true,
		ylabel={Worst-case error $e_{N,s,\alpha,\mathbf{\gamma^{\alpha}}}(\bsz)$},
		ymajorgrids,
		axis background/.style={fill=white},
		legend style={at={(0.03,0.03)},anchor=south west,legend cell align=left,align=left,draw=white!15!black}
		]
		\addplot [color=mycolor1-fig,solid,line width=0.6pt,mark=o,mark options={solid},forget plot]
		table[row sep=crcr]{%
			61	0.0465165526248944\\
			127	0.0148191830011895\\
			251	0.0058727305305099\\
			509	0.00192379711626869\\
			1021	0.000705641046409251\\
			2039	0.000267619872188979\\
			4093	0.000109693338426932\\
			8191	3.00231121813236e-05\\
			16381	1.07704271885695e-05\\
			32749	3.75700338331686e-06\\
			65537	1.24462279387403e-06\\
			131071	4.51703511089662e-07\\
		};
		\addplot [color=mycolor2-fig,solid,line width=0.6pt,mark=triangle,mark options={solid},forget plot]
		table[row sep=crcr]{%
			61	0.039824759913778\\
			127	0.0139177820120805\\
			251	0.00515106196298379\\
			509	0.0018047210006839\\
			1021	0.000634771498002138\\
			2039	0.000229431305471914\\
			4093	8.12095106611222e-05\\
			8191	2.71792226035835e-05\\
			16381	9.21398545032347e-06\\
			32749	3.22000725688966e-06\\
			65537	1.07027323165465e-06\\
			131071	3.56092132125456e-07\\
		};
		\addplot [color=mycolor3-fig,dashed,line width=0.8pt]
		table[row sep=crcr]{%
			61	0.0219036179525779\\
			127	0.00729621097254534\\
			251	0.00262762865131929\\
			509	0.000910499604200895\\
			1021	0.000320697325383512\\
			2039	0.0001137061275822\\
			4093	4.00059785537613e-05\\
			8191	1.41403028921087e-05\\
			16381	5.00299195463577e-06\\
			32749	1.77100399128931e-06\\
			65537	6.25984167487752e-07\\
			131071	2.2146729426821e-07\\
		};
		\addlegendentry{\footnotesize $\mathcal{O}(N^{-1.50})$};
		
		\addplot [color=mycolor4-fig,solid,line width=0.6pt,mark=o,mark options={solid},forget plot]
		table[row sep=crcr]{%
			61	0.000812666738807578\\
			127	9.81940889391407e-05\\
			251	3.28570588552176e-05\\
			509	3.81240298480626e-06\\
			1021	8.74476856808452e-07\\
			2039	2.22674510971568e-07\\
			4093	1.17064245837847e-07\\
			8191	7.2793315572828e-09\\
			16381	1.68737127146558e-09\\
			32749	3.55402606558459e-10\\
			65537	5.93305628202554e-11\\
			131071	1.42057789458795e-11\\
		};
		\addplot [color=mycolor5-fig,solid,line width=0.6pt,mark=triangle,mark options={solid},forget plot]
		table[row sep=crcr]{%
			61	0.000410357685758694\\
			127	7.67517402478254e-05\\
			251	1.59728964657848e-05\\
			509	3.11322261199288e-06\\
			1021	6.15555840136454e-07\\
			2039	1.23165153121391e-07\\
			4093	2.66789829034176e-08\\
			8191	4.72184600915742e-09\\
			16381	8.59553248673856e-10\\
			32749	1.77747729934381e-10\\
			65537	3.1619234247526e-11\\
			131071	5.83540914585413e-12\\
		};
		\addplot [color=darkgray,dashed,line width=0.8pt]
		table[row sep=crcr]{%
			61	0.000225696727167281\\
			127	4.08468402241577e-05\\
			251	8.34606234217235e-06\\
			509	1.60604570909169e-06\\
			1021	3.17010824776453e-07\\
			2039	6.32204544647337e-08\\
			4093	1.24576435615915e-08\\
			8191	2.47236673222663e-09\\
			16381	4.9143924029156e-10\\
			32749	9.77612514639098e-11\\
			65537	1.94025711236727e-11\\
			131071	3.85636023901073e-12\\
		};
		\addlegendentry{\footnotesize $\mathcal{O}(N^{-2.33})$};
		
		\addplot [color=mycolor6-fig,solid,line width=0.6pt,mark=o,mark options={solid},forget plot]
		table[row sep=crcr]{%
			61	3.08980585275085e-05\\
			127	1.15177025020475e-06\\
			251	4.2382745891417e-07\\
			509	1.26354325553213e-08\\
			1021	1.94413977073814e-09\\
			2039	3.17094430414221e-10\\
			4093	3.11128899856657e-10\\
			8191	3.4456527232764e-12\\
			16381	5.17768289303527e-13\\
			32749	6.81925404500222e-14\\
			65537	5.21276702013556e-15\\
			131071	8.10177946909353e-16\\
		};
		\addplot [color=mycolor7-fig,solid,line width=0.6pt,mark=triangle,mark options={solid},forget plot]
		table[row sep=crcr]{%
			61	7.04992016155101e-06\\
			127	7.52515546971171e-07\\
			251	7.99634213202026e-08\\
			509	8.96344021457356e-09\\
			1021	9.78759682617165e-10\\
			2039	1.09672246247277e-10\\
			4093	1.1156541804078e-11\\
			8191	1.25568756763142e-12\\
			16381	1.35453075691307e-13\\
			32749	1.70538909393908e-14\\
			65537	1.53537268291202e-15\\
			131071	1.64254060084748e-16\\
		};
		\addplot [color=black,dashed,line width=0.8pt]
		table[row sep=crcr]{%
			61	3.87745608885306e-06\\
			127	3.83123417756267e-07\\
			251	4.46144758627194e-08\\
			509	4.79009236152865e-09\\
			1021	5.3230515306048e-10\\
			2039	5.99827649971914e-11\\
			4093	6.65034383613962e-12\\
			8191	7.44487729592778e-13\\
			16381	8.35202296744406e-14\\
			32749	9.37964001666494e-15\\
			65537	1.05007582192922e-15\\
			131071	1.17788391396076e-16\\
		};
		\addlegendentry{\footnotesize $\mathcal{O}(N^{-3.16})$};
		
		\end{axis}
		\end{tikzpicture}
		\caption{Weight sequence $\bsgamma=\{\gamma_j\}_{j=1}^s$ with $\gamma_j = (0.7)^j$.}
	\end{subfigure}
	\vskip\baselineskip
	\begin{tikzpicture}
	\hspace{0.05\linewidth}
	\begin{customlegend}[
	legend columns=5,legend style={align=left,draw=none,column sep=1.5ex},
	legend entries={Korobov's CBC, standard fast CBC \quad, $\alpha=2$, $\alpha=3$, $\alpha=4$}
	]
	\addlegendimage{color=gray, mark=o,solid,line width=1.0pt,line legend}
	\addlegendimage{color=gray, mark=triangle,solid,line width=1.0pt}  
	\addlegendimage{color=mycolor-alpha1, only marks,mark=square*, solid, line width=5pt}
	\addlegendimage{color=mycolor-alpha2, only marks,mark=square*, solid, line width=5pt}
	\addlegendimage{color=mycolor-alpha3, only marks,mark=square*, solid, line width=5pt}
	\end{customlegend}
	\end{tikzpicture}
	\caption{Convergence results of the worst-case error $e_{N,s,\alpha,\bsgamma^{\alpha}}(\bsz)$ in the weighted space $E_{s,\bsgamma}^\alpha$ for smoothness parameters $\alpha=2,3,4$ and dimension $s=100$. The generating vectors $\bsz$ are constructed by component-by-component algorithms for prime $N$ using $V_{N,s,\bsgamma}$ 
	and $e_{N,s,\alpha,\bsgamma^{\alpha}}$, respectively, as quality functions.}
	\label{fig:cbc}
\end{figure} 

\begin{remark}
	As indicated in Section \ref{sec:impl-cbc}, Algorithm \ref{alg:cbc} can be implemented in a fast manner analogously to the standard component-by-component algorithm with the worst-case error $e_{N,s,\alpha,\bsgamma^{\alpha}}$ as the quality function. Our numerical tests confirm this observation. For the sake of brevity, we omit listing the timing results here. Indeed, there is no relevant difference 
	between the run-time of the fast implementations of the standard component-by-component algorithm and that of Algorithm \ref{alg:cbc}.
\end{remark}

\section{Conclusion} \label{sec:conc}

In this paper, we studied variants of component-by-component algorithms for the construction of lattice rules 
for numerical integration in weighted function spaces. We analyzed a component-by-component digit-by-digit (CBC-DBD) 
construction and a particular version of the CBC construction. Both algorithms considered are motivated by earlier 
work of Korobov, and both constructions are independent of the value of the smoothness parameter $\alpha$, which may 
be an advantage over the standard CBC constructions in the literature. We showed above that the lattice rules 
constructed by these algorithms yield an error convergence rate that can be arbitrarily close to the optimal 
convergence rate, and the error bounds can be made independent of the dimension if the coordinate weights 
satisfy suitable summability conditions. Moreover, we studied fast implementations of the new algorithms
and showed that the computational effort is of the same order of magnitude as that of the standard CBC construction. 
The lattice rules constructed proved to be competitive to lattice rules constructed by other common methods with 
respect to their integration error. The full generalization of the results on the CBC-DBD algorithm 
to arbitrary weights remains open for future research.

\section*{Acknowledgements}
P.~Kritzer and O.~Osisiogu gratefully acknowledge the support of the Austrian Science Fund (FWF): Project F5506, 
which is part of the Special Research Program ``Quasi-Monte Carlo Methods: Theory and Applications''.
D.~Nuyens likes to thank R.~Matthysen who studied the Korobov article \cite{Kor82,Kor82Eng} in his Master thesis \cite{Mat}.

\section*{Appendix}

\subsection*{The proof of Theorem \ref{thm:wce}}

\begin{proof}
	Recall that the worst-case error for the space $E_{s,\bsgamma}^{\alpha}$ is defined as
	\begin{equation*}
		e_{N,s,\alpha,\bsgamma}(\bsz)
		:=
		\sup_{\substack{f \in E_{s,\bsgamma}^{\alpha} \\ \|f\|_{E_{s,\bsgamma}^{\alpha}} \le 1}} | I(f) - Q_N(f,\bsz)|
		.
	\end{equation*}
	The combination of~\eqref{eq:Holder-infty} and the definition of $\|f\|_{E_{s,\bsgamma}^{\alpha}}$ then leads to the estimate
	\begin{equation*}
		e_{N,s,\alpha,\bsgamma}(\bsz)
		\le
		\sup_{\substack{f \in E_{s,\bsgamma}^{\alpha} \\ \|f\|_{E_{s,\bsgamma}^{\alpha}} \le 1}} \|f\|_{E_{s,\bsgamma}^{\alpha}} 
		\sum_{\substack{\bszero \ne \bsm \in \Z^s \\ \bsm \cdot \bsz \equiv 0 \tpmod{N}}} r_{\alpha,\bsgamma}^{-1}(\bsm)
		\le
		\sum_{\substack{\bszero \ne \bsm \in \Z^s \\ \bsm \cdot \bsz \equiv 0 \tpmod{N}}} r_{\alpha,\bsgamma}^{-1}(\bsm) .
	\end{equation*}
	Consider now the function $g$ which has Fourier coefficients $\hat{g}(\bsm) = r_{\alpha,\bsgamma}^{-1}(\bsm)$, then $\|g\|_{E_{s,\bsgamma}^{\alpha}} = 1$ and 
	\begin{equation*}
		Q_N(g,\bsz) - I(g) 
		=
		\sum_{\substack{\bszero \ne \bsm \in \Z^s \\ \bsm \cdot \bsz \equiv 0 \tpmod{N}}} r_{\alpha,\bsgamma}^{-1}(\bsm).
	\end{equation*}
	Hence, the proven upper bound is attained by $g$ such that the claimed identity follows. 
\end{proof}

\subsection*{The proof of Theorem \ref{thm:existence-T}}

\begin{proof}
	For $\bszero \ne \bsm \in M_{N,s}$ there is at least one $m_j$ for which $\gcd(m_j, N) = 1$ (since $N$ is prime) 
	and therefore the multiplicative inverse $m_j^{-1} \pmod{N}$ exists. Without loss of generality, assume that this $j$ equals $s$, then
	\begin{equation*}
		\sum_{z_1,\ldots,z_s=1}^{N-1} \delta_N(\bsm \cdot \bsz)
		\le
		\sum_{z_1,\ldots,z_{s-1}=1}^{N-1} \underbrace{\sum_{z_s=0}^{N-1} \delta_N(\bsm \cdot \bsz)}_{= 1}
		=
		(N-1)^{s-1}
		,
	\end{equation*}
	since $(m_1 z_1 + \cdots + m_{s-1} z_{s-1}) + m_s z_s \equiv 0 \pmod{N}$ has exactly one solution 
	$z_s \in \bbZ_N$, namely $z_s \equiv - m_s^{-1} (m_1 z_1 + \cdots + m_{s-1} z_{s-1}) \pmod{N}$. See also \cite[Proof of Proposition 20]{Kor63}.
	By the standard averaging argument there exists a $\bsz^\ast \in \{1,\ldots,N-1\}^s$ which satisfies
	\begin{align} \label{eq:average}
		T(N, \bsz^\ast)
		=
		\min_{\bsz \in \{1,\ldots,N-1\}^s} T(N,\bsz)
		&\le
		\frac{1}{(N-1)^s} \sum_{z_1,\ldots,z_s=1}^{N-1} T(N,\bsz) \nonumber \\
		&=
		\frac{1}{(N-1)^s} \sum_{\bszero \ne \bsm \in M_{N,s}} \frac{1}{r_{1,\bsgamma}(\bsm)} 
		\sum_{z_1,\ldots,z_s=1}^{N-1} \delta_N(\bsm \cdot \bsz) \nonumber \\
		&\le
		\frac{1}{N-1} \sum_{\bszero \ne \bsm \in M_{N,s}} \frac{1}{r_{1,\bsgamma}(\bsm)}
		\le
		\frac{2}{N} \sum_{\bszero \ne \bsm \in M_{N,s}} \frac{1}{r_{1,\bsgamma}(\bsm)} .
	\end{align}
	Note that we can write
	\begin{align*}
		\sum_{\bszero \ne \bsm \in M_{N,s}} \frac{1}{r_{1,\bsgamma}(\bsm)}
		&=
		\sum_{\emptyset\neq\setu\subseteq \{1 \mcol s\}} \sum_{\bsm_\setu \in M_{N,|\setu|}^\ast} \frac{1}{r_{1,\bsgamma_\setu}(\bsm_\setu)}
		=
		\sum_{\emptyset\neq\setu\subseteq \{1 \mcol s\}} \gamma_\setu \sum_{\bsm_\setu \in M_{N,|\setu|}^\ast} \prod_{j\in\setu} \frac{1}{\abs{m_j}} \\	
		&=
		\sum_{\emptyset\neq\setu\subseteq \{1 \mcol s\}} \gamma_\setu \prod_{j\in\setu} \sum_{m_j \in M_{N,1}^\ast} \frac{1}{\abs{m_j}}
		= 
		\sum_{\emptyset\neq\setu\subseteq \{1 \mcol s\}} \gamma_\setu \prod_{j\in\setu} \left( 2\sum_{m_j=1}^{N-1}\frac{1}{m_j} \right) . 
	\end{align*}
	Observing that the inequality
	\begin{equation} \label{eq:estimate_ln}
		\sum_{m=1}^{N-1} \frac{1}{m} - 1 = \sum_{m=1}^{N-2} \frac{1}{m+1}
		<
		\sum_{m=1}^{N-2} \int_{m}^{m+1} \frac1x \rd x 
		= 
		\int_{1}^{N-1} \frac1x \rd x 
		= 
		\ln(N-1) < \ln N
	\end{equation} 
	holds, we obtain the estimate
	\begin{equation*}
		\sum_{\bszero \ne \bsm \in M_{N,s}} \frac{1}{r_{1,\bsgamma}(\bsm)}
		\le 
		\sum_{\emptyset\neq\setu\subseteq \{1 \mcol s\}} \gamma_\setu \prod_{j\in\setu} (2(1 + \ln N))
		=
		\sum_{\emptyset\neq\setu\subseteq \{1 \mcol s\}} \gamma_\setu  (2(1 + \ln N))^{\abs{\setu}}.
	\end{equation*}
	Combining this with \eqref{eq:average} yields the existence of a good generating vector $\bsz \in \Z^s$ as claimed.
\end{proof}

\subsection*{The proof of Lemma \ref{lemma:expr_target_function}}

\begin{proof}
	Firstly, we note that by Lemma \ref{lem:ln-sin} with $a=1$ the function $\ln (\sin^{-2}(\pi x) )$ has the expansion
	\begin{equation*}
		\ln (\sin^{-2}(\pi x) )
		=
		-2\ln (\sin(\pi x) )
		=
		\ln 4 + \sum_{m \in \Z_\ast} \frac{\rme^{2\pi \icomp m x}}{\abs{m}}
	\end{equation*}
	for $x \in (0,1)$, and therefore
	\begin{align} \label{eq:expr_target_function}
		\ln (\sin^{-2}(\pi x) )
		&=
		\ln 4 + \sum_{m\in M_{N,1}^\ast} \frac{\rme^{2\pi \icomp m x}}{\abs{m}} + r(x)
	\end{align}
	with remainder $r(x)$ bounded as follows,
	\begin{align*}
		|r (x)|
		&=
		\left| \sum_{m=N}^{\infty} \frac{\rme^{2\pi \icomp m x}}{m} + \sum_{m=N}^{\infty} \frac{\rme^{-2\pi \icomp m x}}{m} \right|
		\le
		2 \left| \sum_{m=N}^{\infty} \frac{\rme^{2\pi \icomp m x}}{m} \right| .
	\end{align*}
	If we make use of the identity
	\begin{align*}
		\frac{\rme^{2\pi \icomp m x}}{m}
		&=
		\frac{1}{\rme^{2\pi \icomp x} - 1}
		\left( \frac{\rme^{2\pi \icomp (m+1) x}}{m+1} - \frac{\rme^{2\pi \icomp m x}}{m} + \frac{\rme^{2\pi \icomp (m+1) x}}{m (m+1)} \right) ,
	\end{align*}
	summing over the $m \ge N$ yields a telescoping sum such that
	\begin{align*}
		\left| \sum_{m=N}^{\infty} \frac{\rme^{2\pi \icomp m x}}{m} \right|
		&=
		\frac{1}{|\rme^{2\pi \icomp x} - 1|} \left|  \sum_{m=N}^{\infty}\left(
		\frac{\rme^{2\pi \icomp (m+1) x}}{m+1} - \frac{\rme^{2\pi \icomp m x}}{m} + \frac{\rme^{2\pi \icomp (m+1) x}}{m (m+1)}\right) \right| \\
		&= 
		\frac{1}{|\rme^{2\pi \icomp x} - 1|} \left| \frac{-\rme^{2 \pi \icomp N x}}{N} + \sum_{m=N}^{\infty} \frac{\rme^{2\pi \icomp (m+1) x}}{m (m+1)} \right| \\
		&\le 
		\frac{1}{2 \sin(\pi x)}\left(\frac1N + \sum_{m=N}^{\infty} \frac{1}{m (m+1)} \right)
		=
		\frac{1}{N \sin(\pi x)},
	\end{align*}
	where we used that $\rme^{2\pi \icomp m x}$ is bounded, 
	$\sum_{m=N}^{\infty} \frac{1}{m(m+1)}= \sum_{m=N}^{\infty} ( \frac1m - \frac{1}{m+1} ) = \frac1N$, and that
	\begin{align*}
		|\rme^{2\pi \icomp x} - 1|
		&=
		|(\cos(\pi x) + \icomp \sin(\pi x))^2 - 1|
		=
		| \cos^2(\pi x) - \sin^2(\pi x) + 2 \icomp \cos( \pi x) \, \sin(\pi x) - 1| \\
		&=
		|2 \sin(\pi x) \, (\icomp \cos(\pi x) - \sin( \pi x))|
		= 
		2 \sin(\pi x) \sqrt{\cos^2(\pi x) + \sin^2(\pi x)}
		=
		2 \sin(\pi x) .
	\end{align*}
	Since for $x \in [0,\tfrac12]$ we have that $\sin(\pi x) \ge 2 x = 2 \|x\|$, the symmetry of $\sin(\pi x)$ and $\|x\|$ around $\tfrac12$
	implies that $\sin(\pi x) \ge 2 \|x\|$ for all $x \in [0,1]$. This then yields that
	\begin{align*}
		|r(x)| 
		&\le
		\frac{2}{N \sin(\pi x)} 
		\le \frac{1}{N \|x\|}
		\quad \text{and thus} \quad
		r(x) = \frac{\tau(x)}{N \|x\|}
	\end{align*}
	for some $\tau(x) \in \RR$ with $|\tau(x)| \le 1$. This together with the expression in \eqref{eq:expr_target_function} yields the claim.
\end{proof}

\bigskip

\bigskip

\begin{small}
	\noindent\textbf{Authors' addresses:}\\
	
	\noindent Adrian Ebert\\
	Department of Computer Science\\
	KU Leuven\\
	Celestijnenlaan 200A, 3001 Leuven, Belgium.\\
	\texttt{adrian.ebert@cs.kuleuven.be}
	
	\medskip
	
	\noindent Peter Kritzer\\
	Johann Radon Institute for Computational and Applied Mathematics (RICAM)\\
	Austrian Academy of Sciences\\
	Altenbergerstr. 69, 4040 Linz, Austria.\\
	\texttt{peter.kritzer@oeaw.ac.at}
	
	\medskip
	
	\noindent Dirk Nuyens\\
	Department of Computer Science\\
	KU Leuven\\
	Celestijnenlaan 200A, 3001 Leuven, Belgium.\\
	\texttt{dirk.nuyens@cs.kuleuven.be}
	
	\medskip
	
	\noindent Onyekachi Osisiogu\\
	Johann Radon Institute for Computational and Applied Mathematics (RICAM)\\
	Austrian Academy of Sciences\\
	Altenbergerstr. 69, 4040 Linz, Austria.\\
	\texttt{onyekachi.osisiogu@ricam.oeaw.ac.at}
	
\end{small}

\end{document}